\documentclass[10pt,a4paper]{article}

\usepackage[a4paper,hmargin=40mm,vmargin=48mm,footskip=15mm]{geometry}
\usepackage[OT1]{fontenc}
\usepackage[utf8]{inputenc}
\usepackage[english]{babel}
\usepackage{csquotes}
\usepackage{lmodern}
\usepackage{amsmath}
\usepackage{amsfonts}
\usepackage{amssymb}
\usepackage{mathtools}
\usepackage{amsthm}
\usepackage{thmtools}
\usepackage{bbm}
\allowdisplaybreaks
\usepackage{float}
\usepackage{caption}
\usepackage{subcaption}
\usepackage{color}
\usepackage[table]{xcolor}
\usepackage{booktabs}
\usepackage{tabularx}
\usepackage{tikz}
\usepackage{enumitem}
\usepackage[backend=biber,style=numeric,sorting=nyt,sortcites=true,doi=false,isbn=false,url=false,eprint=true,maxbibnames=9,giveninits=true]{biblatex}
\DeclareNameAlias{sortname}{last-first}
\renewbibmacro{in:}{}
\AtEveryBibitem{%
  \ifentrytype{misc}{%
  }{%
    \clearfield{archivePrefix}%
    \clearfield{arxivId}%
    \clearfield{eprint}%
    \clearfield{primaryClass}%
  }%
}
\usepackage{hyperref}
\hypersetup{
colorlinks=false,
bookmarksnumbered=false,
pdfpagemode=UseNone,
pdfstartpage=1,
pdfstartview={XYZ null null 1.00},
pdfpagelayout=OneColumn,
pdflang=English}
\usepackage{authblk}

\newtheoremstyle{myplain}{12pt plus 3pt minus 3pt}{12pt plus 3pt minus 3pt}{\itshape}{}{\bfseries}{.}{ }{}
\newtheoremstyle{mydefinition}{12pt plus 3pt minus 3pt}{12pt plus 3pt minus 3pt}{}{}{\bfseries}{.}{ }{}
\theoremstyle{myplain}
\newtheorem{theorem}{Theorem}

\newtheorem{lemma}{Lemma}
\newtheorem{corollary}{Corollary}
\theoremstyle{mydefinition}
\newtheorem{definition}{Definition}
\newtheorem{assumption}{Assumption}

\newtheorem{remarkx}{Remark}
\newtheorem{examplex}{Example}
{\pushQED{\qed}\examplex}
{\popQED\endexamplex\vspace{2ex}}

\setlist{align=right,labelindent=0.2em,labelwidth=1.5em,labelsep*=0.6em,leftmargin=!,topsep=0ex,partopsep=0ex,parsep=0ex,itemsep=0ex}
\setlist[itemize,1]{label=\raisebox{0.0em}{\scalebox{1.0}{$\bullet$}}}
\setlist[itemize,2]{label=\raisebox{0.1em}{\scalebox{0.5}{$\blacksquare$}}}
\setlist[itemize,3]{label=\raisebox{0.1em}{\scalebox{0.7}{$\blacktriangleright$}}}
\setlist[itemize,4]{label=\raisebox{0.1em}{\scalebox{0.6}{$\blacklozenge$}}}
\setlist[enumerate]{label=(\alph*)}

\renewcommand{\P}{\ensuremath{\mathbb{P}}}
\newcommand{\E}{\ensuremath{\mathbb{E}}}
\newcommand{\R}{\ensuremath{\mathbb{R}}}

\newcommand{\e}{\ensuremath{\mathrm{e}}}
\renewcommand{\epsilon}{\varepsilon}

\newcommand{\bigO}{\ensuremath{\mathcal{O}}}
\newcommand{\bigOp}[1][\P]{\ensuremath{\bigO_{\scalebox{0.5}{$#1$}}}}
\newcommand{\smallO}{\ensuremath{o}}
\newcommand{\smallOp}[1][\P]{\ensuremath{\smallO_{\scalebox{0.5}{$#1$}}}}
\newcommand{\bigTheta}{\ensuremath{\Theta}}
\newcommand{\bigThetap}[1][\P]{\ensuremath{\bigTheta_{\scalebox{0.5}{$#1$}}}}
\newcommand{\bigOmega}{\ensuremath{\Omega}}
\newcommand{\bigOmegap}[1][\P]{\ensuremath{\bigOmega_{\scalebox{0.5}{$#1$}}}}
\newcommand{\1}[1]{\ensuremath{\mathbbm{1}{\raisebox{-0.25ex}{\hspace{-0.05em}\scriptsize{}$#1$}}}}

\DeclareMathOperator*{\argmax}{arg\,max}

\newcommand{\medvee}{\mathbin{\raisebox{-0.25ex}{\scalebox{1.6}{$\vee$}}}}
\newcommand{\medwedge}{\mathbin{\raisebox{-0.25ex}{\scalebox{1.6}{$\wedge$}}}}
\newcommand{\invf}{{\mathchoice{\scalebox{0.7}[1.0]{$\displaystyle{}-$}1}{\scalebox{0.7}[1.0]{$\textstyle{}-$}1}{\scalebox{0.7}[1.0]{$\scriptstyle{}-$}1}{\scalebox{0.7}[1.0]{$\scriptscriptstyle{}-$}1}}}

\newcommand{\test}{\psi}
\newcommand{\testlr}{\test_n^\textup{\scriptsize\textsc{lr}}}
\newcommand{\risk}{R}
\newcommand{\w}{\theta}
\newcommand{\wmax}{\w_{\textup{max}}}
\newcommand{\wmin}{\w_{\textup{min}}}
\newcommand{\pmax}{p_{\textup{max}}}
\newcommand{\pmin}{p_{\textup{min}}}
\newcommand{\scaling}[1][]{\rho_{#1}}

\newcommand{\stat}[1][]{T_{#1}}
\newcommand{\statkp}[1][]{\stat[{#1}]^{\textup{\textsf{\scriptsize{}k}}}}
\newcommand{\statup}[1][]{\stat[{#1}]^{\textup{\textsf{\scriptsize{}u}}}}
\newcommand{\est}[1][]{\mathrlap{\mathchoice{\phantom{\scalebox{0.08}{$\displaystyle{}e(C)$}}\raisebox{0.3ex}{$\widehat{\phantom{\scalebox{0.70}{$\displaystyle{}e(C)$}}}$}}{\phantom{\scalebox{0.08}{$\textstyle{}e(C)$}}\raisebox{0.3ex}{$\widehat{\phantom{\scalebox{0.70}{$\textstyle{}e(C)$}}}$}}{\phantom{\scalebox{0.08}{$\scriptstyle{}e(C)$}}\raisebox{-0.1ex}{$\widehat{\phantom{\scalebox{0.70}{$\scriptstyle{}e(C)$}}}$}}{\phantom{\scalebox{0.08}{$\scriptscriptstyle{}e(C)$}}\raisebox{-0.4ex}{$\widehat{\phantom{\scalebox{0.72}{$\scriptscriptstyle{}e(C)$}}}$}}}\mathrlap{\smash{e(#1)}}\phantom{e(#1)}}
\newcommand{\esttr}[1][]{\ensuremath{\est[#1]\mathchoice{\hspace{-0.24em}\raisebox{1ex}{\scalebox{0.8}{$\vee$}}}{\hspace{-0.24em}\raisebox{1ex}{\scalebox{0.8}{$\vee$}}}{\hspace{-0.18em}\raisebox{0.65ex}{\scalebox{0.6}{$\vee$}}}{\hspace{-0.16em}\raisebox{0.55ex}{\scalebox{0.5}{$\vee$}}}}}
\newcommand{\optimalsubgraph}[1][]{D^{\star}\def\temp{#1}\ifx\temp\empty{}\else\!(#1)\fi}

\newcommand{\uniformboundset}{\ensuremath{\mathcal{D}}}
\newcommand{\truncation}[1][]{\ensuremath{\Gamma_{#1}}}
\newcommand{\truncationset}[1][]{\ensuremath{\mathcal{E}_{#1}}}
\newcommand{\cutoff}[1][]{\ensuremath{\zeta_{#1}}}

\title{Detecting a planted community in an inhomogeneous random graph}
\author[{}\hspace{0.5pt}\protect\hyperlink{hyp:affil1}{a},\protect\hyperlink{hyp:email1}{1}]{\protect\hypertarget{hyp:author1}{Kay Bogerd}}
\author[{}\hspace{0.5pt}\protect\hyperlink{hyp:affil1}{a},\protect\hyperlink{hyp:email2}{2}]{\protect\hypertarget{hyp:author2}{Rui M. Castro}}
\author[{}\hspace{0.5pt}\protect\hyperlink{hyp:affil1}{a},\protect\hyperlink{hyp:email3}{3}]{\protect\hypertarget{hyp:author3}{Remco van der Hofstad}}
\author[{}\hspace{0.5pt}\protect\hyperlink{hyp:affil2}{b},\protect\hyperlink{hyp:email4}{4}]{\protect\hypertarget{hyp:author4}{Nicolas Verzelen}}
\affil[ ]{\small\textsuperscript{\protect\hypertarget{hyp:affil1}{a}}\hspace{0.5pt}Eindhoven University of Technology, \textsuperscript{\protect\hypertarget{hyp:affil2}{b}}\hspace{0.5pt}INRA}
\affil[ ]{\small{}
\parbox{400pt}{
\textsuperscript{\protect\hypertarget{hyp:email1}{1}}\hspace{0.5pt}\texttt{\footnotesize\href{mailto:k.m.bogerd@tue.nl}{k.m.bogerd@tue.nl}},
\textsuperscript{\protect\hypertarget{hyp:email2}{2}}\hspace{0.5pt}\texttt{\footnotesize\href{mailto:rmcastro@tue.nl}{rmcastro@tue.nl}},
\textsuperscript{\protect\hypertarget{hyp:email3}{3}}\hspace{0.5pt}\texttt{\footnotesize\href{mailto:r.w.v.d.hofstad@tue.nl}{r.w.v.d.hofstad@tue.nl}},
\textsuperscript{\protect\hypertarget{hyp:email4}{4}}\hspace{0.5pt}\texttt{\footnotesize\href{mailto:nicolas.verzelen@inra.fr}{nicolas.verzelen@inra.fr}}}}
\date{\today}\footnotesize

\begin{document}
\noindent\makebox[\textwidth][c]{\begin{minipage}[c]{1.3\textwidth}\maketitle\end{minipage}}

\begin{abstract}
We study the problem of detecting whether an inhomogeneous random graph contains a planted community. Specifically, we observe a single realization of a graph. Under the null hypothesis, this graph is a sample from an inhomogeneous random graph, whereas under the alternative, there exists a small subgraph where the edge probabilities are increased by a multiplicative scaling factor. We present a scan test that is able to detect the presence of such a planted community, even when this community is very small and the underlying graph is inhomogeneous. We also derive an information theoretic lower bound for this problem which shows that in some regimes the scan test is almost asymptotically optimal. We illustrate our results through examples and numerical experiments.
\end{abstract}

\section{Introduction}
\label{sec:introduction}
Many complex systems can be described by networks of vertices connected by edges. Usually, these systems can be organized in communities, with certain groups of vertices being more densely connected than others. A central topic in the analysis of these systems is that of community detection where the goal is to find these more densely connected groups. This can often reveal interesting properties of the network with important applications in sociology, biology, computer science, and many other areas of science \cite{Fortunato2010}.

Much of the community detection literature is concentrated around methods that extract the communities from a given network, see \cite{Girvan2002,Newman2004,Newman2006}. These methods typically output an estimate of the community structure regardless of whether it really is present. Therefore, it is important to investigate when an estimated community structure is meaningful and when it simply is an artifact of the algorithm.

To answer this question, it has been highly fruitful to analyze the performance of these methods on random graphs with a known community structure. The stochastic block model is arguably the simplest model that still captures the relevant community structure, and the study of this model has led to many interesting results \cite{Massoulie2014,Mossel2015,Mossel2018,Abbe2017,Caltagirone2018,Bordenave2018}. However, there are significant drawbacks because of this simplicity: the communities are typically assumed to be very large (i.e., linear in the graph size), and the graph is homogeneous within each community (i.e., vertices within a community are exchangeable and, in particular have the same degree distribution).

To overcome these issues, several suggestions have been made. For example, the degree-corrected block model allows for inhomogeneity of vertices within each community \cite{Karrer2011}. This allows one to model real-world networks more accurately, while remaining tractable enough to obtain results similar to those obtained for the stochastic block model \cite{Gulikers2017,Gulikers2018,Gao2018a,Jin2018,Jin2019}. However, the degree-corrected block model still assumes that communities are large. To detect small communities, Arias-Castro and Verzelen consider a hypothesis testing problem where the goal is not to find communities, but instead decide whether or not any communities structure is present in an otherwise homogeneous graph \cite{Arias-Castro2014,Arias-Castro2015}.

In this paper, we also focus on the detection of small communities and we investigate when it is possible to detect the presence of a small community in an already inhomogeneous random graph. In particular, we present a scan test and provide conditions under which it is able to detect the presence of a small community. These results are valid under a wide variety of parameter choices, including cases where the underlying graph is inhomogeneous. Furthermore, we show that for some parameter choices the scan test is optimal. Specifically, we identify assumptions that ensure that if the conditions of the scan test are reversed then it is impossible for any test to detect such a community.

\section{Model and results}
\label{sec:model_and_results}
We consider the problem of detecting a planted community inside an inhomogeneous random graph. This is formalized as a hypothesis testing problem, where we observe a single instance of a simple undirected random graph $G = (V, E)$, with vertex set $V$ and edge set $E$. We denote the adjacency matrix of $G$ by $A$, i.e. $A_{ij} = \1{\{(i, j) \in E\}}$. That is, $A_{ij} = 1$ if and only if there is an edge between the vertices $i, j \in V$. Because we only consider simple graphs, we have $A_{ii} = 0$ for all $i \in V$.

Under the null hypothesis, denoted by $H_0$, the observed graph is an inhomogeneous random graph on $|V| = n$ vertices, where an edge between two vertices $i,j \in V$ is present, independently of all other edges, with probability $p_{ij}$. In other words, the entries of the adjacency matrix $A$ are independent Bernoulli random variables such that $\P_0(A_{ij} = 1) = p_{ij}$. The alternative hypothesis, denoted by $H_1$, is similar, but within a subset of the vertices the connection probabilities are increased. Formally, there is a subset $C \subseteq V$ of size $|C| = r$, called the planted community, for which the edge probabilities are increased by a multiplicative scaling factor $\scaling[C] \geq 1$. Concretely, under the alternative hypothesis the edge probabilities are $\P_1(A_{ij} = 1) = \scaling[C] \mspace{1mu} p_{ij}$ for $i, j \in C$ and $\P_1(A_{ij} = 1) = p_{ij}$ otherwise. Note that the scaling $\scaling[C]$ is allowed to depend on the location of the planted community $C \subseteq V$. This is necessary because our graphs are inhomogeneous, making the problem difficulty dependent on the location of the planted community $C \subseteq V$. Specifically, on a sparse region of the graph it is relatively difficult to detect a planted community so a strong signal $\scaling[C]$ is required to ensure a significant difference between the edge probabilities under the null hypothesis $\P_0(A_{ij} = 1) = p_{ij}$ and the edge probabilities under the alternative hypothesis $\P_1(A_{ij} = 1) = \scaling[C] p_{ij}$. On the other hand, when the community is planted on a dense region the problem is easier and a smaller signal $\scaling[C]$ could be sufficient. Throughout this paper, we assume that the location of the planted community $C \subseteq V$ is unknown, but that we do know its size $|C| = r$. In particular, we focus on the setting where $r \to \infty$ and is much smaller than $n$.

In our analysis we begin by considering the (unrealistic) case where the parameters $p_{ij}$ are all known. This allows us to get a precise characterization of the statistical difficulty of the problem. In Section~\ref{subsec:scan_test_for_unknown_rank1_edge_probabilities} we relax this assumption and show that it is possible to adapt to unknown parameters under some conditions on the structure of the edge probabilities $p_{ij}$. In particular, there we will assume that the random graph is rank-1, so that $p_{ij} = \w_i \w_j$ for some vertex weights $(\theta_i)_{i=1}^n$.

To summarize, our goal is to decide whether a given graph contains a planted community, or equivalently to decide between the hypotheses:

\begin{enumerate}[labelindent=1pt]
\vspace{-0.2\belowdisplayskip}
\item[$H_0$:]
There is no planted community, that is
\vspace{-0.7\abovedisplayskip}
\begin{flalign*}
\quad A_{ij} &\sim
  \begin{cases}
    \makebox[60pt][l]{$\text{Bern}(p_{ij}),$} &\qquad \text{if } i \neq j,\\
    \makebox[60pt][l]{$0,$} &\qquad \text{otherwise}.
  \end{cases}&&
\end{flalign*}
\vspace{-1.0\belowdisplayskip}

\item[$H_1$:]
There exists a planted community $C \subseteq V$ of size $|C| = r$, and $\scaling[C] > 1$, such that
\vspace{-0.8\abovedisplayskip}
\begin{flalign*}
\quad A_{ij} \sim
    \begin{cases}
    \makebox[60pt][l]{$\text{Bern}(\scaling[C] \mspace{1mu} p_{ij}),$} &\qquad \text{if } i \neq j, \text{ and } i, j \in C,\\
    \makebox[60pt][l]{$\text{Bern}(p_{ij}),$} &\qquad \text{if } i \neq j, \text{ and } i \notin C \text{ or } j \notin C,\\
    \makebox[60pt][l]{$0,$} &\qquad \text{otherwise}.
  \end{cases}&&
\end{flalign*}
\vspace{-1.2\belowdisplayskip}
\end{enumerate}

Note that in the above definition we are implicitly assuming that $\scaling[C]$ is not too large, so that $\scaling[C] \mspace{1mu} p_{ij} \leq 1$ for all  $i,j \in C$.

Given a graph, we want to determine which of the above models gave rise to the observation. A test $\test_n$ is any function taking as input a graph $g$ on $n$ vertices, and that outputs either $\test_n(g) = 0$ to claim that there is reason to believe that the null hypothesis is true (i.e., no community is present) or $\test_n(g) = 1$ to deem the alternative hypothesis true (i.e., the graph contains a planted community). The worst-case risk of such a test is defined as
\begin{equation}
\label{eq:worst_case_risk}
\risk_n(\test_n) \coloneqq \P_0(\test_n \neq 0) + \max_{C \subseteq V,\, |C|=r} \P_C(\test_n \neq 1) \,,
\end{equation}
where $\P_0(\cdot)$ denotes the distribution under the null hypothesis, and $\P_C(\cdot)$ denotes the distribution under the alternative hypothesis when $C \subseteq V$ is the planted community. A sequence of tests $(\test_n)_{n=1}^{\infty}$ is called asymptotically powerful when it has vanishing risk, that is $\risk_n(\test_n) \to 0$, and asymptotically powerless when it has risk tending to $1$, that is $\risk_n(\test_n) \to 1$.

Our primary goal is to characterize the asymptotic distinguishability between the null and alternative hypothesis as the graph size $n$ increases. Throughout this paper, when limits are unspecified they are taken as the graph size satisfies $n \to \infty$. The other parameters $p_{ij}$, $\scaling[C]$, and $r$ are allowed to depend on $n$, although this dependence is left implicit to avoid notational clutter.

\paragraph{Notation.} We use standard asymptotic notation: $a_n = \bigO(b_n)$ when $| a_n / b_n |$ is bounded, $a_n = \bigOmega(b_n)$ when $b_n = \bigO(a_n)$, $a_n = \bigTheta(b_n)$ when $b_n = \bigO(a_n)$ and $a_n = \bigOmega(b_n)$, $a_n = \smallO(b_n)$ when $a_n / b_n \to 0$, and $a_n \asymp b_n$ when $a_n = (1 + \smallO(1)) b_n$. Also, we use the probabilistic versions of these: $a_n = \bigOp(b_n)$ when $| a_n / b_n |$ is stochastically bounded, $a_n = \bigOmegap(b_n)$ when $b_n = \bigOp(a_n)$, $a_n = \bigThetap(b_n)$ when $b_n = \bigOp(a_n)$ and $a_n = \bigOmegap(b_n)$, and $a_n = \smallOp(b_n)$ when $a_n / b_n$ converges to $0$ in probability.

We write $e(C) \coloneqq \sum_{i,j \in C} A_{ij}$ for the number of edges in the subgraph induced by $C \subseteq V$, and $e(C,-C) \coloneqq \sum_{i \in C, j \notin C} A_{ij}$ for the number of edges between $C$ and its complement $-C = V \setminus C$. For two numbers $a, b \in \R$, we write $a \wedge b = \min\{a,b\}$, $a \vee b = \max\{a,b\}$, and $[a]_{+} = \max\{a,0\}$. Finally, define the entropy function
\begin{equation}
\label{eq:poisson_rate_function}
h(x) \coloneqq (x+1) \log(x+1) - x \,.
\end{equation}
This function plays a prominent role in most of the results.

\subsection{Information theoretic lower bound}
\label{subsec:information_theoretic_lower_bound}
We start with a result highlighting conditions under which all tests are asymptotically powerless. Here we assume that the edge probabilities $p_{ij}$, the scaling parameters $\scaling[C]$, and the size of the planted community $|C| = r$ are all known. When some of these parameters are unknown, the problem of detecting a planted community might become more difficult, hence any test that is asymptotically powerless when these parameters are known remains asymptotically powerless when they are unknown.

We prove a lower bound under two different sets of assumptions. To state these assumptions we define the \emph{average edge probability} as $\overline{p}_D = \E_0[e(D)] / \smash{\binom{|D|}{2}}$ for any $D \subseteq V$. Our assumptions correspond to different regimes of the problem in terms of planted community size $r$. For large communities we need to restrict, in a moderate way, the amount of inhomogeneity in the underlying graph, with larger communities requiring stronger restrictions on the amount of inhomogeneity. This results in the following assumption:%
\renewcommand{\theassumption}{1.1}%
\begin{assumption}
\label{ass:lower_bound_maximum_inhomogeneity}
There exists $\delta \in (0, 1/2)$ such that the following conditions hold:
\begin{enumerate}[label=(\roman*),topsep=2pt,partopsep=0pt,itemsep=2pt]
\item\label{ass:lbl:lower_bound_maximum_inhomogeneity_size}
The planted community cannot be too large, that is $r = \bigO(\smash{n^{1/2 - \delta}})$.
\item\label{ass:lbl:lower_bound_maximum_inhomogeneity_inhomogeneity}
On subgraphs $D$ much smaller than the planted community $C$, the relative edge density $\overline{p}_D / \overline{p}_C$ cannot be too large. Specifically, there exists $0 < \gamma_n = \smallO(1)$ such that
\begin{equation}
\label{eq:lower_bound_maximum_inhomogeneity}
\raisebox{3pt}{$\displaystyle\max_{C \subseteq V,|C| = r} \max_{\substack{D \subseteq C,\\|D| < r / \smash{(n/r)^{\raisebox{1pt}{$\scriptscriptstyle\mspace{-2mu}\gamma_n$}}}}}$} \: \frac{|D| \mspace{3mu} \overline{p}_D}{|C| \mspace{3mu} \overline{p}_C}
  \leq \delta \,.
\end{equation}
\item\label{ass:lbl:lower_bound_maximum_inhomogeneity_sparsity}
Every potential community $C$ must be dense enough. Specifically,
\begin{equation}
\displaystyle\max_{C \subseteq V,|C| = r} \: \frac{1}{\overline{p}_C} = \smallO\left(\frac{r}{\log(n / r)}\right) \,.
\end{equation}
\end{enumerate}
\end{assumption}
Note that the inhomogeneity restriction in Assumption~\ref{ass:lower_bound_maximum_inhomogeneity}~\ref{ass:lbl:lower_bound_maximum_inhomogeneity_inhomogeneity} only applies to small subsets $D \subseteq C$. In particular, we have $|D| / |C| < (r / n)^{\gamma_n}$ in \eqref{eq:lower_bound_maximum_inhomogeneity}, and thus if the edge probabilities differ by at most a multiplicative factor of $\bigO(\log(n)^k)$, for some fixed constant $k > 0$, then \eqref{eq:lower_bound_maximum_inhomogeneity} can always be satisfied by choosing a sequence $\gamma_n$ that converges to zero slowly enough. For example, in the homogeneous setting where the graph is an Erd\H{o}s-R\'{e}nyi random graph we know that all edge probabilities are equal and therefore \eqref{eq:lower_bound_maximum_inhomogeneity} is easily satisfied for any fixed $\delta \in (0, 1/2)$.


If the planted community size $r$ is much smaller than allowed by Assumption~\ref{ass:lower_bound_maximum_inhomogeneity}, then it is not needed to have a restriction on the inhomogeneity, provided that the graph is dense enough. This gives the following assumption:%
\renewcommand{\theassumption}{1.2}%
\begin{assumption}
\label{ass:lower_bound_maximum_community_small}
We assume that the following two conditions hold:
\begin{enumerate}[label=(\roman*),topsep=2pt,partopsep=0pt,itemsep=2pt]
\item\label{ass:lbl:lower_bound_maximum_community_small_size}
The planted community is small enough. In particular, we require that $r = n^{\smallO(1)}$.
\item\label{ass:lbl:lower_bound_maximum_community_small_sparsity}
Every potential community $C$ must be dense enough. Specifically,
\begin{equation}
\label{eq:lower_bound_maximum_community_size}
\max_{C \subseteq V,|C| = r} \: \log\left(\frac{1}{\overline{p}_C}\right) = \smallO\left(\frac{\log(n / r)}{\log(r)}\right) \,.\notag
\end{equation}
\end{enumerate}
\end{assumption}
Note that we only need one of the two assumptions above to hold in order to prove the lower bound in this section. The difference between these two assumptions is that Assumption~\ref{ass:lower_bound_maximum_inhomogeneity} works best when the planted community is large, whereas Assumption~\ref{ass:lower_bound_maximum_community_small} is more easily satisfied if the planted community is small.
Furthermore, we need that the underlying graph is not too dense. This is made precise in the following assumption:
\renewcommand{\theassumption}{\arabic{assumption}}
\setcounter{assumption}{1}
\begin{assumption}
\label{ass:lower_bound_sparsity}
We require that $\max_{C \subseteq V, |C| = r} \max_{i,j \in C} \scaling[C]^2 \mspace{1mu} p_{ij} \to 0$ as $n \to \infty$.
\end{assumption}
This assumption accomplishes two goals. First, since $\scaling[C] > 1$ it forces $p_{ij} \to 0$ for every $i, j \in V$. This ensures that the number of edges in subsets of the vertices is in essence a sufficient statistic for the testing problem. Secondly, at a more technical level, $p_{ij} \to 0$ is necessary for the Poisson approximations we use and it ensures that the differences in edge probabilities $p_{ij}$ are not magnified too much under the alternative. We note that Assumption~\ref{ass:lower_bound_sparsity} is not needed when the underlying graph is homogeneous (i.e., when the null hypothesis corresponds to an Erd\H{o}s-R\'{e}nyi random graph), see \cite{Arias-Castro2014}.

We further discuss Assumptions \ref{ass:lower_bound_maximum_inhomogeneity}, \ref{ass:lower_bound_maximum_community_small}, and \ref{ass:lower_bound_sparsity} in more detail in Section~\ref{sec:examples}. In that section we give several examples of random graphs that satisfy these assumptions.

This brings us to the main result of this section, providing conditions under which all tests are asymptotically powerless by deriving a minimax lower bound:
\begin{theorem}
\label{thm:lower_bound}
Suppose that Assumption~\ref{ass:lower_bound_sparsity} and either Assumption \ref{ass:lower_bound_maximum_inhomogeneity} or \ref{ass:lower_bound_maximum_community_small} holds. Let $0 < \epsilon < 1$ be fixed. Then all tests are asymptotically powerless if, for all $C \subseteq V$ of size $|C| = r$,
\begin{equation}
\label{eq:lower_bound_condition}
\max_{D \subseteq C} \frac{\E_0[e(D)] \mspace{1mu} h\bigl(\scaling[C] - 1\bigr)}{|D| \mspace{1mu} \log(n / |D|)} \leq 1 - \epsilon \,.
\end{equation}
\end{theorem}

Condition \eqref{eq:lower_bound_condition} has its counterpart in the work by Arias-Castro and Verzelen \cite[see~(9)]{Arias-Castro2014}, who derive a similar result when the underlying graph is an Erd\H{o}s-R\'{e}nyi random graph. However, because of the inhomogeneity in our graphs, the maximum in \eqref{eq:lower_bound_condition} is not necessarily attained at the planted community $C \subseteq V$ of size $|C| = r$, but it could be attained at any of its smaller subgraphs $D \subseteq C$. This is why our condition is more complex.

The result in Theorem~\ref{thm:lower_bound} happens to be tight, even in some scenarios where the edge probabilities $p_{ij}$ are unknown, as we construct a scan test that is powerful when the inequality in \eqref{eq:lower_bound_condition} is, roughly speaking, reversed. This is described in the next sections.

Finally, the proof of Theorem~\ref{thm:lower_bound} is given in Section~\ref{subsec:proof_of_information_theoretic_lower_bound} and follows a common methodology in these cases, by first reducing the composite alternative hypothesis to a simple alternative hypothesis and then characterizing the optimal likelihood ratio test. This is done via a second-moment method, but it requires a highly careful truncation argument to attain the sharp characterization above.

\subsection{Scan test for known edge probabilities}
\label{subsec:scan_test_for_known_edge_probabilities}
In this section we present a scan test that is asymptotically powerful. We first consider the case where all edge probabilities $p_{ij}$ and the community size $|C| = r$ are known. Although this case is unrealistic in practice, it allows us to understand the fundamental statistical limits of detection. In a sense, knowing the edge probabilities $p_{ij}$ is the most optimistic scenario, and so the focus is primarily on whether or not it is possible to detect a planted community. In the subsequent section we relax this assumption by showing how the scan test can be extended when the edge probabilities $p_{ij}$ are unknown.

Our test statistic is inspired by Bennett's inequality (see \cite[Theorem 2.9]{Boucheron2013}), which ensures that, for any $t > 0$,
\begin{equation}
\label{eq:bennett_edge_bound}
\P_0(e(D) - \E_0[e(D)] \geq t)
  \leq \exp\left(-\E_0[e(D)] \mspace{1mu} h\!\left(\frac{t}{\E_0[e(D)]}\right)\right) \,,
\end{equation}
where we recall that $h(x) = (x+1) \log(x+1) - x$. Note that this inequality is also valid when we are under the alternative hypothesis (by simply changing the subscripts $0$ to $C$). Plugging in $t = \E_0[e(D)] h^\invf(s / \E_0[e(D)])$ yields the bound
\begin{equation}
\label{eq:bennett_edge_bound_alternative}
\P_0\left(\E_0[e(D)] h\!\left(\left[\frac{e(D)}{\E_0[e(D)]} - 1\right]_{\!+}\right) \geq s\right)
\leq \e^{-s} \,.
\end{equation}
This result motivates the use of the statistic
\begin{equation}
\label{eq:known_probability_scan_test_statistic_part}
\statkp[D] \coloneqq \frac{\E_0[e(D)] h\left(\left[e(D) / \E_0[e(D)] - 1\right]_{+}\right)}{|D| \log(n / |D|)} \,,
\end{equation}
where the superscript \textsf{\small{}k} is used to differentiate between the setting with \emph{known} edge probabilities, and the setting with \emph{unknown} edge probabilities in the next section. Note that the statistic $\statkp[D]$ can be computed because $\E_0[e(D)]$ is a function of the known edge probabilities $p_{ij}$.

To construct our test, we simply scan over the whole graph, rejecting the null hypothesis when there exists a subgraph $D \subseteq V$ of size $|D| \leq r$ with an unusually high value for $\statkp[D]$. To be precise, fix $\epsilon > 0$, then the scan test rejects the null hypothesis when
\begin{equation}
\label{eq:known_probability_scan_test_statistic}
\statkp \coloneqq \max_{D \subseteq V, |D| \leq r} \statkp[D] \geq 1 + \frac{\epsilon}{2} \,.
\end{equation}
This test is essentially based on the number of edges $e(D)$ in subsets $D \subseteq V$ of size $1 \leq |D| \leq r$; rejecting the null hypothesis when there exists a subset $D \subseteq V$ for which the number of edges $e(D)$ becomes substantially larger than its expectation $\E_0[e(D)]$. So we are essentially looking for an \emph{overly} dense subset. Furthermore, the reason we need to scan over subsets smaller than $r$ is because of the possible inhomogeneity in our model; some edges carry little information and therefore it can be beneficial to ignore these edges and simply scan over a smaller subgraph instead.

Note that the proposed test is not computationally practical due to the very large number of sets one must consider in the scan (unless $r$ is very small). However, in this paper we are primarily interested in characterizing the statistical limits of possible tests, apart from computational considerations. See also the discussion in Section~\ref{sec:discussion}.

In order for the scan test to be powerful under the alternative we need $\E_C[e(D)] \to \infty$ for the most informative subgraph $D \subseteq C$, because otherwise there is a non-vanishing probability that $e(D)$ contains no edges under the alternative (by a standard Poisson approximation), making it impossible for the scan test to detect the planted community. This subgraph is characterized in the following definition:
\begin{definition}
\label{def:optimal_subset_argmax}
For every subgraph $C$ of size $|C| = r$, the most informative subgraph is
\begin{equation}
\label{eq:optimal_subset_argmax}
\optimalsubgraph[C] \coloneqq \argmax_{D \subseteq C} \frac{\E_0[e(D)]}{|D| \, \log(n / |D|)} \,.
\end{equation}
\end{definition}
The subgraph $\optimalsubgraph[C]$ in the definition above is essentially the densest subgraph under the null hypothesis. Using the above we can state the main result of this section, which provides conditions under which the scan test in \eqref{eq:known_probability_scan_test_statistic} is asymptotically powerful:
\begin{theorem}
\label{thm:known_probability_scan_test_powerful}
Suppose that all edge probabilities $p_{ij}$ and the community size $r$ are known. Then the scan test \eqref{eq:known_probability_scan_test_statistic} is asymptotically powerful when $r = \smallO(n)$, $\E_C[e(\optimalsubgraph[C])] \to \infty$ for all $C \subseteq V$ of size $|C| = r$, and
\begin{equation}
\label{eq:known_probability_scan_test_powerful_condition}
\max_{D \subseteq C} \frac{\E_0[e(D)] \mspace{1mu} h\bigl(\scaling[C] - 1\bigr)}{|D| \mspace{1mu} \log(n / |D|)} \geq 1 + \epsilon \,,
\end{equation}
where $\epsilon > 0$ comes from the definition of the scan test in \eqref{eq:known_probability_scan_test_statistic}.
\end{theorem}
This result is more widely applicable than the lower bound from Theorem~\ref{thm:lower_bound}. The condition $\E_C[e(\optimalsubgraph[C])] \to \infty$ is less stringent than either Assumption \ref{ass:lower_bound_maximum_inhomogeneity} or \ref{ass:lower_bound_maximum_community_small}. Also, there is no need for a condition like Assumption~\ref{ass:lower_bound_sparsity}. This is because we can use the upper bound from Bennett's inequality and therefore do not need the Poisson approximations necessary in deriving the lower bounds. To make this precise and to make the result in Theorem~\ref{thm:known_probability_scan_test_powerful} directly comparable to Theorem~\ref{thm:lower_bound} we provide the following corollary:%
\begin{corollary}
\label{cor:known_probability_scan_test_powerful}
Suppose that all edge probabilities $p_{ij}$ and the community size $r$ are known, and that either Assumption~\ref{ass:lower_bound_maximum_inhomogeneity} or \ref{ass:lower_bound_maximum_community_small} holds. Then the scan test in \eqref{eq:known_probability_scan_test_statistic} is asymptotically powerful when for all $C \subseteq V$ of size $|C| = r$,
\begin{equation}
\label{eq:known_probability_scan_test_powerful_condition_corollary}
\max_{D \subseteq C} \frac{\E_0[e(D)] \mspace{1mu} h\bigl(\scaling[C] - 1\bigr)}{|D| \mspace{1mu} \log(n / |D|)} \geq 1 + \epsilon \,,
\end{equation}
where $\epsilon > 0$ comes from the definition of the scan test in \eqref{eq:known_probability_scan_test_statistic}.
\end{corollary}

To show that Theorem~\ref{thm:known_probability_scan_test_powerful} applies in a broader setting than the lower bound from Theorem~\ref{thm:lower_bound} we also provide the following corollary. This shows that the scan test \eqref{eq:known_probability_scan_test_statistic} is able to detect large communities (of size larger than $\sqrt{n}$), even when the edge probabilities are very small and highly inhomogeneous:
\begin{corollary}
\label{cor:known_probability_scan_test_powerful2}
Suppose that all edge probabilities $p_{ij}$ and the community size $r$ are known. Define $\pmax \coloneqq \max_{i,j \in V} p_{ij}$ and $\pmin \coloneqq \min_{i\neq j \in V} p_{ij}$. If $r \geq n^a$, $\pmin \geq n^{-2b}$, and $\pmax / \pmin = \smallO(n^{a - b})$ for $0 < b < a < 1$, then the scan test in \eqref{eq:known_probability_scan_test_statistic} is asymptotically powerful when for all $C \subseteq V$ of size $|C| = r$,
\begin{equation}
\max_{D \subseteq C} \frac{\E_0[e(D)] \mspace{1mu} h\bigl(\scaling[C] - 1\bigr)}{|D| \mspace{1mu} \log(n / |D|)} \geq 1 + \epsilon \,,
\end{equation}
where $\epsilon > 0$ comes from the definition of the scan test in \eqref{eq:known_probability_scan_test_statistic}.
\end{corollary}
In the corollary above, both $a$ and $b$ above may depend on the graph size $n$. In particular, if $\pmax / \pmin = \bigO(1)$ then it is possible that $a - b = \smallO(1)$, provided that $(a - b)\log(n) \to \infty$. For instance, it is necessary to have $a - b = \smallO(1)$ in order to satisfy Assumption~\ref{ass:lower_bound_maximum_inhomogeneity}~\ref{ass:lbl:lower_bound_maximum_inhomogeneity_inhomogeneity}.

A downside of the scan test presented in this section is that it requires knowledge of all edge probabilities $p_{ij}$. In practice, these are often unavailable to a statistician. The next section is devoted to extending the scan test to cope with unknown edge probabilities, assuming that the edge probabilities have a rank-1 structure.

\subsection{Scan test for unknown rank-1 edge probabilities}
\label{subsec:scan_test_for_unknown_rank1_edge_probabilities}
In this section we show how the scan test from the previous section can be extended to the setting where the edge probabilities $p_{ij}$ are unknown. We do still assume that the community size $|C| = r$ is known. As can be seen in \eqref{eq:known_probability_scan_test_statistic_part}, the scan statistic depends on the edge probabilities $p_{ij}$ only through $\E_0[e(D)] = \sum_{i<j \in D} p_{ij}$. Therefore, a natural way to approach the situation where the edge probabilities $p_{ij}$ are unknown is to devise a good surrogate for $\E_0[e(D)]$ that can be computed solely based on the observed graph (which could be a sample from either the null hypothesis or alternative hypothesis). Clearly, this is not possible in full generality, but if the edge probabilities have some additional structure then this become possible.

Here we consider the scenario where, under the null hypothesis, the edge probabilities $p_{ij}$ have a so-called rank-1 structure. The resulting model is sometimes also called a hidden-variable model. That is, we assume that each vertex $i \in V$ is assigned a weight $\w_i \in (0, 1)$ and that the edge probabilities are given by $p_{ij} = \w_i \w_j$. This is probably one of the simplest models for inhomogeneous random graphs possible. Note that this model is very similar to the degree corrected stochastic block model \cite{Karrer2011,Gulikers2017,Gao2018a}, except that our focus is on the detection of small communities, whereas the literature on stochastic block models is typically concerned with the detection of much larger communities. Further, there are strong connections between this model and the configuration model \cite{Britton2006,VanderHofstad2017}.

To make it possible to estimate $\E_0[e(D)]$ we need to assume that the graph is not too inhomogeneous and not too sparse, as formulated in the following assumption:
\begin{assumption}
\label{ass:unknown_probability_maximum_inhomogeneity}
Let $\wmax = \max_{i \in V} \w_i$ and $\wmin = \min_{i \in V} \w_i$, then the maximum allowed inhomogeneity is 
\begin{equation}
\left(\frac{\wmax}{\wmin}\right)^{\!2} = \smallO\left(r^{2/3} \medwedge \frac{n}{r} \mspace{1mu} \wmin^2\right) \,.
\end{equation}
\end{assumption}
Using the above assumption, we will show that it is possible to estimate $\E_0[e(D)]$ by using the observed edges going from $D$ to the rest of the graph $-D = V \setminus D$. Note that the exponent $2 / 3$ in Assumption~\ref{ass:unknown_probability_maximum_inhomogeneity} is not an arbitrary choice, but as we explain below, it is actually the best possible exponent that still ensures that our estimator works.

When $C \subseteq V$ is the planted community, and we estimate $\E_0[e(D)]$ for a large enough subgraph $D \subseteq C$ using this approach, we will obtain an almost unbiased estimate both under $H_0$ as well as under $H_1$. This is because enough of the edges used in this estimate have the same distribution under the null and alternative hypothesis. Our estimator is based on the identity
\begin{equation}
\label{eq:edge_count_identity}
\E_0[e(D)]
  = \frac{\!\biggl(\raisebox{-1pt}{$\!\sqrt{\E_0[e(V)] \mspace{-2mu}+\mspace{-2mu} \frac{1}{2} \rule{0pt}{10pt}\smash{\raisebox{2.3pt}{\scalebox{0.8}{$\displaystyle\sum_{i \in V}\w_i^2$}}}} - \sqrt{\E_0[e(V)] \mspace{-2mu}+\mspace{-2mu} \frac{1}{2} \rule{0pt}{10pt}\smash{\raisebox{2.3pt}{\scalebox{0.8}{$\displaystyle\sum_{i \in V}\w_i^2$}}} \mspace{-2mu}-\mspace{-2mu} 2 \E_0[e(D,-D)]}$}\biggr)^{\!2}\!}{4} - \frac{1}{2} \raisebox{2pt}{\scalebox{0.9}{$\displaystyle\sum_{i \in D}\w_i^2$}} \,.
\end{equation}
This identity is explained in more detail in Section~\ref{subsec:derivation_edge_count_identity}, and it is valid when Assumption~\ref{ass:unknown_probability_maximum_inhomogeneity} holds and $n$ is large enough. Note that both $\E_0[e(V)]$ and $\E_0[e(D,-D)]$ are the sum of a large number of edge probabilities $p_{ij} = \w_i \w_j$, and most of these remain unaffected under the alternative hypothesis. Because of this, and since $\sum_{i\in V} \theta_i^2$ will generally be negligible, we will estimate $\E_0[e(D)]$ by
\begin{equation}
\label{eq:edge_count_estimator}
\est[D] \coloneqq \frac{\left(\sqrt{e(V)} - \sqrt{e(V) - 2 e(D,-D)}\right)^2}{4} \,.
\end{equation}
Here we have used that $(\wmax / \wmin)^2 \leq r^{2/3}$ by Assumption~\ref{ass:unknown_probability_maximum_inhomogeneity}, which ensures that the term $\sum_{i \in D} \w_i^2 / 2$ in \eqref{eq:edge_count_identity} becomes negligible, and therefore that our estimator $\est[C]$ is a good surrogate for $\E_0[e(D)]$. This also explains the exponent $2/3$ appearing in Assumption~\ref{ass:unknown_probability_maximum_inhomogeneity}, as this is the largest exponent that still guarantees that the term $\sum_{i \in D} \w_i^2 / 2$ is negligible. This is discussed in more detail in Section~\ref{subsec:proof_of_scan_test_for_unknown_rank1_edge_probabilities_is_powerful}.

In most cases, the estimator in \eqref{eq:edge_count_estimator} can essentially be used as a plugin for the scan test of the previous section. However, this estimator might not concentrate very well when $\E_0[e(D)]$ becomes too small. To remedy this, we use a thresholded version of the estimator given by
\begin{equation}
\label{eq:edge_count_estimator_threshold}
\esttr[D] \coloneqq \biggl(\est[D] \medvee \frac{|D|^2}{n} \log^{4}(n / |D|)\biggr) \,.
\end{equation}
Using the thresholded estimator in \eqref{eq:edge_count_estimator_threshold}, we can consider the same scan test as in the previous section but with $\E_0[e(D)]$ replaced by the estimator $\esttr[D]$. This leads to the definition of the scan test for unknown edge probabilities as
\begin{equation}
\label{eq:unknown_probability_scan_test_statistic_part}
\statup[D] \coloneqq \frac{\esttr[D] \mspace{1mu} h\Bigl(\bigl[e(D) / \esttr[D] - 1\bigr]_{+}\Bigr)}{|D| \mspace{1mu} \log(n / |D|)} \,,
\end{equation}
where the superscript \textsf{\small{}u} is used to indicate that we consider the setting with \emph{unknown} rank-1 edge probabilities.

As in the previous section, we scan over subgraphs and reject the null hypothesis when $\statup[D]$ becomes too large. However, as explained above, when scanning over subgraphs $D \subseteq V$ whose size $|D|$ is much smaller than $|C| = r$ we run into a problem because of the bias in $\esttr[D]$. Luckily this is not a problem because Assumption~\ref{ass:unknown_probability_maximum_inhomogeneity} ensures that asymptotically the maximum of $\statup[D]$ will always be attained at a subgraph of size $|D| \geq r^{1/3}$, see the proof of Lemma~\ref{lem:optimal_subset_size} in Section~\ref{subsec:proof_of_scan_test_for_unknown_rank1_edge_probabilities_is_powerful}. Therefore, for $\epsilon > 0$ fixed, the scan test for unknown edge probabilities rejects the null hypothesis when
\begin{equation}
\label{eq:unknown_probability_scan_test_statistic}
\statup \coloneqq \max_{D \subseteq V,\, r^{1/3} \leq |D| \leq r} \statup[D] \geq 1 + \frac{\epsilon}{3} \,.
\end{equation}

This brings us to the main result of this section, which provides conditions for the scan test in \eqref{eq:unknown_probability_scan_test_statistic} to be asymptotically powerful:
\begin{theorem}
\label{thm:unknown_probability_scan_test_powerful}
Suppose that the community size $r$ is known and that Assumption~\ref{ass:unknown_probability_maximum_inhomogeneity} holds. Then the scan test \eqref{eq:unknown_probability_scan_test_statistic} is asymptotically powerful when $r = \smallO(n)$, $\E_C[e(\optimalsubgraph[C])] \to \infty$ for all $C \subseteq V$ of size $|C| = r$, and
\begin{equation}
\label{eq:unknown_probability_scan_test_powerful_condition}
\max_{D \subseteq C} \frac{\E_0[e(D)] \mspace{1mu} h\bigl(\scaling[C] - 1\bigr)}{|D| \mspace{1mu} \log(n / |D|)} \geq 1 + \epsilon \,,
\end{equation}
where $\epsilon > 0$ comes from the definition of the scan test in \eqref{eq:unknown_probability_scan_test_statistic}.
\end{theorem}
Comparing this result with Theorem~\ref{thm:known_probability_scan_test_powerful}, we see that for rank-1 random graphs, Assumption~\ref{ass:unknown_probability_maximum_inhomogeneity} is the only extra condition necessary when the edge probabilities are unknown. Furthermore, by the same argument as in the previous section it can be shown that either Assumption \ref{ass:lower_bound_maximum_inhomogeneity} or \ref{ass:lower_bound_maximum_community_small} is sufficient to ensure that $\E_C[e(\optimalsubgraph[C])] \to \infty$. Therefore, to make the result in Theorem~\ref{thm:unknown_probability_scan_test_powerful} directly comparable to Theorem~\ref{thm:lower_bound} we provide the following corollary:
\begin{corollary}
\label{cor:unknown_probability_scan_test_powerful}
Suppose that the community size $r$ is known and that Assumption~\ref{ass:unknown_probability_maximum_inhomogeneity}, and either Assumption~\ref{ass:lower_bound_maximum_inhomogeneity} or \ref{ass:lower_bound_maximum_community_small} holds. Then the scan test \eqref{eq:known_probability_scan_test_statistic} is asymptotically powerful when, for all $C \subseteq V$ of size $|C| = r$,
\begin{equation}
\label{eq:unknown_probability_scan_test_powerful_condition_corollary}
\max_{D \subseteq C} \frac{\E_0[e(D)] \mspace{1mu} h\bigl(\scaling[C] - 1\bigr)}{|D| \mspace{1mu} \log(n / |D|)} \geq 1 + \epsilon \,,
\end{equation}
where $\epsilon > 0$ comes from the definition of the scan test in \eqref{eq:unknown_probability_scan_test_statistic}.
\end{corollary}

Moreover, a result similar to Corollary~\ref{cor:known_probability_scan_test_powerful2} also applies in the setting with unknown edge probabilities. This leads to the following result:
\begin{corollary}
\label{cor:unknown_probability_scan_test_powerful2}
Suppose the community size $r$ is known and that Assumption~\ref{ass:unknown_probability_maximum_inhomogeneity} holds. If $r \geq n^a$, $\wmin \geq n^{-b}$, and $(\wmax / \wmin)^2 = \smallO(n^{a - b})$ for $0 < b < a < 1$, then the scan test in \eqref{eq:known_probability_scan_test_statistic} is asymptotically powerful when for all $C \subseteq V$ of size $|C| = r$,
\begin{equation}
\max_{D \subseteq C} \frac{\E_0[e(D)] \mspace{1mu} h\bigl(\scaling[C] - 1\bigr)}{|D| \mspace{1mu} \log(n / |D|)} \geq 1 + \epsilon \,,
\end{equation}
where $\epsilon > 0$ comes from the definition of the scan test in \eqref{eq:unknown_probability_scan_test_statistic}.
\end{corollary}

\section{Examples}
\label{sec:examples}
The results in the previous section provide conditions for when it is possible to detect a planted community $C \subseteq V$. When the scaling $\scaling[C]$ is large enough it is asymptotically possible to detect a planted community using the scan test, and when the scaling $\scaling[C]$ is too small it is impossible for any test to detect a planted community. To understand at which scaling $\scaling[C]$ this change happens, we need to characterize the behavior of
\begin{equation}
\label{eq:optimal_subgraph_condition}
\max_{D \subseteq C} \frac{\E_0[e(D)] \mspace{1mu} h\bigl(\scaling[C] - 1\bigr)}{|D| \mspace{1mu} \log(n / |D|)} \approx 1\,.
\end{equation}
The subgraph that attains the maximum above will be denoted by $\optimalsubgraph = \optimalsubgraph[C]$ and was defined in Definition~\ref{def:optimal_subset_argmax}. In this section, we present several examples of different random graph models and illustrate how \eqref{eq:optimal_subgraph_condition} depends on the inhomogeneity structure. For clarity of presentation, the parameters in these examples are chosen such that the scaling $\scaling[C]$ satisfying \eqref{eq:optimal_subgraph_condition} always converges to a constant.

In the examples below, the lower bound from Theorem~\ref{thm:lower_bound} as well as the upper bound from Theorems \ref{thm:known_probability_scan_test_powerful} and \ref{thm:unknown_probability_scan_test_powerful} are applicable because Assumptions \ref{ass:lower_bound_maximum_community_small}, \ref{ass:lower_bound_sparsity} and \ref{ass:unknown_probability_maximum_inhomogeneity} are all satisfied\footnote{\scriptsize{}The examples in Section~\ref{subsec:example_arbitrary_number_of_weights} consider randomly sampled vertex weights, and therefore the assumptions in this section hold with high probability. Furthermore, this section also contains some examples where Assumption~\ref{ass:lower_bound_maximum_inhomogeneity} instead of Assumption~\ref{ass:lower_bound_maximum_community_small} holds.}. Furthermore, it can be checked that Assumption~\ref{ass:lower_bound_maximum_inhomogeneity} \ref{ass:lbl:lower_bound_maximum_inhomogeneity_size} and \ref{ass:lbl:lower_bound_maximum_inhomogeneity_inhomogeneity} are also satisfied. Thus, the only reason why Assumption~\ref{ass:lower_bound_maximum_inhomogeneity} does not hold in the examples below is because the edge density condition from Assumption~\ref{ass:lower_bound_maximum_inhomogeneity}~\ref{ass:lbl:lower_bound_maximum_inhomogeneity_sparsity} is not satisfied. The reason for this is that it is not possible to simultaneously satisfy that edge density condition and have the scaling $\scaling[C]$ from \eqref{eq:optimal_subgraph_condition} converge to a constant larger than $1$. This means that a choice had to be made between either selecting examples that satisfy Assumption~\ref{ass:lower_bound_maximum_inhomogeneity} or having $\scaling[C] - 1$ converge to a positive constant. We choose for the latter option to improve the clarity of presentation.

There are, however, also many interesting examples where Assumption~\ref{ass:lower_bound_maximum_inhomogeneity} does hold. For instance, it is possible to satisfy Assumption~\ref{ass:lower_bound_maximum_inhomogeneity} in any of the examples below by simply increasing the community size $r$ or the edge density (by increasing all vertex weights by the same factor). 
Thus, in the examples below, it is possible to apply Theorems \ref{thm:lower_bound}, \ref{thm:known_probability_scan_test_powerful}, and \ref{thm:unknown_probability_scan_test_powerful} because Assumptions \ref{ass:lower_bound_maximum_community_small}, \ref{ass:lower_bound_sparsity} and \ref{ass:unknown_probability_maximum_inhomogeneity} hold, and this remains true for larger community sizes or denser graphs but then because of Assumptions \ref{ass:lower_bound_maximum_inhomogeneity}, \ref{ass:lower_bound_sparsity} and \ref{ass:unknown_probability_maximum_inhomogeneity}. This explains how Assumptions \ref{ass:lower_bound_maximum_inhomogeneity} and \ref{ass:lower_bound_maximum_community_small} are nicely complementing each other to make our results applicable in a wide range of scenarios.

\subsection{Erd\texorpdfstring{\H{o}}{\"{o}}s-R\'{e}nyi random graph}
\label{subsec:example_erdos_renyi}
The arguably simplest setting where we can apply our results is that of an Erd\H{o}s-R\'{e}nyi random graph, where all edge probabilities $p_{ij} = p$ are equal, so that the graph is completely homogeneous. In this case, the subgraph $\optimalsubgraph$ that attains the maximum in \eqref{eq:optimal_subgraph_condition} is always the complete planted community $C \subseteq V$. Let $r = \smallO(n)$, $r \to \infty$ and $p \to 0$ be such that $r^2 p \to \infty$. One easily sees that \eqref{eq:optimal_subgraph_condition} becomes
\begin{equation}
\label{eq:optimal_subgraph_condition_erdos_renyi}
\max_{D \subseteq C} \frac{\E_0[e(D)] \mspace{1mu} h\bigl(\scaling[C] - 1\bigr)}{|D| \mspace{1mu} \log(n / |D|)}
  = \frac{\E_0[e(C)] \mspace{1mu} h\bigl(\scaling[C] - 1\bigr)}{|C| \mspace{1mu} \log(n / |C|)}
  \asymp \frac{r \mspace{1mu} p \mspace{1mu} h\bigl(\scaling[C] - 1\bigr)}{2 \log(n / r)}
  \asymp \frac{r \mspace{1mu} H_{p}(\scaling[C] \mspace{1mu} p)}{2 \log(n / r)} \,.
\end{equation}
where $H_{p}(\scaling[C] \mspace{1mu} p)$ is the Kullback-Leibler divergence between $\text{Bern}(p)$ and $\text{Bern}(\scaling[C] \mspace{1mu} p)$. Note that this is the same condition found by Arias-Castro and Verzelen, who considered the problem of detecting a planted community in an Erd\H{o}s-R\'{e}nyi random graph \cite[see (9) and (15)]{Arias-Castro2014}.

\subsection{Rank-1 random graph with 2 weights}
\label{subsec:example_stochastic_block_model}
A slightly more complex setting is where the underlying graph has a rank-1 structure with two different weights. Some of the vertices have large weight $\wmax$, and the remaining vertices have small weight $\wmin$. Therefore, there are three different edge probabilities in the underlying graph: $p_{ij} = \wmax^2$ when both endpoints have large weight, $p_{ij} = \wmin^2$ when both endpoints have small weight, and $p_{ij} = \wmax \wmin$ when one of the endpoints has large weight and the other small weight.

The subgraph $\optimalsubgraph[C]$ that attains the maximum in \eqref{eq:optimal_subgraph_condition} depends crucially on the amount of inhomogeneity in $C \subseteq V$, and because we only have two different weights this translates to the ratio of vertices with large weight $\wmax$ and vertices with small weight $\wmin$ in $C$. Moreover, it can be checked that the maximum in \eqref{eq:optimal_subgraph_condition} is attained either on the whole subgraph $C$, or on the subgraph $C_\textup{max} \subseteq C$ consisting of only the large-weight vertices in $C$. Specifically, assuming $\log(n / |C|) \asymp \log(n)$, the maximum in \eqref{eq:optimal_subgraph_condition} is attained at $C_\textup{max}$ when
\begin{equation}
\label{eq:optimal_subgraph_condition_2_weight}
|C_\textup{max}| > (1 + \smallO(1)) \frac{|C| - 1 + (\wmax / \wmin)^2}{(\wmax / \wmin - 1)^2} \,,
\end{equation}
and otherwise it is attained at $C$. Here we can see that the amount of inhomogeneity plays an important role in determining the maximum in \eqref{eq:optimal_subgraph_condition}, and therefore in determining whether a planted community can be detected or not.

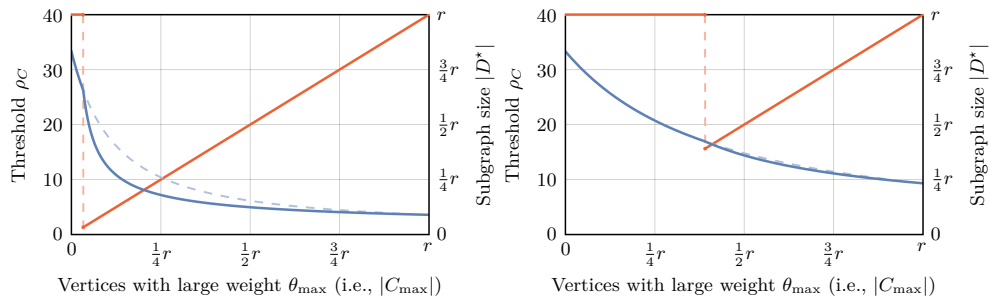
\begin{figure}[b!]
\centering
\begin{subfigure}{.5\textwidth}
\raggedright
\begin{tikzpicture}[scale=0.8, every node/.style={transform shape}]
\draw ( 0.0,  0.0) node[inner sep=0] {\includegraphics[scale=0.1]{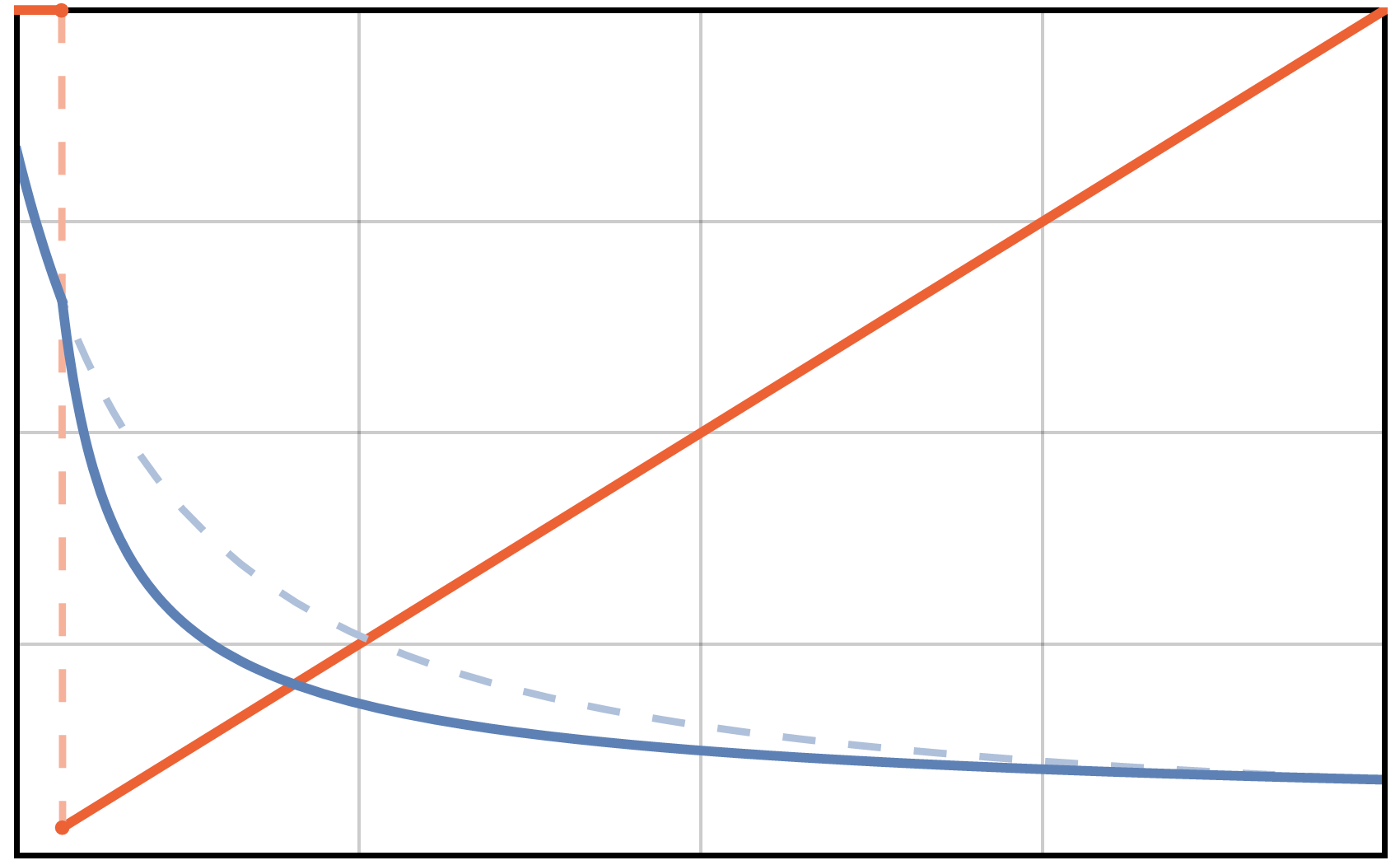}};

\draw (-2.930, -0.008) node[inner sep=4pt,anchor=south,xshift=-16pt,rotate=90] {\small{}Threshold $\scaling[C]$};
\draw (-2.930,  1.800) node[inner sep=4pt,anchor=east] {\small$40$};
\draw (-2.930,  0.894) node[inner sep=4pt,anchor=east] {\small$30$};
\draw (-2.930, -0.008) node[inner sep=4pt,anchor=east] {\small$20$};
\draw (-2.930, -0.912) node[inner sep=4pt,anchor=east] {\small$10$};
\draw (-2.930, -1.816) node[inner sep=4pt,anchor=east] {\small$0$};

\draw ( 0.000, -1.816) node[inner sep=4pt,anchor=north,yshift=-16pt] {\small{}Vertices with large weight $\wmax$ (i.e., $|C_\textup{max}|$)};
\draw (-2.930, -1.816) node[inner sep=4pt,anchor=north] {\small$0$};
\draw (-1.464, -1.816) node[inner sep=4pt,anchor=north] {\small$\frac{1}{4} r$};
\draw ( 0.000, -1.816) node[inner sep=4pt,anchor=north] {\small$\frac{1}{2} r$};
\draw ( 1.464, -1.816) node[inner sep=4pt,anchor=north] {\small$\frac{3}{4} r$};
\draw ( 2.930, -1.816) node[inner sep=4pt,anchor=north] {\small$r$};

\draw ( 2.930, -0.008) node[inner sep=4pt,anchor=north,xshift=18pt,rotate=90] {\small{}Subgraph size $|\optimalsubgraph|$};
\draw ( 2.930, -1.816) node[inner sep=4pt,anchor=west] {\small$0$};
\draw ( 2.930, -0.912) node[inner sep=4pt,anchor=west] {\small$\frac{1}{4} r$};
\draw ( 2.930, -0.008) node[inner sep=4pt,anchor=west] {\small$\frac{1}{2} r$};
\draw ( 2.930,  0.894) node[inner sep=4pt,anchor=west] {\small$\frac{3}{4} r$};
\draw ( 2.930,  1.800) node[inner sep=4pt,anchor=west] {\small$r$};
\end{tikzpicture}
\caption{\fussy $r = \lfloor \log(n)^{3} \rfloor$, $\wmax = \smash{\frac{1}{\log(n)}}$, $\wmin = \smash{\frac{1}{6.5 \log(n)}}$.$\;\;\;$}
\label{fig:two_weights_scaling_example_sub1}
\end{subfigure}%
\begin{subfigure}{.5\textwidth}
\raggedleft
\begin{tikzpicture}[scale=0.8, every node/.style={transform shape}]
\draw ( 0.0,  0.0) node[inner sep=0] {\includegraphics[scale=0.1]{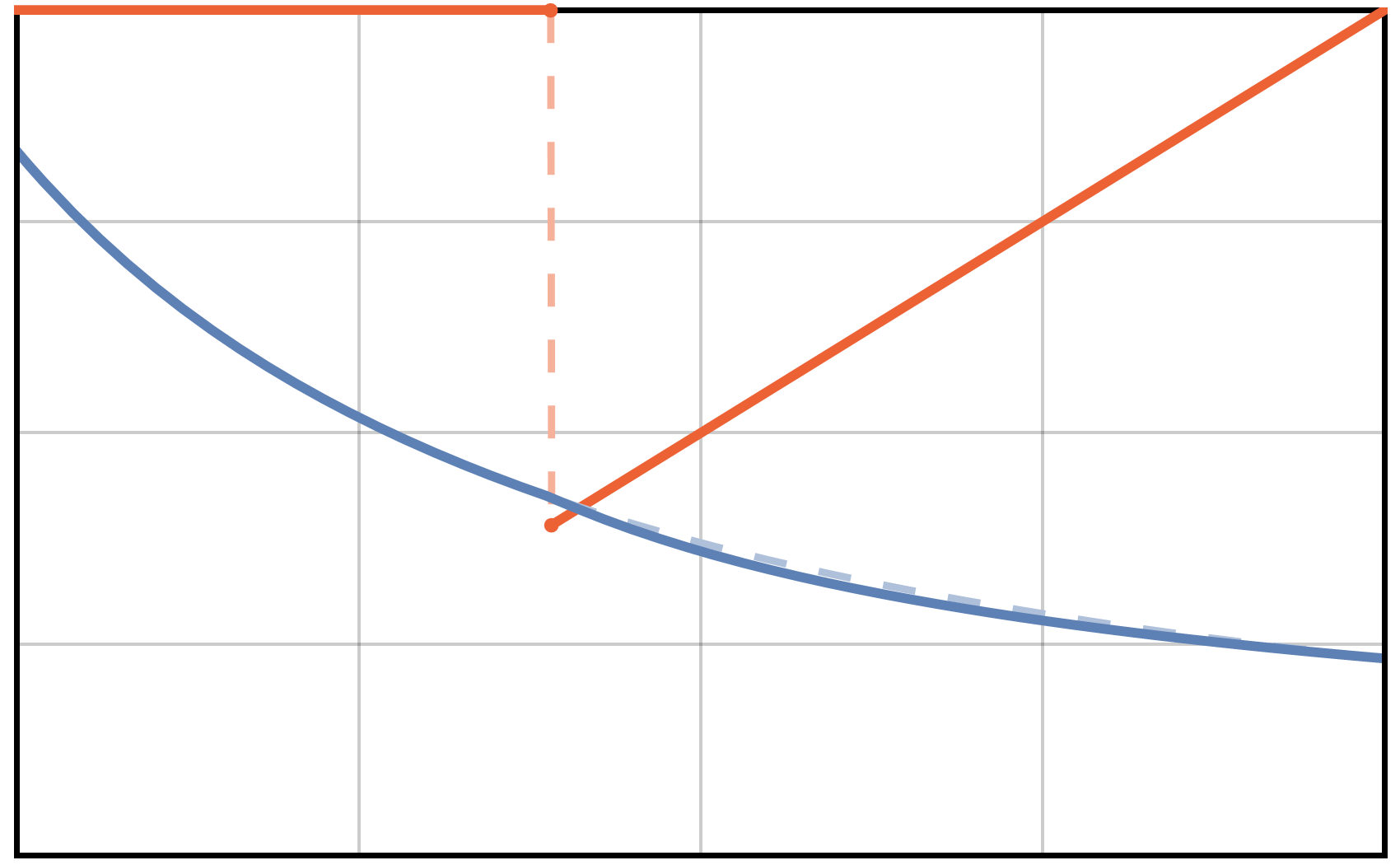}};

\draw (-2.930, -0.008) node[inner sep=4pt,anchor=south,xshift=-16pt,rotate=90] {\small{}Threshold $\scaling[C]$};
\draw (-2.930,  1.800) node[inner sep=4pt,anchor=east] {\small$40$};
\draw (-2.930,  0.894) node[inner sep=4pt,anchor=east] {\small$30$};
\draw (-2.930, -0.008) node[inner sep=4pt,anchor=east] {\small$20$};
\draw (-2.930, -0.912) node[inner sep=4pt,anchor=east] {\small$10$};
\draw (-2.930, -1.816) node[inner sep=4pt,anchor=east] {\small$0$};

\draw ( 0.000, -1.816) node[inner sep=4pt,anchor=north,yshift=-16pt] {\small{}Vertices with large weight $\wmax$ (i.e., $|C_\textup{max}|$)};
\draw (-2.930, -1.816) node[inner sep=4pt,anchor=north] {\small$0$};
\draw (-1.464, -1.816) node[inner sep=4pt,anchor=north] {\small$\frac{1}{4} r$};
\draw ( 0.000, -1.816) node[inner sep=4pt,anchor=north] {\small$\frac{1}{2} r$};
\draw ( 1.464, -1.816) node[inner sep=4pt,anchor=north] {\small$\frac{3}{4} r$};
\draw ( 2.930, -1.816) node[inner sep=4pt,anchor=north] {\small$r$};

\draw ( 2.930, -0.008) node[inner sep=4pt,anchor=north,xshift=18pt,rotate=90] {\small{}Subgraph size $|\optimalsubgraph|$};
\draw ( 2.930, -1.816) node[inner sep=4pt,anchor=west] {\small$0$};
\draw ( 2.930, -0.912) node[inner sep=4pt,anchor=west] {\small$\frac{1}{4} r$};
\draw ( 2.930, -0.008) node[inner sep=4pt,anchor=west] {\small$\frac{1}{2} r$};
\draw ( 2.930,  0.894) node[inner sep=4pt,anchor=west] {\small$\frac{3}{4} r$};
\draw ( 2.930,  1.800) node[inner sep=4pt,anchor=west] {\small$r$};
\end{tikzpicture}
\caption{\fussy $r = \lfloor \log(n)^{3} \rfloor$, $\wmax = \smash{\frac{1}{2.5 \log(n)}}$, $\wmin = \smash{\frac{1}{6.5 \log(n)}}$.}
\label{fig:two_weights_scaling_example_sub2}
\end{subfigure}
\caption{Example of the threshold scaling $\scaling[C]$ required for detecting a planted community using the optimal subgraph $\optimalsubgraph[C]$ (blue, left axis) and the threshold scaling $\scaling[C]$ required when using the whole subgraph $C$ instead (dashed blue, left axis), together with the size of the optimal subgraph $|\optimalsubgraph[C]|$ (red, right axis). The specific numerical values are simply chosen to highlight the different regimes possible; other choices produce similar results.}
\label{fig:two_weights_scaling_example}
\end{figure}

In Figure~\ref{fig:two_weights_scaling_example} we give two examples of the threshold scaling $\scaling[C]$ required for the scan test to be asymptotically powerful. When, for every $C \subseteq V$, the scaling $\scaling[C]$ is chosen above the blue curve then the scan test is asymptotically powerful by Theorems \ref{thm:known_probability_scan_test_powerful} and \ref{thm:unknown_probability_scan_test_powerful}, and when it is chosen below the blue curve then all tests are asymptotically powerless by Theorem~\ref{thm:lower_bound}. Here we can clearly see a sharp bend in the blue curve at the point where $|C_\textup{max}|$ crosses the threshold in \eqref{eq:optimal_subgraph_condition_2_weight}. This happens because there are many vertices with large weight when $|C_\textup{max}|$ is large and it is optimal to only use these vertices when trying to detect a planted community. However, there no longer are enough vertices with large weight when $|C_\textup{max}|$ becomes too small and it becomes more beneficial to also use the vertices with small weight.

\subsection{Rank-1 random graph with 3 weights}

Extending the setting in the previous section, we can consider a rank-1 random graph with three different weights. Some vertices have large weight $\wmax$, some vertices have medium weight $\w_\textup{med}$, and the remaining vertices have small weight $\wmin$. In this setting the situation becomes even more complex, and the subgraph $\optimalsubgraph$ that attains the maximum in \eqref{eq:optimal_subgraph_condition} depends on the amount of vertices of each type in $C \subseteq V$.

In Figure~\ref{fig:three_weights_scaling_example} we give an example of the threshold scaling $\scaling[C]$ required for the scan test to be asymptotically powerful in the setting with three weights. When, for every $C \subseteq V$, the scaling $\scaling[C]$ is chosen above the surface then the scan test is asymptotically powerful by Theorems \ref{thm:known_probability_scan_test_powerful} and \ref{thm:unknown_probability_scan_test_powerful}, and when it is chosen below the surface then all tests are asymptotically powerless by Theorem~\ref{thm:lower_bound}. We can see that when there are enough vertices with large weight $\wmax$ then it is optimal to only use these large-weight vertices (green region), but as the number of large-weight vertices decreases it becomes beneficial to include also medium-weight vertices (orange region) or even small-weight vertices (blue region). Note that the cross-section with no medium-weight vertices is the same as Figure~\ref{fig:two_weights_scaling_example_sub1} and the cross-section with no large-weight vertices is the same as Figure~\ref{fig:two_weights_scaling_example_sub2}.

\begin{figure}[b!]
\centering
\begin{tikzpicture}[scale=0.8, every node/.style={transform shape}]
\draw ( 0.0,  0.0) node[inner sep=0] {\includegraphics[scale=0.1]{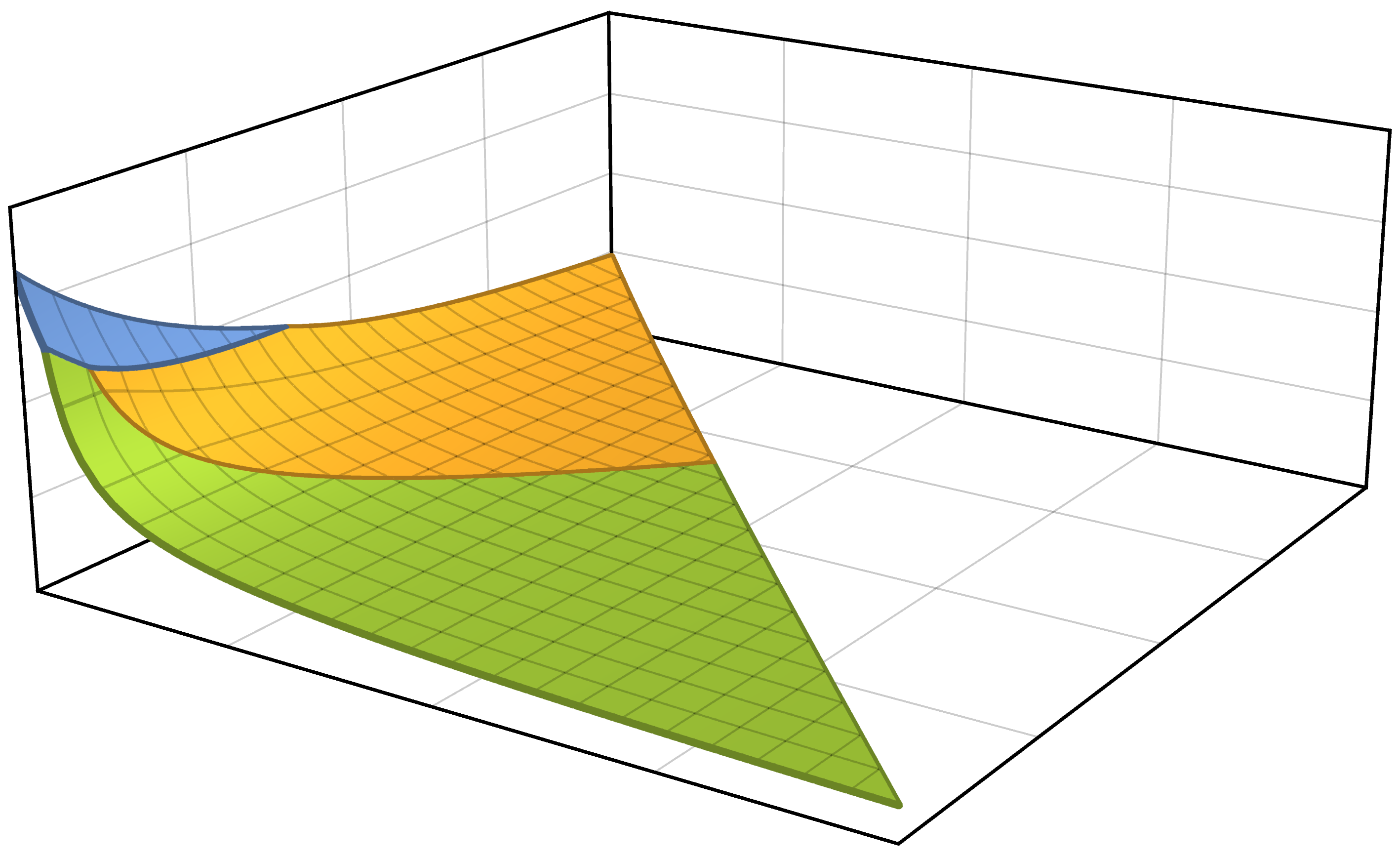}};


\draw (-4.832,  0.186) node[inner sep=4pt,anchor=south,xshift=-16pt,rotate=94.1] {\small{}Threshold $\scaling[C]$};
\draw (-4.931,  1.580) node[inner sep=4pt,anchor=east] {\small$40$};
\draw (-4.884,  0.878) node[inner sep=4pt,anchor=east] {\small$30$};
\draw (-4.832,  0.186) node[inner sep=4pt,anchor=east] {\small$20$};
\draw (-4.784, -0.496) node[inner sep=4pt,anchor=east] {\small$10$};

\draw (-1.915, -1.988) node[inner sep=4pt,anchor=north,xshift=-14pt,yshift=-16pt,rotate=-16.4] {\small{}Vertices with large weight $\wmax$};
\draw (-4.726, -1.154) node[inner xsep=3pt,inner ysep=2pt,anchor=north east] {\small$0$};
\draw (-3.374, -1.558) node[inner xsep=3pt,inner ysep=2pt,anchor=north east] {\small$\frac{1}{4} r$};
\draw (-1.915, -1.988) node[inner xsep=3pt,inner ysep=2pt,anchor=north east] {\small$\frac{1}{2} r$};
\draw (-0.322, -2.454) node[inner xsep=3pt,inner ysep=2pt,anchor=north east] {\small$\frac{3}{4} r$};
\draw ( 1.424, -2.968) node[inner xsep=3pt,inner ysep=2pt,anchor=north east] {\small$r$};

\draw ( 3.295, -1.538) node[inner sep=4pt,anchor=north,xshift=18pt,yshift=-12pt,rotate=37.3] {\small{}Vertices with medium weight $\w_\textup{med}$};
\draw ( 1.424, -2.968) node[inner xsep=3pt,inner ysep=1pt,anchor=north west] {\small$0$};
\draw ( 2.415, -2.207) node[inner xsep=3pt,inner ysep=1pt,anchor=north west] {\small$\frac{1}{4} r$};
\draw ( 3.295, -1.538) node[inner xsep=3pt,inner ysep=1pt,anchor=north west] {\small$\frac{1}{2} r$};
\draw ( 4.070, -0.945) node[inner xsep=3pt,inner ysep=1pt,anchor=north west] {\small$\frac{3}{4} r$};
\draw ( 4.756, -0.420) node[inner xsep=3pt,inner ysep=1pt,anchor=north west] {\small$r$};
\end{tikzpicture}
\caption{Example of the threshold scaling $\scaling[C]$ required for detecting a planted community when using the optimal subgraph $\optimalsubgraph[C]$. In the blue region $\optimalsubgraph[C]$ consists of all vertices, in the orange region $\optimalsubgraph[C]$ consists of both large and medium-weight vertices, and in the green region $\optimalsubgraph[C]$ consists only of large-weight vertices. The parameters used are 
$r = \lfloor \log(n)^{3} \rfloor$, $\wmax = \smash{\frac{1}{\log(n)}}$, $\w_\textup{med} = \smash{\frac{1}{2.5 \log(n)}}$, $\wmin = \smash{\frac{1}{6.5 \log(n)}}$. These values are chosen for ease of comparison with Figure~\ref{fig:two_weights_scaling_example}.}
\label{fig:three_weights_scaling_example}
\end{figure}
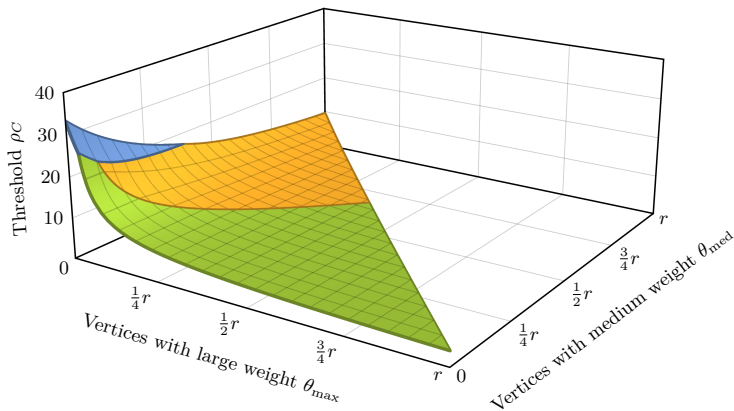

\subsection{Rank-1 random graph with an arbitrary number of weights}
\label{subsec:example_arbitrary_number_of_weights}
In this section we consider the setting where the graph contains several different vertex weights. In this case it is more difficult to characterize the subgraph $\optimalsubgraph[C]$ that maximizes \eqref{eq:optimal_subgraph_condition} for a given subgraph $C \subseteq V$, and finding this subgraph becomes optimization problem. This is because, for a given size $|D|$, we only need to consider the subgraph $D$ consisting of the $|D|$ largest weights in $C$. Using this insight we can approximate \eqref{eq:optimal_subgraph_condition}. Let $\hat{F}_C(x)$ be the empirical distribution function of the weights in $C$, then
\begin{align}
\label{eq:optimal_subgraph_condition_simplified}
\max_{D \subseteq C} \frac{\E_0[e(D)] \mspace{1mu} h\bigl(\scaling[C] - 1\bigr)}{|D| \mspace{1mu} \log(n / |D|)}
  &\approx \max_{k \in \{1,\ldots,r\}} \: \frac{\binom{k}{2} \bigl(\frac{r}{k} \int_{\frac{r-k}{r}}^1 \hat{F}_C^\invf(y) \mspace{1mu} dy\bigr)^{\!2} \mspace{1mu} h(\scaling[C] - 1)}{k \log\bigl(n / k\bigr)}\\
  &\approx \max_{\alpha \in (0, 1]} \: \frac{r}{2 \mspace{1mu} \alpha} \mspace{1mu} \frac{\bigl(\int_{1 - \alpha}^1 \hat{F}_C^\invf(y) \mspace{1mu} dy\bigr)^{\!2} \mspace{1mu} h(\scaling[C] - 1)}{\log(n)} \,,\notag
\end{align}
where $\hat{F}_C^\invf(y) = \inf\bigl\{x \in \R : y \leq \hat{F}_C(x)\bigr\}$ is the quantile function of $\hat{F}_C(x)$, and we have assumed that $r = n^{\smallO(1)}$ such that $\log(n / r) \asymp \log(n)$ in the second approximation above.

To apply \eqref{eq:optimal_subgraph_condition_simplified} we need to know $\hat{F}_C(x)$, which is different for every subgraph $C \subseteq V$. However, instead of characterizing the threshold scaling $\scaling[C]$ for every subgraph $C$, we can instead consider a uniformly chosen subgraph $C$. In this way, if the vertex weights are sampled from a distribution $W$ with distribution function $F(x)$, then we know from the Glivenko-Cantelli theorem that $\hat{F}_C(x)$ will eventually be close to $F(x)$, uniformly in $x$. With this in mind, we can consider the required threshold scaling $\scaling[C]$ when $C$ is a uniformly chosen subgraph and the vertex weights are sampled from a distribution $W$.

In Table~\ref{tbl:random_weights_example_analytic} this is done for a community of size $r = \lfloor \log(n)^4 \rfloor$ and weight distribution $W = (s + X) / \log(n)^{3/2}$, where we consider several different distributions $X$. We add a small constant $s$ to ensure that none of the vertex weights can become too small and we have normalized the weights by $\log(n)^{3/2}$ to ensure that in each example the maximum weight is less than $1$ with high probability. These choices ensure that Assumptions \ref{ass:lower_bound_maximum_community_small}, \ref{ass:lower_bound_sparsity}, and \ref{ass:unknown_probability_maximum_inhomogeneity} hold with high probability. Furthermore, we have that $r \mspace{1mu} \E[W]^2 / \log(n/r) = \bigO(1)$, and by \eqref{eq:optimal_subgraph_condition} this guarantees that $\scaling[C] = \bigO(1)$, so we obtain a numerical value for $\scaling[C]$ that is asymptotically independent of $n$.

Moreover, in Table~\ref{tbl:random_weights_example_analytic2} we consider the same examples as in Table~\ref{tbl:random_weights_example_analytic} but with a larger community size $r = \lfloor n^{1/4} \log(n)^4 \rfloor$. In this case Assumption~\ref{ass:lower_bound_maximum_community_small} does not hold because the community size $r$ is too large. However, we can now apply Assumption~\ref{ass:lower_bound_maximum_inhomogeneity} instead. To see this, note that Assumption~\ref{ass:lower_bound_maximum_inhomogeneity} \ref{ass:lbl:lower_bound_maximum_inhomogeneity_size} and \ref{ass:lbl:lower_bound_maximum_inhomogeneity_sparsity} hold with high probability provided $\delta < 1/4$. Furthermore,  Assumption~\ref{ass:lower_bound_maximum_inhomogeneity}~\ref{ass:lbl:lower_bound_maximum_inhomogeneity_inhomogeneity} also holds with high probability because the edge probabilities differ by at most a factor $\log(n)^2$ (i.e., $\pmax / \pmin = \bigO(\log(n)^2)$) with high probability.

This shows that Assumptions \ref{ass:lower_bound_maximum_inhomogeneity} and \ref{ass:lower_bound_maximum_community_small} are nicely complementing each other. For small communities (as in Table~\ref{tbl:random_weights_example_analytic}) our results can be applied because Assumptions \ref{ass:lower_bound_maximum_community_small}, \ref{ass:lower_bound_sparsity}, and \ref{ass:unknown_probability_maximum_inhomogeneity} hold with high probability, and for large communities (as in Table~\ref{tbl:random_weights_example_analytic2}) our results can still be applied because Assumptions \ref{ass:lower_bound_maximum_inhomogeneity}, \ref{ass:lower_bound_sparsity}, and \ref{ass:unknown_probability_maximum_inhomogeneity} hold with high probability.

\begin{table}[p]
\centering
\caption{\hyphenpenalty=10000%
The threshold scaling $\scaling[C]$ required to detect a planted community $C$ that is planted uniformly at random (based on setting the approximation in \eqref{eq:optimal_subgraph_condition_simplified} equal to $1$ and then solving for $\scaling[C]$). We provide the analytic results together with a numerical example where $\E[X] = 1$. The community size is $r = \lfloor \log(n)^4 \rfloor$.}
\label{tbl:random_weights_example_analytic}
\begin{tabularx}{0.92\textwidth}{p{14mm} p{20mm} X p{16mm}}
\arrayrulecolor{black}\toprule
\multicolumn{2}{l}{\hspace{15mm}$W$} & \multicolumn{1}{l}{Threshold $\scaling[C]$} & \multicolumn{1}{l}{$|\optimalsubgraph|$}\\
\arrayrulecolor{black}\midrule
\addlinespace[0.9ex]
$\frac{s + X}{\log(n)^{3/2}}$, & $\hspace{-0.7em}\scriptstyle{}X \sim \text{Degen}(\delta)$
  & $h^\invf\hspace{-0.1em}\left(\frac{2}{(s+1)^2}\right) \!+\! 1$
  & $r$\\
\addlinespace[0.9ex]
& $\hspace{-0.7em}\scriptstyle{}\delta=1,\; s=0.1$
  & $3.311$ & $1.000 \cdot r$\\
\addlinespace[0.6ex]
\arrayrulecolor{black!5!white}\midrule
\addlinespace[0.9ex]
$\frac{s + t X}{\log(n)^{3/2}}$, & $\hspace{-0.7em}\scriptstyle{}X \sim \text{Bern}(q)$
  & $h^\invf\hspace{-0.1em}\left(\frac{2}{q(s+t)^2} \!\medwedge\! \frac{2}{(s + q t)^2}\right) \!+\! 1$
  & $q \mspace{1mu} r \;\;\text{or}\;\; r$\\
\addlinespace[0.9ex]
& $\hspace{-0.7em}\scriptstyle{}q=0.5,\; t=2,\; s=0.1$
  & $2.624$ & $0.500 \cdot r$\\
\addlinespace[0.6ex]
\arrayrulecolor{black!5!white}\midrule
\addlinespace[0.9ex]
$\frac{s + X}{\log(n)^{3/2}}$, & $\hspace{-0.7em}\scriptstyle{}X \sim \text{Unif}(a, b)$
  & $h^\invf\hspace{-0.1em}\left(\frac{27}{4} \mspace{1mu} \frac{b - a}{(b + s)^3}\right) \!+\! 1$
  & $\frac{2}{3} \mspace{1mu} \frac{b+s}{b-a} \mspace{1mu} r$\\
\addlinespace[0.9ex]
& $\hspace{-0.7em}\scriptstyle{}a=0,\; b=2,\; s=0.1$
  & $3.144$ & $0.700 \cdot r$\\
\addlinespace[0.6ex]
\arrayrulecolor{black!5!white}\midrule
\addlinespace[0.9ex]
$\frac{s + X}{\log(n)^{3/2}}$, & $\hspace{-0.7em}\scriptstyle{}X \sim \text{Exp}(\lambda)$
  & $h^\invf\hspace{-0.1em}\left(\frac{\lambda^2}{2 \e^{s \lambda - 1}}\right) \!+\! 1$
  & $\e^{s \lambda - 1} \mspace{1mu} r$\\
\addlinespace[0.9ex]
& $\hspace{-0.7em}\scriptstyle{}\lambda=1,\; s=0.1$
  & $2.939$ & $0.407 \cdot r$\\
\addlinespace[0.6ex]
\arrayrulecolor{black}\bottomrule
\end{tabularx}
\vspace{40pt}
\centering
\caption{\hyphenpenalty=10000%
The threshold scaling $\scaling[C]$ required to detect a planted community $C$ that is planted uniformly at random (based on setting the approximation in \eqref{eq:optimal_subgraph_condition_simplified} equal to $1$ and then solving for $\scaling[C]$). We provide the analytic results for community size $r = \lfloor n^{1/4} \log(n)^4 \rfloor$. Note that, the threshold scaling $\scaling[C]$ is equal to $1 + \bigTheta(n^{-1/8})$ in these examples because $h(x) \asymp x^2 / 2$ as $x \to 0$.}
\label{tbl:random_weights_example_analytic2}
\begin{tabularx}{0.92\textwidth}{p{14mm} p{20mm} X p{16mm}}
\arrayrulecolor{black}\toprule
\multicolumn{2}{l}{\hspace{15mm}$W$} & \multicolumn{1}{l}{Threshold $\scaling[C]$} & \multicolumn{1}{l}{$|\optimalsubgraph|$}\\
\arrayrulecolor{black}\midrule
\addlinespace[0.9ex]
$\frac{s + X}{\log(n)^{3/2}}$, & $\hspace{-0.7em}\scriptstyle{}X \sim \text{Degen}(\delta)$
  & $h^\invf\hspace{-0.1em}\left(\frac{1}{n^{1/4}} \mspace{1mu} \frac{2}{(s+1)^2}\right) \!+\! 1$
  & $r$\\
\addlinespace[0.6ex]
\arrayrulecolor{black!5!white}\midrule
\addlinespace[0.9ex]
$\frac{s + t X}{\log(n)^{3/2}}$, & $\hspace{-0.7em}\scriptstyle{}X \sim \text{Bern}(q)$
  & $h^\invf\hspace{-0.1em}\left(\frac{1}{n^{1/4}} \left(\frac{2}{q(s+t)^2} \!\medwedge\! \frac{2}{(s + q t)^2}\right)\!\right) \!+\! 1$
  & $q \mspace{1mu} r \;\;\text{or}\;\; r$\\
\addlinespace[0.6ex]
\arrayrulecolor{black!5!white}\midrule
\addlinespace[0.9ex]
$\frac{s + X}{\log(n)^{3/2}}$, & $\hspace{-0.7em}\scriptstyle{}X \sim \text{Unif}(a, b)$
  & $h^\invf\hspace{-0.1em}\left(\frac{1}{n^{1/4}} \mspace{1mu} \frac{27}{4} \mspace{1mu} \frac{b - a}{(b + s)^3}\right) \!+\! 1$
  & $\frac{2}{3} \mspace{1mu} \frac{b+s}{b-a} \mspace{1mu} r$\\
\addlinespace[0.6ex]
\arrayrulecolor{black!5!white}\midrule
\addlinespace[0.9ex]
$\frac{s + X}{\log(n)^{3/2}}$, & $\hspace{-0.7em}\scriptstyle{}X \sim \text{Exp}(\lambda)$
  & $h^\invf\hspace{-0.1em}\left(\frac{1}{n^{1/4}} \mspace{1mu} \frac{\lambda^2}{2 \e^{s \lambda - 1}}\right) \!+\! 1$
  & $\e^{s \lambda - 1} \mspace{1mu} r$\\
\addlinespace[0.6ex]
\arrayrulecolor{black}\bottomrule
\end{tabularx}
\end{table}

\section{Discussion}
\label{sec:discussion}
In this section we remark on our results and discuss some possibilities for future work.

\vspace{-2pt}
\paragraph{Alternatives to the scan test.}
When the community size $|C| = r$ becomes much larger than allowed by Assumption \ref{ass:lower_bound_maximum_inhomogeneity} or \ref{ass:lower_bound_maximum_community_small}, that is $r \geq \sqrt{n}$, then the scan test is no longer optimal. This was considered by Arias-Castro and Verzelen for an Erd\H{o}s-R\'{e}nyi random graph \cite{Arias-Castro2014}, where they show that for large communities, a statistic based on simply counting the total number of edges is optimal. A similar idea can also be applied in the inhomogeneous settings. This suggests that such a test is asymptotically powerful, if for all $C \subseteq V$ of size $|C| = r$,
\begin{equation}
\frac{\E_C[e(C)] - \E_0[e(C)]}{\sqrt{\E_0[e(V)]}} \to \infty \,.
\end{equation}
Alternatively, when the communities become extremely large such that $r = \bigTheta(n)$ then our model becomes a version of the degree corrected stochastic block model \cite{Karrer2011}. In this case, it might be beneficial to consider tests based on spectral methods \cite{Massoulie2014,Gulikers2017,Mossel2018,Bordenave2018}.

Another setting where the scan test is no longer optimal is when the underlying graph is very sparse. In this case, one could consider tests similar to those considered by Arias-Castro and Verzelen \cite{Arias-Castro2015}.

\vspace{-2pt}
\paragraph{Unknown community size.}
When presenting our results, we have always assumed that the size of the planted community is known. In practice, this is often not the case and it would be necessary to estimate the community size before testing. In our case, the scan test can easily be extended to the setting of unknown community size. To see this, note that the scan test can detect any planted community provided that it is not larger than $r$. Hence, one can simply use the scan test with a large enough value for $r$ and it will detect a planted community of size at most $r$.

\vspace{-2pt}
\paragraph{Beyond the rank-1 case.}
In Section~\ref{subsec:scan_test_for_unknown_rank1_edge_probabilities} we consider unknown edge probabilities by additionally assuming a rank-1 structure. This can likely be generalized to edge probabilities that have different structural assumptions, provided Assumption~\ref{ass:unknown_probability_maximum_inhomogeneity} is suitably adjusted. The main difficulty in obtaining a result similar to Theorem~\ref{thm:unknown_probability_scan_test_powerful} would then be to find an estimator for $\E_0[e(C)]$ and show a consistency result similar to Lemma~\ref{lem:estimator_taylor_expansion}. Such a result will depend heavily on the precise structural assumptions made.

\vspace{-2pt}
\paragraph{Relaxation of Assumptions~\ref{ass:lower_bound_maximum_inhomogeneity}, \ref{ass:lower_bound_maximum_community_small}, and \ref{ass:lower_bound_sparsity}.}
All assumptions needed to prove the information theoretic lower bound in Section~\ref{subsec:information_theoretic_lower_bound} require that certain conditions hold for all sets $C \subseteq V$ of size $|C| = r$. This can be slightly relaxed because it is only necessary that these conditions hold for most sets $C \subseteq V$. Specifically, there needs to exists a class $\mathcal{C}$ such that the conditions in Assumptions~\ref{ass:lower_bound_maximum_inhomogeneity}, \ref{ass:lower_bound_maximum_community_small}, and \ref{ass:lower_bound_sparsity} hold for all $C \in \mathcal{C}$ and $\bar{\P}(C \in \mathcal{C}) \to 1$, where $\bar{\P}(\cdot)$ denotes probability with respect to a uniformly chosen set $C \subseteq V$ of size $|C| = r$. To see this, one only needs to modify the truncation event in \eqref{eq:truncation_event} to also include all sets $C \notin \mathcal{C}$. That is, one needs to modify the truncation event to $\truncation[C]' = \truncation[C] \cup \{C \notin \mathcal{C}\}$, where $\truncation[C]$ is the original truncation event from \eqref{eq:truncation_event}.

\vspace{-2pt}
\paragraph{Computational complexity.}
In general, the computational complexity of scan tests is not polynomial in the graph size $n$. In the homogeneous settings, it has been conjectured that polynomial time algorithms are not able to achieve the minimax rate~\cite{Hajek2015}. Inhomogeneity in the graphs can make computations easier -- for instance in very inhomogeneous cases it is possible to recover the largest clique of a graph in polynomial time \cite{Friedrich2015}. It thus remains an interesting avenue for future work to thoroughly characterize the statistical limits of tests under computational constraints.

\mathtoolsset{showonlyrefs} 
\section{Proofs}
\label{sec:proofs}
In this section we prove our results. We start with the proof of Theorem~\ref{thm:known_probability_scan_test_powerful} because it is the simplest and it sets the stage for some of the arguments in the proof of Theorem~\ref{thm:unknown_probability_scan_test_powerful}. We end this section with the proof of Theorem~\ref{thm:lower_bound}, which shows that the results obtained in Theorem~\ref{thm:known_probability_scan_test_powerful} and Theorem~\ref{thm:unknown_probability_scan_test_powerful} are, roughly speaking, the best possible.

\subsection{Proof of Theorem~\ref{thm:known_probability_scan_test_powerful}: Scan test for known edge probabilities is powerful}
\label{subsec:proof_of_scan_test_for_known_edge_probabilities_is_powerful}
\begin{proof}[\unskip\nopunct]
In this section we prove that the scan test in \eqref{eq:known_probability_scan_test_statistic} is asymptotically powerful. That is, under the conditions of the theorem, both type-I and type-II errors vanish.

\paragraph{Type-I error.} We will show that $\P_0(\statkp \geq 1 + \epsilon/2) \to 0$. This is done through a relatively straightforward use of Bennett's inequality and the union bound. Using $\binom{n}{k} \leq \smash{\left(\frac{n\,\e}{k}\right)^k}$, it follows that
\begin{align}
\hspace{40pt}&\hspace{-40pt}
\P_0\left(\statkp \geq 1 + \frac{\epsilon}{2}\right)
  = \P_0\left(\max_{D \subseteq V,\, |D| \leq r} \statkp[D] \geq 1 + \frac{\epsilon}{2}\right)\\
  &= \P_0\Biggl(\max_{1 \leq k \leq r} \max_{D \subseteq V,\, |D| = k} \frac{\E_0[e(D)] \, h\left(\bigl[e(D) / \E_0[e(D)] - 1\bigr]_{+}\right)}{k \log(n / k)} \geq 1 + \frac{\epsilon}{2}\Biggr)\\
  &\leq \sum_{1 \leq k \leq r} \sum_{D \subseteq V, |D| = k} \P_0\Biggl(\frac{\E_0[e(D)] \, h\left(\bigl[e(D) / \E_0[e(D)] - 1\bigr]_{+}\right)}{k \log(n / k)} \geq 1 + \frac{\epsilon}{2}\Biggr)\\
  &\leq \sum_{1 \leq k \leq r} \binom{n}{k} \exp\left(-\left(1+\frac{\epsilon}{2}\right) k \log\left(\frac{n}{k}\right)\right)\\
  &\leq \sum_{1 \leq k \leq r} \left(\e \left(\frac{k}{n}\right)^{\epsilon/2}\right)^k\\
  &\leq \frac{\e \left(\frac{r}{n}\right)^{\epsilon/2}}{1 - \e \left(\frac{r}{n}\right)^{\epsilon/2}}
  \to 0 \,.
\end{align}
The first and second inequality follow from a simple union bound and Bennett's inequality given in \eqref{eq:bennett_edge_bound_alternative}. The final step relies on the fact that $k / n \leq r / n$ and $r = \smallO(n)$. Therefore we conclude that the scan test \eqref{eq:known_probability_scan_test_statistic} has vanishing type-I error.

\paragraph{Type-II error.} Showing that we have vanishing type-II error starts by realizing that $\P_C(\statkp \geq 1 + \epsilon / 2) \geq \P_C(\statkp[{\optimalsubgraph[C]}] \geq 1 + \epsilon / 2)$, for every $C \subseteq V$ of size $|C| = r$, where $\optimalsubgraph[C]$ was introduced in Definition~\ref{def:optimal_subset_argmax}. The rest of the proof entails showing that for every $C \subseteq V$ of size $|C| = r$,
\begin{equation}
\label{eq:asymptotic_equivalence}
\statkp[{\optimalsubgraph[C]}] \geq (1 + \smallOp[\P_C](1)) \, \frac{\E_0[e(\optimalsubgraph[C])] \, h(\scaling[C] - 1)}{|\optimalsubgraph[C]| \, \log(n / |\optimalsubgraph[C]|)} \,.
\end{equation}
Together with \eqref{eq:known_probability_scan_test_powerful_condition} this implies that, for every $C$, we have $\P_C(\statkp[{\optimalsubgraph[C]}] \geq 1 + \epsilon / 2)\to 1$.

Let $C \subseteq V$ be an arbitrary subgraph of size $|C| = r$ and recall $\optimalsubgraph \coloneqq \optimalsubgraph[C]$ from Definition~\ref{def:optimal_subset_argmax} (we drop the explicit dependence of $\optimalsubgraph$ on $C$ to avoid notational clutter). To prove \eqref{eq:asymptotic_equivalence} it suffices to show that
\begin{equation}
\label{eq:scan_test_asymptotics}
\E_0[e(\optimalsubgraph)] \, h\!\left(\bigl[e(\optimalsubgraph) / \E_0[e(\optimalsubgraph)] - 1\bigr]_{+}\right)
  \geq (1 + \smallOp[\P_C](1)) \, \E_0[e(\optimalsubgraph)] \, h\left(\scaling[C] - 1\right) \,.
\end{equation}
To see this, note that $x \mapsto h(x-1)$ is convex, with derivative $h'(x-1) = \log(x)$ and therefore $h(x - 1) \geq h(y - 1) + (x - y) \log(y)$. Using this, together with $x = e(\optimalsubgraph) / \E_0[e(\optimalsubgraph)]$ and $y = \E_C[e(\optimalsubgraph)] / \E_0[e(\optimalsubgraph)] = \scaling[C] > 1$, we obtain the lower bound
\begin{align}
\label{eq:alternative_edge_count_taylor}
\hspace{40pt}&\hspace{-40pt}
\E_0[e(\optimalsubgraph)]\, h\!\left(\left[\frac{e(\optimalsubgraph)}{\E_0[e(\optimalsubgraph)]}-1\right]_{\!+}\right) - \E_0[e(\optimalsubgraph)]\, h\left(\scaling[C]-1\right)\\
  &= \E_0[e(\optimalsubgraph)]\, h\!\left(\left[\frac{e(\optimalsubgraph)}{\E_0[e(\optimalsubgraph)]}-1\right]_{\!+}\right) - \E_0[e(\optimalsubgraph)]\, h\!\left(\frac{\E_C[e(\optimalsubgraph)]}{\E_0[e(\optimalsubgraph)]}-1\right)\\
  &\geq \bigl(e(\optimalsubgraph) - \E_C[e(\optimalsubgraph)]\bigr) \log\left(\frac{\E_C[e(\optimalsubgraph)]}{\E_0[e(\optimalsubgraph)]}\right)\\
  &= \bigl(e(\optimalsubgraph) - \E_C[e(\optimalsubgraph)]\bigr) \log\left(\scaling[C]\right) \,.
\end{align}
It follows by Chebyshev's inequality that
\begin{equation}
\label{eq:alternative_edge_count_chebyshev}
\bigl(e(\optimalsubgraph) - \E_C[e(\optimalsubgraph)]\bigr) \log\left(\scaling[C]\right)
  = \bigO_{\P_C}\!\left(\sqrt{\E_C[e(\optimalsubgraph)]} \log(\scaling[C])\right) \,.
\end{equation}
Therefore, the inequality in \eqref{eq:scan_test_asymptotics} holds when
\begin{equation}
\label{eq:alternative_edge_count_deviations}
\frac{\sqrt{\E_C[e(\optimalsubgraph)]} \log(\scaling[C])}{\E_0[e(\optimalsubgraph)] h\left(\scaling[C]-1\right)} = \smallO(1) \,.
\end{equation}
To show this, we consider three cases depending on the asymptotic behavior of $\scaling[C]$. Although these three cases do not cover all possibilities, they suffice, by the argument in Remark~\ref{rem:subsubsequence} below.

\paragraph{\normalfont\emph{Case 1} ($\scaling[C] \to 1$):}
Using $\sqrt{x} \log(x) \asymp (x - 1)$ as $x \to 1$, and $h(x - 1) \asymp (x - 1)^2 / 2$ as $x \to 1$ gives
\begin{align}
\sqrt{\E_C[e(\optimalsubgraph)]} \log(\scaling[C])
  &= (1 + \smallO(1)) \, \sqrt{\E_0[e(\optimalsubgraph)]} (\scaling[C] - 1) \,,\\
\intertext{and}
\E_0[e(\optimalsubgraph)] h(\scaling[C] - 1)
  &= (1 + \smallO(1)) \, \E_0[e(\optimalsubgraph)] (\scaling[C] - 1)^2 / 2 \,.
\end{align}
Hence, by \eqref{eq:known_probability_scan_test_powerful_condition} we have
$$\E_0[e(\optimalsubgraph)] (\scaling[C] - 1)^2 \asymp 2 \E_0[e(\optimalsubgraph)] h(\scaling[C] - 1) > 2 |\optimalsubgraph| \log(n / |\optimalsubgraph|) \to \infty\, .$$
Combining the above gives
\begin{equation}
\label{eq:alternative_edge_count_deviations_case1}
\frac{\sqrt{\E_C[e(\optimalsubgraph)]} \log(\scaling[C])}{\E_0[e(\optimalsubgraph)] h(\scaling[C] - 1)}
  = (1 + \smallO(1)) \, \frac{2}{\sqrt{\E_0[e(\optimalsubgraph)] (\scaling[C] - 1)^2}}
  = \smallO(1) \,.
\end{equation}
This shows that \eqref{eq:alternative_edge_count_deviations} holds when $\scaling[C] \to 1$.

\paragraph{\normalfont\emph{Case 2} ($\scaling[C] \to \alpha \in (1, \infty)$):}
In this case $\sqrt{\scaling[C]} \log(\scaling[C]) = \bigO\left(h(\scaling[C] - 1)\right)$, and by \eqref{eq:known_probability_scan_test_powerful_condition} we have $\E_0[e(\optimalsubgraph)] h(\scaling[C] - 1) \geq |\optimalsubgraph| \log(n / |\optimalsubgraph|) \to \infty$. Therefore
\begin{equation}
\label{eq:alternative_edge_count_deviations_case2}
\sqrt{\E_C[e(\optimalsubgraph)]} \log(\scaling[C])
  = \sqrt{\E_0[e(\optimalsubgraph)]} \, \sqrt{\scaling[C]} \log(\scaling[C])
  = \smallO(\E_0[e(\optimalsubgraph)] h(\scaling[C] - 1)) \,.
\end{equation}
This shows that \eqref{eq:alternative_edge_count_deviations} holds when $\scaling[C] \to \alpha \in (1, \infty)$.

\paragraph{\normalfont\emph{Case 3} ($\scaling[C] \to \infty$):}
Using $h(x - 1) \asymp x \log(x)$ as $x \to \infty$ and because $\E_C[e(\optimalsubgraph)] \to \infty$ we have
\begin{equation}
\label{eq:alternative_edge_count_deviations_case3}
\frac{\sqrt{\E_C[e(\optimalsubgraph)]} \log(\scaling[C])}{\E_0[e(\optimalsubgraph)] h(\scaling[C] - 1)}
  = \frac{1}{\sqrt{\E_C[e(\optimalsubgraph)]}} = \smallO(1) \,.
\end{equation}
This shows that \eqref{eq:alternative_edge_count_deviations} holds when $\scaling[C] \to \infty$, and therefore that \eqref{eq:scan_test_asymptotics} holds in all the three cases.

\begin{remarkx}[General {$\scaling[C]$} sequences]
\label{rem:subsubsequence}
Note that $\scaling[C]$ might not fit one of the above cases, but may rather oscillate between a combination of the three. However, this is not a problem. For every subsequence of $\scaling[C]$, there exists a further subsequence along which the scaling $\scaling[C]$ satisfies one of the three cases. Hence, \eqref{eq:alternative_edge_count_deviations} holds along this (further) subsequence, which implies that \eqref{eq:alternative_edge_count_deviations} also holds along the full sequence. This type of argument will be used in several more places in the proofs.
\end{remarkx}

The proof of Theorem~\ref{thm:known_probability_scan_test_powerful} is now easily completed using \eqref{eq:asymptotic_equivalence} together with \eqref{eq:known_probability_scan_test_powerful_condition}. For every $C \subseteq V$ of size $|C| = r$,
\begin{align}
\statkp
  \geq \statkp[\optimalsubgraph]
  &= \frac{\E_0[e(\optimalsubgraph)] \, h\Bigl(\bigl[e(\optimalsubgraph) / \E_0[e(\optimalsubgraph)] - 1\bigr]_{+}\Bigr)}{|\optimalsubgraph| \, \log(n / |\optimalsubgraph|)}\\
  &\geq (1 + \smallOp[\P_C](1)) \, \frac{\E_0[e(\optimalsubgraph)] \, h\bigl([\scaling[C] - 1]_{+}\bigr)}{|\optimalsubgraph| \, \log(n / |\optimalsubgraph|)}\\
  &\geq (1 + \smallOp[\P_C](1)) (1 + \epsilon) \,.
\end{align}
Hence, $\P_C\left(\statkp \geq 1 + \epsilon/2\right) \to 1$. This shows that the type-II error vanishes, completing the proof.
\end{proof}

\subsection{Proof of Theorem~\ref{thm:unknown_probability_scan_test_powerful}: Scan test for unknown rank-1 edge probabilities is powerful}
\label{subsec:proof_of_scan_test_for_unknown_rank1_edge_probabilities_is_powerful}
In this section we prove that the scan test in \eqref{eq:unknown_probability_scan_test_statistic} is asymptotically powerful, but we first derive some auxiliary results. The first of these shows that if a planted community can be detected then it can be detected based on the evidence of the subgraph $\optimalsubgraph[C]$ from Definition~\ref{def:optimal_subset_argmax}. Moreover, by Assumption~\ref{ass:unknown_probability_maximum_inhomogeneity} it follows that $\optimalsubgraph[C]$ must be relatively large. Specifically, we show that $|\optimalsubgraph[C]| \geq r^{1/3}$. This explains why the scan test in \eqref{eq:unknown_probability_scan_test_statistic} is defined to only scan over subgraphs larger than $r^{1/3}$.
\begin{lemma}
\label{lem:optimal_subset_size}
For any $C \subseteq V$ of size $|C| = r$, let $\optimalsubgraph[C]$ be as given in Definition~\ref{def:optimal_subset_argmax}. When Assumption~\ref{ass:unknown_probability_maximum_inhomogeneity} holds then $|\optimalsubgraph[C]| \geq r^{1/3}$.
\end{lemma}
\begin{proof}
We use a proof by contradiction. For any $D \subseteq V$ of size $|D| \leq r^{1/3}$, it follows by Assumption~\ref{ass:unknown_probability_maximum_inhomogeneity} that
\begin{align}
\frac{\E_0[e(D)]}{|D| \log(n / |D|)}
  &\leq \frac{|D| - 1}{2} \, \frac{\wmax^2}{\log(n / |D|)}\\
  &\leq \frac{\smallO(r)}{2} \, \frac{\wmin^2}{\log(n / r^{1/3})}\\
  &< \frac{|C| - 1}{2} \, \frac{\wmin^2}{\log(n / |C|)}\\
  &\leq \frac{\E_0[e(C)]}{|C| \log(n / |C|)} \,.
\end{align}
Hence, a subset $D \subseteq V$ of size $|D| \leq r^{1/3}$ does not maximize the right-hand side of \eqref{eq:optimal_subset_argmax}, and therefore $|\optimalsubgraph[C]| \geq r^{1/3}$.
\end{proof}

In the second auxiliary result we quantify the deviations of $\est[D]$ around $\E_0[e(D)]$. We note that the lemma below remains true when all $(1 + \smallOp[\P_0](1))$ terms are replaced by $(1 + \smallOp[\P_C](1))$ terms. So, this results holds under both the null and alternative hypothesis. This crucial property is key to ensure that we can deal with unknown edge probabilities.
\begin{lemma}
\label{lem:estimator_taylor_expansion}
Let $\mathcal{D}$ be a set of subsets of the vertices $V$, such that $r^{1/3}\leq |D|\leq r$ for all $D \in \mathcal{D}$. Under Assumption~\ref{ass:unknown_probability_maximum_inhomogeneity} and
\begin{align}
e(V)
  &= (1 + \smallOp[\P_0](1))\E_0[e(V)] \,,\\
e(D,-D)
  &= (1 + \smallOp[\P_0](1))\E_0[e(D,-D)] \,,
  &\qquad \text{uniformly over all } D \in \mathcal{D} \,.
\intertext{the deviations of $\est[D]$ around $\E_0[e(D)]$ satisfy}
\label{eq:estimator_taylor_expansion}
\frac{\est[D]}{\E_0[e(D)]}
  &= 1 + \smallOp[\P_0](1) \,,
  &\qquad \text{uniformly over all } D \in \mathcal{D} \,.
\end{align}
Additionally, the statement above remains true when all $(1 + \smallOp[\P_0](1))$ terms are replaced by $(1 + \smallOp[\P_C](1))$ terms.
%
\end{lemma}
\begin{proof}
Define $f(x_1,x_2) \coloneqq \left(\sqrt{x_1}-\sqrt{x_1-2x_2}\right)^2$ for $x_1 \geq 2 x_2$. Then the partial derivatives of $f(x_1, x_2)$ are given by
\begin{align}
\label{eq:f_gradient}
\frac{\partial f}{\partial x_1}(x_1, x_2)
  &= -\frac{\left(\sqrt{x_1}-\sqrt{x_1-2x_2}\right)^2}{\sqrt{x_1\vphantom{2}}\sqrt{x_1-2x_2}}
  = - \frac{f(x_1, x_2)}{\sqrt{x_1\vphantom{2}}\sqrt{x_1-2x_2}} \,,\\
\frac{\partial f}{\partial x_2}(x_1, x_2)
  &= 2\,\frac{\sqrt{x_1}-\sqrt{x_1-2x_2}}{\sqrt{x_1-2x_2}}
  = \frac{2 \, f(x_1, x_2)}{\sqrt{x_1-2x_2} \left(\sqrt{x_1}-\sqrt{x_1-2x_2}\right)} \,.
\end{align}
We use a Taylor expansion of $f(x_1, x_2)$ around $(a_1, a_2)$ with $a_1>2a_2$. Specifically, there exists $(\xi_1, \xi_2)$ with $\xi_1$ in between $x_1$ and $a_1$, and $\xi_2$ in between $x_2$ and $a_2$, such that
\begin{equation}
\label{eq:f_taylor}
f(x_1, x_2)
  = f(a_1, a_2) + \frac{\partial f}{\partial x_1}(\xi_1, \xi_2) \, (x_1 - a_1)
    + \frac{\partial f}{\partial x_2}(\xi_1, \xi_2) \, (x_2 - a_2) \,.
\end{equation}
We use \eqref{eq:f_taylor} with $(x_1, x_2) = (e(V), e(D,-D))$ and $(a_1, a_2) = (\E_0[e(V)], \E_0[e(D,-D)])$. Because $e(V) = (1 + \smallOp[\P_0](1))\E_0[e(V)]$ by assumption, it follows that for any $\xi_1$ between $e(V)$ and $\E_0[e(V)]$ we have $\xi_1 = (1 + \smallOp[\P_0](1))\E_0[e(V)]$. Similarly, by assumption we have $e(D,-D) = (1 + \smallOp[\P_0](1))\E_0[e(D,-D)]$ uniformly over all $D \in \mathcal{D}$, and therefore it follows that $\xi_2 = (1 + \smallOp[\P_0](1))\E_0[e(D,-D)]$. Hence,
\begin{align}
\label{eq:f_taylor_applied}
\hspace{10pt}&\hspace{-10pt}
\frac{f(e(V), e(D,-D))}{f(\E_0[e(V)], \E_0[e(D,-D)])}\\
  &= \begin{multlined}[t][.9\displaywidth]
    1 - \frac{(1 + \smallOp[\P_0](1)) \, (e(V) - \E_0[e(V)])}{\sqrt{\E_0[e(V)]}\sqrt{\E_0[e(V)] - 2\E_0[e(D,-D)]}}\\
    {} + \frac{(2 + \smallOp[\P_0](1)) \, (e(D,-D) - \E_0[e(D,-D)])}{\sqrt{\E_0[e(V)] - 2\E_0[e(D,-D)]} \bigl(\sqrt{\E_0[e(V)]}-\sqrt{\E_0[e(V)] - 2\E_0[e(D,-D)]}\bigr)}
  \end{multlined}\\[0.5ex]
  &= 1 - (1 + \smallOp[\P_0](1)) \, \frac{e(V) - \E_0[e(V)]}{\E_0[e(V)]} + (2 + \smallOp[\P_0](1)) \, \frac{e(D,-D) - \E_0[e(D,-D)]}{\E_0[e(D,-D)]}\\
  &= 1 + \smallOp[\P_0](1) \,,
\end{align}
where we have used $\E_0[e(D,-D)] = \smallO(\E_0[e(V)])$ and $\E_0[e(V)] \to \infty$ in the second equality above, which is ensured by Assumption~\ref{ass:unknown_probability_maximum_inhomogeneity}. To see this, note that $(\wmax / \wmin)^2 \leq \smallO\bigl(\frac{n}{r} \wmin^2\bigr) \leq \smallO\bigl(\frac{n}{r}\bigr)$ because $\wmin^2 \leq 1$, hence
\begin{equation}
\frac{\E_0[e(D,-D)]}{\E_0[e(V)]}
  \leq (1 + \smallO(1)) \frac{|D| \mspace{1mu} n \mspace{1mu} \wmax^2}{n^2 \mspace{1mu} \wmin^2}
  \leq \frac{|D|}{n} \smallO\left(\frac{n}{r}\right)
  = \smallO\left(\frac{|D|}{r}\right)
  = \smallO(1) \,.
\end{equation}
To continue, we will show that
\begin{align}
\label{eq:f_estimator}
f\left(e(V), e(D,-D)\right)
  &= 4 \, \est[D] \,,\\
\label{eq:f_expectation}
f\left(\E_0[e(V)],\E_0[e(D,-D)]\right)
  &= (1 + \smallO(1)) \, \biggl(4 \,\E_0[e(D)] + 2 \sum_{i \in D} \w_i^2\biggr)\\*
  &= (1 + \smallO(1)) \, 4 \,\E_0[e(D)] \,.
\end{align}
Here \eqref{eq:f_estimator} follows directly from the definition in \eqref{eq:edge_count_estimator}. To obtain the first equality in \eqref{eq:f_expectation} we use Assumption~\ref{ass:unknown_probability_maximum_inhomogeneity} to ensure that $\E_0[e(V)] + \frac{1}{2} \sum_{i \in V} \w_i^2 = (1 + \smallO(1))\E_0[e(V)]$. This is easily shown since
\begin{equation}
\frac{\E_0[e(V)] + \sum_{i \in V} \w_i^2}{\E_0[e(V)]}
  \leq 1 + \frac{n \mspace{1mu} \wmax^2}{\binom{n}{2} \wmin^2}
  \leq 1 + \frac{n \mspace{1mu} r^{2/3} \mspace{1mu} \wmin^2}{\binom{n}{2} \wmin^2}
  = 1 + 2 \, \frac{r^{2/3}}{n-1}
  = 1 + \smallO(1) \,.
\end{equation}
For the second equality in \eqref{eq:f_expectation} we need to show $\sum_{i \in D} \w_i^2 / \E_0[e(D)] = \smallO(1)$. To this end, we first show
\begin{equation}
\label{eq:maximum_inhomogeneity_per_subset}
\frac{\sum_{i \in D} \w_i^2}{\left(\sum_{i \in D} \w_i\right)^2}
  \leq \frac{1}{4 |D|} \frac{(\wmin + \wmax)^2}{\wmin \, \wmax} \,.
\end{equation}
To see this, note that the ratio $\sum_{i \in D} \w_i^2 \big/ \bigl(\sum_{i \in D} \w_i\bigr)^2$ is maximized when a fraction $\alpha = \wmin / (\wmin + \wmax)$ of the vertices in $D$ has weight $\wmax$ and the remaining $1 - \alpha$ fraction of vertices has weight $\wmin$. Plugging this in we obtain \eqref{eq:maximum_inhomogeneity_per_subset}. Then, by Assumption~\ref{ass:unknown_probability_maximum_inhomogeneity} it follows that $\wmax / \wmin = \smallO(r^{1/3})$ and using that $|D| \geq r^{1/3}$ together with \eqref{eq:maximum_inhomogeneity_per_subset}, we obtain
\begin{equation}
\frac{\sum_{i \in D} \w_i^2}{\left(\sum_{i \in D} \w_i\right)^2}
  \leq \frac{1}{4 |D|} \frac{(\wmin + \wmax)^2}{\wmin \, \wmax}
  \leq \frac{1}{|D|} \frac{\wmax}{\wmin}
  = \frac{\smallO(r^{1/3})}{r^{1/3}}
  = \smallO(1) \,.
\end{equation}
Hence, plugging \eqref{eq:f_estimator} and \eqref{eq:f_expectation} into \eqref{eq:f_taylor_applied} gives
\begin{equation}
\frac{\est[D]}{\E_0[e(D)]}
  = (1 + \smallO(1)) \, \frac{f(e(V), e(D,-D))}{f(\E_0[e(V)], \E_0[e(D,-D)])}
  = (1 + \smallOp[\P_0](1)) \,.
\end{equation}
Finally, it can easily be checked, using the same steps as above, that the lemma remains true under the alternative hypothesis (i.e., when all $(1 + \smallOp[\P_0](1))$ terms are replaced by $(1 + \smallOp[\P_C](1))$ terms).
\end{proof}

\begin{proof}[\unskip\nopunct]
We are now ready to prove Theorem~\ref{thm:unknown_probability_scan_test_powerful}, which shows that the scan test in \eqref{eq:unknown_probability_scan_test_statistic} is still asymptotically powerful even when the edge probabilities are not known. To this end, we again show that both the type-I and the type-II error vanish, which we do separately below.

\paragraph{Type-I error.}
Here we show $\P_0(\statup \geq 1 + \epsilon/3) \to 0$. To this end, we show that using the truncated estimator $\esttr[D]$ from \eqref{eq:edge_count_estimator_threshold} is asymptotically as good as using $\E_0[e(D)]$. Specifically, we show that uniformly over all subgraphs $D \subseteq V$ of size $r^{1/3} \leq |D| \leq r$,
\begin{equation}
\label{eq:truncated_estimator_uniform_lower_bound}
\max_{D \subseteq V,\,r^{1/3} \leq |D| \leq r} \, \frac{\E_0[e(D)]}{\esttr[D]} \leq 1 + \smallOp[\P_0](1) \,.
\end{equation}
To show this, define the random set $\uniformboundset \coloneqq \{D \subseteq V : r^{1/3} \leq |D| \leq r,\, \esttr[D] \leq \E_0[e(D)]\}$ and rewrite \eqref{eq:truncated_estimator_uniform_lower_bound} as
\begin{equation}
\label{eq:truncated_estimator_uniform_lower_bound_alternative}
  \max_{D \subseteq V,\,r^{1/3} \leq |D| \leq r} \, \left( \frac{\E_0[e(D)]}{\esttr[D]} \1{\{D \in \uniformboundset\}} + \frac{\E_0[e(D)]}{\esttr[D]} \1{\{D \notin \uniformboundset\}} \right)\,.
\end{equation}
In the second term above we have $D \notin \uniformboundset$, so this term is trivially less than or equal to $1$. Therefore we will focus on the first term in \eqref{eq:truncated_estimator_uniform_lower_bound_alternative}. For any $D \in \uniformboundset$, it follows by definition of the thresholded estimator $\esttr[D]$ in \eqref{eq:edge_count_estimator_threshold} that
\begin{equation}
\frac{|D|^2}{n} \log^4\!\left(\frac{n}{|D|}\right) \leq \esttr[D] \leq \E_0[e(D)] \leq \biggl(\sum_{i \in D} \w_i\biggr)^{\!2} ,
\quad\text{hence}\quad
\frac{|D|}{\sqrt{n}} \log^2\!\left(\frac{n}{|D|}\right) \leq \sum_{i \in D} \w_i \,.
\end{equation}
Now, by the second part of Assumption~\ref{ass:unknown_probability_maximum_inhomogeneity} we have $1 \leq \smash{\bigl(\frac{\wmax}{\wmin}\bigr)^2} \leq \frac{n}{r} \mspace{1mu} \wmin^2$, and therefore $\wmin \geq \sqrt{r / n} \geq 1 / \sqrt{n}$. Using this we obtain
\begin{align}
\label{eq:expected_outside_edges}
\E_0[e(D,-D)]
  &= \biggl(\sum_{i \in D} \w_i\biggr) \biggl(\sum_{j \notin D} \w_j\biggr)
  \geq \frac{|D|}{\sqrt{n}} \log^2\!\left(\frac{n}{|D|}\right) \, \frac{n - |D|}{\sqrt{n}}\\
  &= (1 + \smallO(1)) |D| \log^2\!\left(\frac{n}{|D|}\right) \,.
\end{align}
Recall that Bennett's inequality ensures that, for $t > 0$,
\begin{equation}
\P_0(e(D,-D) - \E_0[e(D,-D)] \leq -t)
  \;\leq\; \exp\left(-\E_0[e(D,-D)] \, h\left(\frac{t}{\E_0[e(D,-D)]}\right)\right) \,.
\end{equation}
To get a uniform bound over all subgraphs $D \in \uniformboundset$, we use a union bound together with \eqref{eq:expected_outside_edges}. For any $\delta > 0$ and $n$ large enough, this gives
\begin{align}
\hspace{15pt}&\hspace{-15pt}
\P\biggl(\min_{D \in \uniformboundset} e(D,-D) - \E_0[e(D,-D)] \leq - (1 + \delta) \sqrt{2 \, \E_0[e(D,-D)] \, |D| \log(n / |D|)}\biggr)\\
  &\leq \begin{multlined}[t][0.9\displaywidth]
  \sum_{1 \leq k \leq r} \sum_{D \subseteq V,\, |D| = k} \1{\bigl\{\E_0[e(D,-D)] \geq (1 - \delta) |D| \log^2(n / |D|)\bigr\}}\notag\\*
    {} \times \P\biggl(e(D,-D) - \E_0[e(D,-D)] \leq - (1 + \delta) \sqrt{2 \, \E_0[e(D,-D)] \, |D| \log(n / |D|)}\biggr)\notag
  \end{multlined}\\[0.5ex]
  &\leq \begin{multlined}[t][0.9\displaywidth]
  \sum_{1 \leq k \leq r} \sum_{D \subseteq V,\, |D| = k} \1{\bigl\{\E_0[e(D,-D)] \geq (1 - \delta) |D| \log^2(n / |D|)\bigr\}}\notag\\*[-1ex]
    {} \times \exp\Biggl(-\E_0[e(D,-D)] \, h\Biggl((1 + \delta) \sqrt{\frac{2 \, |D| \log(n / |D|)}{\E_0[e(D,-D)]}}\Biggr)\Biggr)\notag
  \end{multlined}\\[0.5ex]
  &\leq \sum_{1 \leq k \leq r} \binom{n}{k} \exp\left(-(1+\delta) k \log\left(\frac{n}{k}\right)\right)\label{eq:step}\\
  &\leq \sum_{1 \leq k \leq r} \left(\e \left(\frac{k}{n}\right)^{\delta}\right)^k
  \leq \frac{\e \left(\frac{r}{n}\right)^{\delta}}{1 - \e \left(\frac{r}{n}\right)^{\delta}}
  \to 0 \,,
\end{align}
For the step in \eqref{eq:step} we have used the result in \eqref{eq:expected_outside_edges} together with $h(x) \asymp x^2 / 2$ as $x \to 0$, and the final step relies on the fact that $k / n \leq r / n$ and $r = \smallO(n)$.

Then, using the above together with \eqref{eq:expected_outside_edges}, it follows that uniformly over $D \in \uniformboundset$,
\begin{equation}
\label{eq:between_degree_deviations}
\frac{e(D,-D) - \E_0[e(D,-D)]}{\E_0[e(D,-D)]}
  = \bigOp[\P_0]\!\left(\sqrt{\frac{|D| \log(n / |D|)}{\E_0[e(D,-D)]}}\right)
  = \smallOp[\P_0](1) \,.
\end{equation}
To bound the deviations of $e(V)$ we use Chebyshev's inequality,
\begin{equation}
\label{eq:total_degree_deviations}
\frac{e(V) - \E_0[e(V)]}{\E_0[e(V)]}
  = \bigOp[\P_0]\!\left(\sqrt{\frac{1}{\E_0[e(V)]}}\right)
  = \smallOp[\P_0](1) \,.
\end{equation}
Using \eqref{eq:between_degree_deviations} and \eqref{eq:total_degree_deviations}, it follows by Lemma~\ref{lem:estimator_taylor_expansion} that uniformly over $D \in \uniformboundset$,
\begin{equation}
\label{eq:taylor_approximation_estimator}
\frac{\esttr[D]}{\E_0[e(D)]}
  \geq \frac{\est[D]}{\E_0[e(D)]}
  = 1 + \smallOp[\P_0](1) \,.
\end{equation}
This shows that the first term in \eqref{eq:truncated_estimator_uniform_lower_bound_alternative} is less than or equal to $1 + \smallOp[\P_0](1)$, and therefore that \eqref{eq:truncated_estimator_uniform_lower_bound} holds.

Then, using \eqref{eq:truncated_estimator_uniform_lower_bound} it becomes relatively straightforward to show that the type-I error vanishes. Indeed, note that $a \,h\bigl(\bigl[\frac{x}{a}-1\bigr]_{+}\bigr) \leq b\,h\bigl(\bigl[\frac{x}{b} - 1\bigr]_{+}\bigr)$ for $a > b$, and therefore
\begin{align}
\hspace{10pt}&\hspace{-10pt}
\P_{0}\left(\statup \geq 1 + \frac{\epsilon}{3}\right)
  = \P_{0}\left(\max_{\substack{D \subseteq V,\\ r^{1/3} \leq |D| \leq r}} 
    \frac{\esttr[D] \, h\left(\left[e(D) / \esttr[D] - 1\right]_{\!+}\right)}{|D| \log\left(n / |D|\right)} \geq 1 + \frac{\epsilon}{3}\right)\\[1ex]
  &\leq \P_{0}\left(\max_{\substack{D \subseteq V,\\ r^{1/3} \leq |D| \leq r}} 
      \frac{(1 + \smallOp[\P_0](1)) \E_0[e(D)] \, h\left(\left[(1 + \smallOp[\P_0](1)) \frac{e(D)}{\E_0[e(D)]} - 1\right]_{\!+}\right)}{|D| \log\left(n / |D|\right)} \geq 1 + \frac{\epsilon}{3}\right) \,.\notag
\end{align}
Then using the same reasoning as in the proof of Theorem~\ref{thm:known_probability_scan_test_powerful}, it follows that the type-I error vanishes.

\paragraph{Type-II error.} 
Here we show that $\P_C(\statup \geq 1 + \epsilon / 3) \geq \P_C(\statup[\optimalsubgraph] \geq 1 + \epsilon / 3) \to 1$, for every $C \subseteq V$ of size $|C| = r$, where $\optimalsubgraph = \optimalsubgraph[C]$ is defined as in \eqref{eq:optimal_subset_argmax}. To this end, we start by quantifying the deviation of $\est[\optimalsubgraph] / \E_0[e(\optimalsubgraph)]$ under the alternative.

By Chebyshev's inequality,
\begin{align}
\label{eq:chebyshev_alternative1}
\frac{e(\optimalsubgraph,-\optimalsubgraph) - \E_C[e(\optimalsubgraph,-\optimalsubgraph)]}{\E_C[e(\optimalsubgraph,-\optimalsubgraph)]}
  &= \bigOp[\P_C]\!\left(\sqrt{\frac{1}{\E_C[e(\optimalsubgraph,-\optimalsubgraph)]}}\right) 
  = \smallOp[\P_C](1) \,,\\
\label{eq:chebyshev_alternative2}
\frac{e(V) - \E_C[e(V)]}{\E_C[e(V)]}
  &= \bigOp[\P_C]\!\left(\sqrt{\frac{1}{\E_C[e(V)]}}\right)
  = \smallOp[\P_C](1) \,.
\end{align}
Moreover, $\scaling[C] \wmin^2 \leq 1$ and therefore $\frac{\wmax^2}{\wmin^2} = \smallO(\frac{n}{r} \wmin^2) \leq \smallO(\frac{n}{r} \frac{1}{\scaling[C]})$ by Assumption~\ref{ass:unknown_probability_maximum_inhomogeneity}. Hence $\scaling[C] \leq \smallO\left(\frac{n}{r} \frac{\wmin^2}{\wmax^2}\right)$. Therefore
\begin{align}
1 \leq \frac{\E_C[e(\optimalsubgraph,-\optimalsubgraph)]}{\E_0[e(\optimalsubgraph,-\optimalsubgraph)]}
  &\leq 1 + \frac{\E_C[e(\optimalsubgraph,C \setminus \optimalsubgraph)]}{\E_0[e(\optimalsubgraph,V \setminus \optimalsubgraph))]}
  = 1 + \scaling[C] \, \frac{\E_0[e(\optimalsubgraph,C \setminus \optimalsubgraph)]}{\E_0[e(\optimalsubgraph,V \setminus \optimalsubgraph))]}\\
  &\leq 1 + \scaling[C] \, \frac{|\optimalsubgraph| (|C| - |\optimalsubgraph|)}{|\optimalsubgraph| (|V| - |\optimalsubgraph|)} \frac{\wmax^2}{\wmin^2}
  \leq 1 + \scaling[C] \, \frac{r}{n} \frac{\wmax^2}{\wmin^2}
  \leq 1 + \smallO(1) \,.
\end{align}
By the above it follows that $\E_C[e(\optimalsubgraph,-\optimalsubgraph)] = (1 + \smallO(1))\E_0[e(\optimalsubgraph,-\optimalsubgraph)]$, and similarly $\E_C[e(V)] = (1 + \smallO(1))\E_0[e(V)]$. Therefore
\begin{equation}
e(\optimalsubgraph,-\optimalsubgraph) = (1 + \smallOp[\P_C](1)) \E_0[e(\optimalsubgraph,-\optimalsubgraph)] \,,
\quad\text{and}\quad
e(V) = (1 + \smallOp[\P_C](1)) \E_0[e(V)] \,.
\end{equation}
Then, applying Lemma~\ref{lem:estimator_taylor_expansion} (using the set $\mathcal{D} = \{\optimalsubgraph\}$), we obtain
\begin{align}
\frac{\est[\optimalsubgraph]}{\E_0[e(\optimalsubgraph)]}
  &= 1 + \smallOp[\P_C](1) \,.
\end{align}
Therefore by definition of the thresholded estimator in \eqref{eq:edge_count_estimator_threshold},
\begin{align}
\label{eq:estimator_approximation_truncated}
\esttr[\optimalsubgraph]
  &= \left(\est[\optimalsubgraph] \medvee \frac{|\optimalsubgraph|^2}{n} \log^4\left(\frac{n}{|\optimalsubgraph|}\right)\right)\\
  &= \left((1 + \smallOp[\P_C](1))\E_0[e(\optimalsubgraph)] \medvee \frac{|\optimalsubgraph|^2}{n} \log^4\left(\frac{n}{|\optimalsubgraph|}\right)\right) \,.
\end{align}
We continue by considering the two cases in the maximum of \eqref{eq:estimator_approximation_truncated} separately.

\paragraph{\normalfont\emph{Case 1}:}
Here we have $\esttr[\optimalsubgraph] = (1 + \smallOp[\P_C](1))\E_0[e(\optimalsubgraph)]$. Plugging this into the definition of the test statistic we obtain
\begin{align}
\statup[\optimalsubgraph]\mspace{-1mu}
  &= \frac{\esttr[\optimalsubgraph] h\Bigl(\Bigl[\frac{e(\optimalsubgraph)}{\esttr[\optimalsubgraph]} - 1\Bigr]_{\!+}\Bigr)}{|\optimalsubgraph| \log(n / |\optimalsubgraph|)}\\
  &= \frac{(1 \mspace{-2mu} + \mspace{-2mu} \smallOp[\P_C](1))\E_0[e(\optimalsubgraph)] h\Bigl(\Bigl[(1 \mspace{-2mu} + \mspace{-2mu} \smallOp[\P_C](1))\frac{e(\optimalsubgraph)}{\E_0[e(\optimalsubgraph)]} - 1\Bigr]_{\!+}\Bigr)}{|\optimalsubgraph| \log(n / |\optimalsubgraph|)} \,.
\end{align}
The proof can then be completed by using the same reasoning as in the proof of Theorem~\ref{thm:known_probability_scan_test_powerful} from \eqref{eq:asymptotic_equivalence} to \eqref{eq:alternative_edge_count_deviations_case3}. Here the additional $\smallOp[\P_C](1)$ terms do not make any difference.

\paragraph{\normalfont\emph{Case 2}:}
Here we have $\esttr[\optimalsubgraph] = (|\optimalsubgraph|^2 / n)\log^4\left(n / |\optimalsubgraph|\right)$. This corresponds to the case where the underlying graph is very sparse, and therefore a very large signal $\scaling[C]$ is required to detect a planted community.

We start by deriving a lower bound on $\scaling[C]$. Using condition \eqref{eq:unknown_probability_scan_test_powerful_condition} and the fact that $\E_0[e(\optimalsubgraph)] \leq (|\optimalsubgraph|^2 / n) \log^4\left(n / |\optimalsubgraph|\right)$ and $h^\invf(x) \geq \sqrt{x}$, we obtain
\begin{equation}
\label{eq:minimum_scaling_for_truncated_estimator}
\scaling[C]
  \geq h^\invf\!\left(\frac{|\optimalsubgraph| \log(n / |\optimalsubgraph|)}{\E_0[e(\optimalsubgraph)]}\right)
  \geq h^\invf\!\left(\frac{n}{|\optimalsubgraph|} \frac{1}{\log^3(n / |\optimalsubgraph|)}\right)
  \geq (1 + \smallO(1)) \sqrt{n / |\optimalsubgraph|} \,.
\end{equation}
Moreover, by the second part of Assumption~\ref{ass:unknown_probability_maximum_inhomogeneity} we have $1 \leq \smash{\bigl(\frac{\wmax}{\wmin}\bigr)^2} \leq \frac{n}{r} \mspace{1mu} \wmin^2$, and therefore $\wmin \geq \sqrt{r / n} \geq 1 / \sqrt{n}$. Using this together with \eqref{eq:minimum_scaling_for_truncated_estimator} gives
\begin{equation}
\frac{\E_C[e(\optimalsubgraph)]}{\esttr[\optimalsubgraph]}
  \geq \frac{\scaling[C] |\optimalsubgraph|^2 \wmin^2}{\frac{|\optimalsubgraph|^2}{n} \log^4\left(\frac{n}{|\optimalsubgraph|}\right)}
  \geq \frac{\scaling[C]}{\log^4\left(\frac{n}{|\optimalsubgraph|}\right)}
  \to \infty \,.
\end{equation}
Then, using that $e(\optimalsubgraph) = (1 + \smallOp[\P_C](1))\E_C[e(\optimalsubgraph)]$ by Chebyshev's inequality and $h(x - 1) \asymp x \log(x)$ as $x \to \infty$, we obtain
\begin{align}
\label{eq:thresholded_statistic_bound}
\esttr[\optimalsubgraph] h\!\left(\frac{e(\optimalsubgraph)}{\esttr[\optimalsubgraph]} - 1\right)
  &\geq \esttr[\optimalsubgraph] h\!\left((1 + \smallOp[\P_C](1))\frac{\E_C[e(\optimalsubgraph)]}{\esttr[\optimalsubgraph]} - 1\right)\\
  &= (1 + \smallOp[\P_C](1)) \E_C[e(\optimalsubgraph)] \log\!\left(\frac{\E_C[e(\optimalsubgraph)]}{\esttr[\optimalsubgraph]}\right)\\
  &= (1 + \smallOp[\P_C](1)) \E_C[e(\optimalsubgraph)] \log\!\left(\scaling[C] \, \frac{\E_0[e(\optimalsubgraph)]}{\esttr[\optimalsubgraph]}\right) \,.
\end{align}
Now, by the same argument as above we have $\wmin \geq 1 / \sqrt{n}$. Hence, it follows that $\E_0[e(\optimalsubgraph)] \geq |\optimalsubgraph|^2 \mspace{1mu} \wmin^2 \geq |\optimalsubgraph|^2 / n$, so that $\E_0[e(\optimalsubgraph)] / \esttr[\optimalsubgraph] \geq \log^{-4}\left(n / |\optimalsubgraph|\right)$. Then, using \eqref{eq:minimum_scaling_for_truncated_estimator},
\begin{align}
\frac{\log(\scaling[C] \, \E_0[e(\optimalsubgraph)] / \esttr[\optimalsubgraph])}{\log(\scaling[C])}
  &\geq \frac{\log(\scaling[C] / \log^4(n / |\optimalsubgraph|))}{\log(\scaling[C])}\\
  &= 1 - 4 \frac{\log\log(n / |\optimalsubgraph|)}{\log(\scaling[C])}
  = 1 + \smallO(1) \,.
\end{align}
Plugging this into \eqref{eq:thresholded_statistic_bound}, we obtain
\begin{align}
\esttr[\optimalsubgraph] h\!\left(\frac{e(\optimalsubgraph)}{\esttr[\optimalsubgraph]} - 1\right)
  &= (1 + \smallOp[\P_C](1)) \E_C[e(\optimalsubgraph)] \log\!\left(\scaling[C] \, \frac{\E_0[e(\optimalsubgraph)]}{\esttr[\optimalsubgraph]}\right)\\
  &\geq \vphantom{\frac{1}{2}} (1 + \smallOp[\P_C](1)) \E_C[e(\optimalsubgraph)] \log\left(\scaling[C]\right)\\
  &= \vphantom{\frac{1}{2}} (1 + \smallOp[\P_C](1)) \E_0[e(\optimalsubgraph)] h(\scaling[C] - 1)\\
  &\geq (1 + \smallOp[\P_C](1)) (1 + \epsilon) |\optimalsubgraph| \log\!\left(\frac{n}{|\optimalsubgraph|}\right) \,,
\end{align}
where the final step follows from \eqref{eq:unknown_probability_scan_test_powerful_condition}. Therefore, $\statup[\optimalsubgraph] \geq 1 + \epsilon / 3$ with high probability, completing the proof.
\end{proof}

\subsection{Proof of Corollaries \ref{cor:known_probability_scan_test_powerful} and \ref{cor:unknown_probability_scan_test_powerful}}
\label{subsec:proof_of_corollaries}
\begin{proof}[\unskip\nopunct]
To prove Corollaries \ref{cor:known_probability_scan_test_powerful} and \ref{cor:unknown_probability_scan_test_powerful} we need to show that either Assumption~\ref{ass:lower_bound_maximum_inhomogeneity} or \ref{ass:lower_bound_maximum_community_small} is sufficient to ensure that $\E_C[e(D^{\star})] \to \infty$ for every $C \subseteq V$ of size $|C| = r$. When $\scaling[C] = \bigO(1)$ this is a direct consequence of conditions \eqref{eq:known_probability_scan_test_powerful_condition_corollary} and \eqref{eq:unknown_probability_scan_test_powerful_condition_corollary}, therefore we will consider the case where $\scaling[C] \to \infty$.

Using $h(x - 1) \asymp x \log(x)$ as $x \to \infty$ together with \eqref{eq:known_probability_scan_test_powerful_condition_corollary} or \eqref{eq:unknown_probability_scan_test_powerful_condition_corollary}, we obtain
\begin{equation}
\label{eq:alternative_informative_subgraph_edge_count}
\E_C[e(D^{\star})]
  = (1 + \smallO(1)) \, \frac{\E_0[e(D^{\star})] h(\scaling[C] - 1)}{\log(\scaling[C])}
  \geq (1 + \smallO(1)) \, \frac{|D^{\star}| \log(n / |D^{\star}|)}{\log(\scaling[C])} \,.
\end{equation}
Below we consider the two cases where Assumption~\ref{ass:lower_bound_maximum_inhomogeneity} or Assumption~\ref{ass:lower_bound_maximum_community_small} hold separately.

\paragraph{\normalfont\emph{Case 1} (Assumption~\ref{ass:lower_bound_maximum_inhomogeneity} holds):}
First, by Definition~\ref{def:optimal_subset_argmax} and Assumption~\ref{ass:lower_bound_maximum_inhomogeneity}, it follows that for every $C \subseteq V$,
\begin{equation}
\frac{\E_0[e(D^{\star})]}{|D^{\star}| \log(n / |D^{\star}|)}
  \geq \frac{\E_0[e(C)]}{|C| \log(n / |C|)}
  = (1 + \smallO(1)) \, \frac{r \mspace{2mu} \overline{p}_C}{2 \log(n / r)}
  \to \infty \,,
\end{equation}
where the final step is a consequence of Assumption~\ref{ass:lower_bound_maximum_inhomogeneity}~\ref{ass:lbl:lower_bound_maximum_inhomogeneity_sparsity}. In particular, this means that we must have that $|D^{\star}| \to \infty$.

Then, for every $C \subseteq V$ we have $\scaling[C] \overline{p}_C \leq 1$, and therefore $\scaling[C] \leq 1 / \overline{p}_C \leq r \leq \sqrt{n}$ for $n$ large enough by Assumption~\ref{ass:lower_bound_maximum_inhomogeneity} \ref{ass:lbl:lower_bound_maximum_inhomogeneity_size} and \ref{ass:lbl:lower_bound_maximum_inhomogeneity_sparsity}. Therefore, using \eqref{eq:alternative_informative_subgraph_edge_count} and because $|D^{\star}| \to \infty$, it follows that
\begin{equation}
\E_C[e(D^{\star})]
  \geq (1 + \smallO(1)) \, \frac{|D^{\star}| \log(n / |D^{\star}|)}{\log(\scaling[C])}
  \geq (1 + \smallO(1)) \, |D^{\star}| \, \frac{\log(\sqrt{n})}{\log(\sqrt{n})}
  \to \infty \,.
\end{equation}

\paragraph{\normalfont\emph{Case 2} (Assumption~\ref{ass:lower_bound_maximum_community_small} holds):}
For every $C \subseteq V$ we have $\scaling[C] \overline{p}_C \leq 1$, and therefore $\log(\scaling[C]) \leq \log(1 / \overline{p}_C) = \smallO(\log(n))$ by Assumption~\ref{ass:lower_bound_maximum_community_small}~\ref{ass:lbl:lower_bound_maximum_community_small_sparsity}. Hence, using \eqref{eq:alternative_informative_subgraph_edge_count}, we obtain
\begin{equation}
\E_C[e(D^{\star})]
  \geq (1 + \smallO(1)) \, \frac{|D^{\star}| \log(n / |D^{\star}|)}{\log(\scaling[C])}
  \geq (1 + \smallO(1)) \, \frac{|D^{\star}| \log(n)}{\smallO(\log(n))}
  \to \infty \,.
\end{equation}
The above two cases show that either Assumption~\ref{ass:lower_bound_maximum_inhomogeneity} or \ref{ass:lower_bound_maximum_community_small} is sufficient to ensure that $\E_C[e(D^{\star})] \to \infty$ for every $C \subseteq V$ of size $|C| = r$.
\end{proof}

\subsection{Proof of Corollaries \ref{cor:known_probability_scan_test_powerful2} and \ref{cor:unknown_probability_scan_test_powerful2}}
\label{subsec:proof_of_corollaries2}
\begin{proof}[\unskip\nopunct]
Begin by noting that the conditions in Corollary~\ref{cor:unknown_probability_scan_test_powerful2} imply the conditions on $\pmax$ and $\pmin$ that are stated in Corollary~\ref{cor:known_probability_scan_test_powerful2}. To prove Corollaries \ref{cor:known_probability_scan_test_powerful2} and \ref{cor:unknown_probability_scan_test_powerful2} we need to show that $\E_C[e(\optimalsubgraph)] \to \infty$ for every $C \subseteq V$ of size $|C| = r$. Because $\pmax / \pmin = \smallO(n^{a - b})$, there exists a sequence $x_n \to \infty$ such that $\pmax / \pmin = n^{a - b} / x_n$. We will first show that $|\optimalsubgraph| \geq n^b \sqrt{x_n}$, which we will do by a similar argument as in the proof of Lemma~\ref{lem:optimal_subset_size}. Suppose $|\optimalsubgraph| \leq n^b \sqrt{x_n}$, then because $r \geq n^a$,
\begin{align}
\frac{\E_0[e(\optimalsubgraph)]}{|\optimalsubgraph| \log(n / |\optimalsubgraph|)}
  &\leq \frac{|\optimalsubgraph| - 1}{2} \, \frac{\pmax}{\log(n / |\optimalsubgraph|)}\\
  &= \frac{|\optimalsubgraph| - 1}{2} \, \frac{n^{a - b}}{x_n} \, \frac{\pmin}{\log(n / |\optimalsubgraph|)}\\
  &\leq \bigO(1) \frac{n^a}{\sqrt{x_n}} \, \frac{\pmin}{\log(n / r)}\\
  &< \frac{r - 1}{2} \, \frac{\pmin}{\log(n / r)}
  \leq \frac{\E_0[e(C)]}{|C| \log(n / |C|)} \,.
\end{align}
Hence, $\optimalsubgraph$ cannot be the maximizer in \eqref{eq:optimal_subset_argmax} when $|\optimalsubgraph| \leq n^b \sqrt{x_n}$, and therefore we must have $|\optimalsubgraph| \geq n^b \sqrt{x_n}$. Therefore,
\begin{equation}
\E_C[e(\optimalsubgraph)]
  \geq \E_0[e(\optimalsubgraph)]
  \geq \frac{|\optimalsubgraph|^2}{2} \pmin
  \geq \frac{n^{2b} \mspace{1mu} x_n}{2} n^{-2b}
  \to \infty.
\end{equation}
The proof of Corollary~\ref{cor:known_probability_scan_test_powerful2} is then completed by applying Theorem~\ref{thm:known_probability_scan_test_powerful}, and similarly the proof of Corollary~\ref{cor:unknown_probability_scan_test_powerful2} is completed by applying Theorem~\ref{thm:unknown_probability_scan_test_powerful}.
\end{proof}

\subsection{Proof of Theorem~\ref{thm:lower_bound}: Information theoretic lower bound}
\label{subsec:proof_of_information_theoretic_lower_bound}
\begin{proof}[\unskip\nopunct]
To prove Theorem~\ref{thm:lower_bound} we need to show that $\risk_n(\test_n) \to 1$, where $\risk_n$ is the worst-case risk given in \eqref{eq:worst_case_risk} and $\test_n \mapsto \{0, 1\}$ is any test deciding between the null and alternative hypothesis. The first step is a reduction from the worst-case risk to the average risk
\begin{equation}
\bar{\risk}_n(\test_n) \coloneqq \P_0(\test_n(G) = 1) + \binom{n}{r}^{\!-1} \!\! \sum_{C \subseteq V,\, |C| = r} \P_C(\test_n(G) = 0) \,.
\end{equation}
Note that the average risk is a lower bound for the worst-case risk, that is $\risk_n(\test_n) \geq \bar{\risk}_n(\test_n)$. This average risk corresponds to a hypothesis test between two simple hypotheses, because the alternative hypothesis is now simple. The means that the likelihood ratio test is optimal (by the Neyman-Pearson lemma). In particular, the test $\testlr(G) = \1{\{L(G) > 1\}}$ minimizes the average risk, where $L(G)$ is the likelihood ratio, given in \eqref{eq:likelihood_ratio} below. To avoid overloading the notation we write simply $L$ to denote $L(G)$. The risk of this test is given by
\begin{equation}
\label{eq:likelihood_ratio_test_risk}
\bar{\risk}_n(\testlr)
  = \P_0(L > 1) + \E_0[L \1{\{L \leq 1\}}] 
  = 1 - \frac{1}{2} \, \E_0[|L - 1|] \,.
\end{equation}
Therefore, to prove Theorem~\ref{thm:lower_bound}, it suffices to show that $\E_0[|L - 1|] \to 0$.

Given a graph $g$, the likelihood ratio $L(g)$ is given by
\begin{equation}
\label{eq:likelihood_ratio}
L(g)
  \coloneqq{} \binom{n}{r}^{\!-1} \!\! \sum_{C \subseteq V,\, |C|=r} \frac{\P_C(G = g)}{\P_0(G = g)}
  = \binom{n}{r}^{\!-1} \!\! \sum_{C \subseteq V,\, |C|=r} L_C(g)
  = \bar{\E}[L_C(g)] \,,
\end{equation}
where $\bar{\E}[\cdot]$ denotes the expectation with respect to a uniformly chosen set $C \subseteq V$ of size $|C| = r$, and
\begin{equation}\label{eq:likelihood_ratio_part0}
L_C(g) \coloneqq{}
  \prod_{i < j \in C} \left(\frac{\scaling[C] p_{ij}}{p_{ij}}\right)^{A_{ij}} \left(\frac{1 - \scaling[C]p_{ij}}{1 - p_{ij}}\right)^{1 - A_{ij}}\, .
\end{equation}
To bound $\E_0[|L - 1|]$ one generally resorts to the Cauchy-Schwarz inequality to control instead the second moment of $L$ and obtain $\E_0[|L - 1|]\leq \E_0[L^2]-1$. However, in our setting this bound is too crude, and the variance of $L$ will be rather large in comparison to the first moment. To see this note that the second moment can be written as
\begin{align}
\hspace{10pt}&\hspace{-10pt}
\E_0[L^2]
  = \bar{\E}^{\otimes 2}\left[\E_0\left[L_{C_1} L_{C_2}\right]\right]\\
  &= \bar{\E}^{\otimes 2}\Biggl[\E_0\Biggl[\prod_{i < j \in C_1\cap C_2} \left(\frac{\scaling[C_1] p_{ij} \mspace{2mu} \scaling[C_2] p_{ij}}{p_{ij}^2}\right)^{\!A_{ij}} \left(\frac{(1 - \scaling[C_1]p_{ij})(1 - \scaling[C_2]p_{ij})}{(1 - p_{ij})^2}\right)^{\!1 - A_{ij}}\Biggr]\Biggr] \,,
\end{align}
where $\bar{\E}^{\otimes 2}[\cdot]$ denotes expectation with respect to two independently and uniformly chosen sets $C_1, C_2 \subseteq V$ of size $|C_1| = |C_2| = r$. This second moment depends crucially on $e(C_1 \cap C_2)$, the number of edges in the intersection of $C_1$ and $C_2$. Although this intersection is empty or very small with high probability it can be large with small probability, resulting in a very large second moment if the number of edges inside it is large as well.

To deal with this issue we use a more refined approach suggested by Ingster \cite{Ingster1997} and later used by Butucea and Ingster \cite{Butucea2011} and Arias-Castro and Verzelen \cite{Arias-Castro2014}. This approach relies on a truncation of the likelihood ratio
\begin{equation}
\tilde{L}
  \coloneqq{} \binom{n}{r}^{\!-1} \!\! \sum_{C \subseteq V,\, |C| = r} \1{\truncation[C]} L_C
  = \bar{\E}[\1{\truncation[C]} L_C] \,,
\end{equation}
where $L_C$ is as given by \eqref{eq:likelihood_ratio_part0} and $\truncation[C]$ is some truncation event. Using $\tilde{L} \leq L$, the triangle inequality, and the Cauchy-Schwarz inequality, we obtain the upper bound
\begin{equation}
\E_0[|L - 1|]
  \leq \E_0[|\tilde{L} - 1|] + \E_0[L - \tilde{L}]
  \leq \sqrt{\E_0[\tilde{L}^2] - 2 \E_0[\tilde{L}] + 1} + 1 - \E_0[\tilde{L}] \,.
\end{equation}
Therefore, $\bar{\risk}_n(\testlr) \to 1$ when both $\E_0[\tilde{L}] \to 1$ and $\E_0[\tilde{L}^2] \to 1$. So, the ideal truncation event should lower the variance of $\tilde{L}$ while still ensuring that the first moment of $\tilde{L}$ approaches $1$.

Intuitively, we would like to use the truncation event to prevent ``bad behavior'' at the intersection of two sets $C_1$ and $C_2$. However, we can only state the truncation event in terms of one of these sets. This creates a challenge. For a given set $C \subseteq V$, the potentially problematic intersections are sets $D \subseteq C$ for which $\E_0[e(D)]$ is large. We will denote by $\truncationset[C]$ (see \eqref{eq:truncationset_definition} below) this class of ``potentially problematic'' sets. The idea is then to construct the truncation event so that it removes the set $C$ from consideration if it contains a subset $D \in \truncationset[C]$ for which the number of edges $e(D)$ is significantly larger than its expectation $\E_0[e(D)]$.

To formalize this, it is helpful to express the likelihood ratio in a more convenient form. Namely,
\begin{align}
\label{eq:likelihood_ratio_part}
L_C(g)
  &=  \exp\biggl(\sum_{i < j \in C} A_{ij} \log\left(\frac{\scaling[C] p_{ij}}{p_{ij}}\right) + (1 - A_{ij}) \log\left(\frac{1 - \scaling[C]p_{ij}}{1 - p_{ij}}\right)\biggr)\\
  &= \exp\biggl(\sum_{i < j \in C} A_{ij} \theta_{ij}(\scaling[C] p_{ij}) - \Lambda_{ij}(\theta_{ij}(\scaling[C] p_{ij}))\biggr) \,,
\end{align}
with
\begin{equation}
\theta_{ij}(q) \coloneqq \log\left(\frac{q (1 - p_{ij})}{p_{ij} (1 - q)}\right)\,,
\qquad\text{and}\qquad
\Lambda_{ij}(\theta) \coloneqq \log\left(1 - p_{ij} + p_{ij} \e^\theta\right) \,.
\end{equation}
Note that $\Lambda_{ij}(\theta)$ is the cumulant generating function of $\text{Bern}(p_{ij})$, with Fenchel-Legendre transform given by
\begin{equation}
\label{eq:fenchel_lengendre_transform}
H_{p_{ij}}(q)
  = \sup_{x \geq 0} \left\{ q x - \Lambda_{ij}(x) \right\}
  = q \, \theta_{ij}(q) - \Lambda_{ij}(\theta_{ij}(q)) \,,
\qquad\text{for } q \in (p_{ij}, 1) \,,
\end{equation}
\noeqref{eq:fenchel_lengendre_transform}%
where $H_{p}(q) \coloneqq q \log\bigl(\frac{q}{p}\bigr) + (1 - q) \log\bigl(\frac{1-q}{1-p}\bigr)$ is the Kullback-Leibler divergence between $\text{Bern}(p)$ and $\text{Bern}(q)$.

Now, to construct the truncation event $\truncation[C]$, we begin by defining for each set $C \subseteq V$ a class of ``potentially problematic'' intersection sets as
\begin{equation}
\label{eq:truncationset_definition}
\truncationset[C] \coloneqq \left\{D \subseteq C : (\scaling[C] - 1)^2 \, \E_0[e(D)] > (1 - \epsilon/2) |D| \left(\log\left(\frac{n |D|}{r^2}\right) - b_n\right)\right\} \,,
\end{equation}
where $b_n \to \infty$ very slowly. For concreteness we will take $b_n = \log\log(n / r)$. Using this, we define the numbers $\cutoff[D]$ in the lemma below, the proof of this lemma is mainly technical and is therefore deferred to Section~\ref{subsec:proof_of_auxiliary_results}:

\begin{lemma}
\label{lem:cutoff_definition}
Let Assumption~\ref{ass:lower_bound_sparsity}, and either Assumption \ref{ass:lower_bound_maximum_inhomogeneity} or \ref{ass:lower_bound_maximum_community_small} hold. Then for any $C \subseteq V$ of size $|C| = r$ and $D \in \truncationset[C]$ there exists a unique number $\cutoff[D] \geq 1$, such that for $n$ large enough,
\begin{equation}
(1 + \epsilon) \E_0[e(D)] h(\cutoff[D] - 1) = |D| \log\left(\frac{n}{|D|}\right) \,.
\end{equation}
Moreover, $\cutoff[D]$ satisfies $\theta_{ij}(\cutoff[D] \, p_{ij}) \leq 2\theta_{ij}(\scaling[C] \, p_{ij})$ for every $i, j \in D$.
\end{lemma}

Using the numbers $\cutoff[D] \geq 1$ and $\truncationset[C]$ from \eqref{eq:truncationset_definition}, we finally define the truncation events as
\begin{equation}
\label{eq:truncation_event}
\truncation[C]
  \coloneqq \left\{\sum_{i<j \in D} A_{ij} \, \theta_{ij}\bigl(\scaling[C] p_{ij}\bigr) \leq \sum_{i<j \in D} p_{ij} \cutoff[D] \, \theta_{ij}\bigl(\scaling[C] p_{ij}\bigr) \,,
\quad\text{for all } D \in \truncationset[C] \right\} \,.
\end{equation}
\noeqref{eq:truncation_event}%
Loosely speaking $\theta_{ij}\bigl(\scaling[C] p_{ij}\bigr) \approx \log(\scaling[C])$, so the above truncation event will remove all sets $C \subseteq V$ for which there exists a subset $D \in \truncationset[C]$ with $e(D) > \cutoff[D] \E_0[e(D)]$. Utilizing this truncation event, we need to show that both $\E_0[\tilde{L}] \to 1$ and $\E_0[\tilde{L}^2] \to 1$.

\paragraph{First truncated moment.} Here we show that $\E_0[\tilde{L}] \to 1$. Since we are simply considering a truncation of the likelihood, it follows from Fubini's theorem that
\begin{equation}
\label{eq:first_truncated_moment}
\E_0[\tilde{L}]
  = \bar{\E}[\E_0[\1{\truncation[C]} \, L_C]]
  = \bar{\E}[\P_C(\truncation[C])] \,.
\end{equation}
Hence, it suffices to show that $\P_C(\truncation[C]) \to 1$ for most $C \subseteq V$. Below we will show the slightly stronger result that $\min_{C\subseteq V, |C|=r} \P_C(\truncation[C]) \to 1$, which together with \eqref{eq:first_truncated_moment} shows that $\E_0[\tilde{L}] \to 1$.

Begin by noting that
\begin{equation}
\label{eq:scaling_bound}
\max_{C \subseteq V, |C|=r} \, \max_{i,j \in C} \, \left|\frac{\theta_{ij}\bigl(\scaling[C] p_{ij}\bigr)}{\log\left(\scaling[C]\right)} - 1\right|\to 0 \,,
\qquad
\text{as } n \to \infty \,.
\end{equation}
To see this, consider
\begin{equation}
\label{eq:scaling_bound_rewrite}
\frac{\theta_{ij}\bigl(\scaling[C] p_{ij}\bigr)}{\log\left(\scaling[C]\right)} - 1
  = \frac{\log\left(\frac{1 - p_{ij}}{1 - \scaling[C] p_{ij}}\right)}{\log(\scaling[C])} \,.
\end{equation}
Using Assumption~\ref{ass:lower_bound_sparsity} we see that the above converges to $0$ uniformly over all $i,j \in C$, if $\scaling[C]$ is bounded away from $1$. Otherwise, when $\scaling[C] \to 1$, we can simply use Taylor's theorem to obtain $\log\bigl(\frac{1 - p_{ij}}{1 - \scaling[C] p_{ij}}\bigr) / \log(\scaling[C]) \leq p_{ij} / (1 - p_{ij}) + 3 p_{ij}$ provided $p_{ij}$ is small enough (e.g., $p_{ij} \leq 1 / 3$ suffices). Hence, also in this case it follows from Assumption~\ref{ass:lower_bound_sparsity} that \eqref{eq:scaling_bound_rewrite} converges uniformly to $0$. Loosely speaking, this means that $\theta_{ij}\bigl(\scaling[C] p_{ij}\bigr) \asymp \log(\scaling[C])$ for all sets $C \subseteq V$ and $i,j \in C$. This, together with a union bound and Bennett's inequality, allows us to control $\P_C(\truncation[C])$. Indeed,
\begin{align}
1 - \P_C(\truncation[C])
  &\leq \sum_{D \in \truncationset[C]}
\P_C\biggl(\sum_{i<j \in D} A_{ij} \theta_{ij}\bigl(\scaling[C] p_{ij}\bigr) > \sum_{i<j \in D} p_{ij} \cutoff[D] \theta_{ij}\bigl(\scaling[C] p_{ij}\bigr)\biggr)\\
  &\leq \sum_{D \in \truncationset[C]} \P_C\biggl(\sum_{i<j \in D} A_{ij} > (1 + \smallO(1))\cutoff[D] \sum_{i<j \in D} p_{ij}\biggr)\\
  &\leq \sum_{D \in \truncationset[C]} \exp\left(- \E_C[e(D)] \, h\!\left((1 + \smallO(1))\left(\frac{\cutoff[D]}{\scaling[C]} - 1\right)\right)\right) \\
  &= \sum_{D \in \truncationset[C]} \exp\left(- (1 + \smallO(1)) \E_C[e(D)] \, h\!\left(\frac{\cutoff[D]}{\scaling[C]} - 1\right)\right) \,,
\end{align}
where the last step uses a property of the $h$ function, which ensures that for $t \geq 1$, $x \geq 0$ we have $\sqrt{t}h(x)\leq h(t x) \leq t^2h(x)$.

To show that this vanishes we need the following lemma, the proof of which is mainly technical and therefore deferred to Section~\ref{subsec:proof_of_auxiliary_results}. We remark that the definition of $a_n$ in this lemma comes from the exponent in \eqref{eq:relation_defining_an} below.
\begin{restatable}{lemma}{cutoffproperty}
\label{lem:cutoff_property}
Define the sequence $a_n$ as
\begin{equation}
a_n \coloneqq \min_{C \subseteq V, \, |C| = r}  \min_{D \in \truncationset[C]}
\left((1 - \epsilon) \frac{\E_C[e(D)]}{|D|} \, h\!\left(\frac{\cutoff[D]}{\scaling[C]} - 1\right) - \log\left(\frac{r}{|D|}\right)\right)\,.
\end{equation}
When \eqref{eq:lower_bound_condition}, Assumption~\ref{ass:lower_bound_sparsity}, and either Assumption \ref{ass:lower_bound_maximum_inhomogeneity} or \ref{ass:lower_bound_maximum_community_small} hold, then $a_n \to \infty$.
\end{restatable}

Using Lemma~\ref{lem:cutoff_property} and grouping the sets $D \in \truncationset[C]$ by their size $|D|$, together with the bound on the binomial coefficient $\binom{r}{k} \leq \left(\frac{r\,\e}{k}\right)^k$ we conclude that, for $n$ large enough,
\begin{align}
\hspace{14pt}&\hspace{-14pt}
\hspace{4pt}1 - \min_{C\subseteq V, |C|=r} \P_C(\truncation[C])\\
  &\leq \max_{C\subseteq V, |C|=r} \sum_{D \in \truncationset[C]} \exp\left(- (1 + \smallO(1)) \E_C[e(D)] \, h\!\left(\frac{\cutoff[D]}{\scaling[C]} - 1\right)\right)\\
  &= \max_{C\subseteq V, |C|=r} \sum_{k=1}^r \sum_{D \in \truncationset[C], \, |D| = k} \exp\left(- (1 + \smallO(1)) \E_C[e(D)] \, h\!\left(\frac{\cutoff[D]}{\scaling[C]} - 1\right)\right)\\
  &= \max_{C\subseteq V, |C|=r} \sum_{k=1}^r \sum_{D \in \truncationset[C], \, |D| = k} \frac{1}{(r\e/k)^k}\exp\biggl(- k \biggl((1 + \smallO(1)) \frac{\E_C[e(D)]}{k}\, h\!\left(\frac{\cutoff[D]}{\scaling[C]} - 1\right) - \log(r\e/k)\biggr)\biggr)\notag\\
  &\leq \max_{C\subseteq V, |C|=r} \sum_{k=1}^r \binom{r}{k}^{-1} \!\!\! \sum_{D \in \truncationset[C], \, |D| = k} \exp\biggl(- k \biggl((1 + \smallO(1)) \frac{\E_C[e(D)]}{k}\, h\!\left(\frac{\cutoff[D]}{\scaling[C]} - 1\right) - \log(r\e/k)\biggr)\biggr)\notag\\
\label{eq:relation_defining_an}
  &\leq \sum_{k = 1}^r \binom{r}{k}^{-1} \!\!\! \sum_{D \subseteq C, \, |D| = k} \exp\left(-k \, (a_n - 1)\right)\\
  &= \sum_{k = 1}^r \exp\left(-k \, (a_n - 1)\right)
  \leq \frac{\exp(-(a_n-1))}{1-\exp(-(a_n-1))}
  \to 0 \,,
\end{align}
where the final step follows because $a_n \to \infty$ by Lemma~\ref{lem:cutoff_property}. Hence, from \eqref{eq:first_truncated_moment} we see that $\E_0[\tilde{L}] \to 1$.

\paragraph{Second truncated moment.} Here we show that $\E_0[\tilde{L}^2] \to 1$. In other words,
\begin{equation}
\E_0[\tilde{L}^2]
  = \bar{\E}^{\otimes 2}\left[\E_0\left[\1{\truncation[C_1]} \1{\truncation[C_2]} L_{C_1} L_{C_2}\right]\right] 
  \leq 1 + \smallO(1) \,,
\end{equation}
where we recall that $\bar{\E}^{\otimes 2}[\cdot]$ denotes expectation with respect to two independently and uniformly chosen sets $C_1, C_2 \subseteq V$ of size $|C_1| = |C_2| = r$. Let $D = C_1 \cap C_2$, then using \eqref{eq:likelihood_ratio_part} this becomes
\begin{align}
\E_0[\tilde{L}^2]
  &= \bar{\E}^{\otimes 2}\left[\E_0\left[\1{\truncation[C_1]} \1{\truncation[C_2]} L_{C_1} L_{C_2}\right]\right]\\
  &= \begin{multlined}[t][0.85\displaywidth]
  \bar{\E}^{\otimes 2}\Bigg[\E_0\Bigg[\1{\truncation[C_1] \cap \truncation[C_2]} \exp\Bigg(\sum_{i < j \in D} A_{ij} \left(\theta_{ij}\bigl(\scaling[C_1] p_{ij}\bigr) + \theta_{ij}\bigl(\scaling[C_2] p_{ij}\bigr)\right)\\*
    {} - \Lambda_{ij}\bigl(\theta_{ij}\bigl(\scaling[C_1] p_{ij}\bigr)\bigr) - \Lambda_{ij}\bigl(\theta_{ij}\bigl(\scaling[C_2] p_{ij}\bigr)\bigr)\Bigg)\Bigg]\Bigg] \,,
  \end{multlined}
\end{align}
where we note that the sum runs only over $i < j \in D = C_1 \cap C_2$. The remaining terms in the sum above (i.e., the terms $i, j \in C_1 \cup C_2$ with $i \notin D$ or $j \notin D$) can all be factorized because the $A_{ij}$ are independent, and all these terms have a zero contribution because their expectation equals one.

Using the Cauchy-Schwarz inequality inside the expectation $\bar{\E}^{\otimes 2}[\cdot]$, and that the sets $C_1$ and $C_2$ are chosen independently, we obtain
\begin{align}
\E_0[\tilde{L}^2]
  &= \begin{multlined}[t][0.85\displaywidth]
  \bar{\E}^{\otimes 2}\Bigg[\E_0\Bigg[\1{\truncation[C_1]} \exp\Bigg(\sum_{i < j \in D} A_{ij} \theta_{ij}\bigl(\scaling[C_1] p_{ij}\bigr) - \Lambda_{ij}\bigl(\theta_{ij}\bigl(\scaling[C_1] p_{ij}\bigr)\bigr)\Bigg)\notag\\*
     {}{\scriptstyle\times} \; \1{\truncation[C_2]} \exp\Bigg(\sum_{i < j \in D} A_{ij} \theta_{ij}\bigl(\scaling[C_2] p_{ij}\bigr) - \Lambda_{ij}\bigl(\theta_{ij}\bigl(\scaling[C_2] p_{ij}\bigr)\bigr)\Bigg)\Bigg]\Bigg]\notag
  \end{multlined}\\[0.5ex]
  &\leq \begin{multlined}[t][0.85\displaywidth]
  \bar{\E}^{\otimes 2}\Bigg[\E_0\Bigg[\1{\truncation[C_1]} \exp\Bigg(\sum_{i < j \in D} 2 A_{ij} \theta_{ij}\bigl(\scaling[C_1] p_{ij}\bigr) - 2\Lambda_{ij}\bigl(\theta_{ij}\bigl(\scaling[C_1] p_{ij}\bigr)\bigr)\Bigg)\Bigg]^{1/2}\notag\\*
    {}{\scriptstyle\times}\, \E_0\Bigg[\1{\truncation[C_2]} \exp\Bigg(\sum_{i < j \in D} 2 A_{ij} \theta_{ij}\bigl(\scaling[C_2] p_{ij}\bigr) - 2\Lambda_{ij}\bigl(\theta_{ij}\bigl(\scaling[C_2] p_{ij}\bigr)\bigr)\Bigg)\Bigg]^{1/2} \, \Bigg]\notag
  \end{multlined}\\[0.5ex]
  &= \bar{\E}^{\otimes 2}\Bigg[\E_0\Bigg[\1{\truncation[C_1]} \exp\Bigg(\sum_{i < j \in D} 2 A_{ij} \theta_{ij}\bigl(\scaling[C_1] p_{ij}\bigr) - 2\Lambda_{ij}\bigl(\theta_{ij}\bigl(\scaling[C_1] p_{ij}\bigr)\bigr)\Bigg)\Bigg]\Bigg] \,.
\end{align}
Next, we split this expectation into two parts based on whether $D \notin \truncationset[C_1]$ or $D \in \truncationset[C_1]$. Thus we have the partition
\begin{equation}
\E_0[\tilde{L}^2]
  \leq \textup{\textsf{P}}_1 + \textup{\textsf{P}}_2 \,,
\end{equation}
where
\begin{align}
\textup{\textsf{P}}_1
  &\coloneqq \bar{\E}^{\otimes 2}\Bigg[\1{\{D \notin \truncationset[C_1]\}} \, \E_0\Bigg[\1{\truncation[C_1]} \exp\Bigg(\sum_{i < j \in D} 2 A_{ij} \theta_{ij}\bigl(\scaling[C_1] p_{ij}\bigr) - 2\Lambda_{ij}\bigl(\theta_{ij}\bigl(\scaling[C_1] p_{ij}\bigr)\bigr)\Bigg)\Bigg]\Bigg] \,,\\
\textup{\textsf{P}}_2
  &\coloneqq \bar{\E}^{\otimes 2}\Bigg[\1{\{D \in \truncationset[C_1]\}} \, \E_0\Bigg[\1{\truncation[C_1]} \exp\Bigg(\sum_{i < j \in D} 2 A_{ij} \theta_{ij}\bigl(\scaling[C_1] p_{ij}\bigr) - 2\Lambda_{ij}\bigl(\theta_{ij}\bigl(\scaling[C_1] p_{ij}\bigr)\bigr)\Bigg)\Bigg]\Bigg] \,.
\end{align}
Using this split, we first show that $\textup{\textsf{P}}_1 \leq 1 + \smallO(1)$ and then show that $\textup{\textsf{P}}_2 \leq \smallO(1)$.

\paragraph{\normalfont\emph{Part 1}:}
Here we show that $\textup{\textsf{P}}_1 \leq 1 + \smallO(1)$. In this part we can simply ignore the truncation events $\truncation[C_1]$ and obtain the bound
\begin{align}
\textup{\textsf{P}}_1
  &\leq \bar{\E}^{\otimes 2}\Bigg[\1{\{D \notin \truncationset[C_1]\}} \, \E_0\Bigg[ \exp\Bigg(\sum_{i < j \in D} 2 A_{ij} \theta_{ij}\bigl(\scaling[C_1] p_{ij}\bigr) - 2\Lambda_{ij}\bigl(\theta_{ij}\bigl(\scaling[C_1] p_{ij}\bigr)\bigr)\Bigg)\Bigg]\Bigg]\\
  &\leq \bar{\E}^{\otimes 2}\Bigg[\1{\{D \notin \truncationset[C_1]\}} \exp\Bigg(\sum_{i < j \in D} \Delta_{ij}^{\!\scriptscriptstyle(1)}\Bigg)\Bigg] \,,
\end{align}
where
\begin{equation}
\Delta_{ij}^{\!\scriptscriptstyle(1)}
  \vcentcolon= \log\left(1 + \frac{(\scaling[C_1] p_{ij} - p_{ij})^2}{p_{ij} (1 - p_{ij})}\right) \,.
\end{equation}
Then using $\log(1 + x) \leq x$ and by Assumption~\ref{ass:lower_bound_sparsity}, uniformly over all $i,j \in D$,
\begin{equation}
\Delta_{ij}^{\!\scriptscriptstyle(1)}
  \leq \log\bigl(1 + (1 + \smallO(1)) (\scaling[C_1] - 1)^2 p_{ij}\bigr)
  \leq (1 + \smallO(1)) (\scaling[C_1] - 1)^2 p_{ij} \,.
\end{equation}
Now, by definition of $\truncationset[C_1]$ it follows that $(1 + \smallO(1)) (\scaling[C_1] - 1)^2 \, \E_0[e(D)] \leq |D| \bigl(\log\bigl(\frac{n |D|}{r^2}\bigr) - b_n\bigr)$ for every $D \notin \truncationset[C_1]$. Therefore
\begin{align}
\textup{\textsf{P}}_1
  &\leq \bar{\E}^{\otimes 2}\left[\1{\{D \notin \truncationset[C_1]\}} \exp\Bigl((1 + \smallO(1)) (\scaling[C_1] - 1)^2 \E_0[e(D)]\Bigr)\right]\\
  &\leq \bar{\E}^{\otimes 2}\left[\1{\{|D|\leq 1\}}+\1{\{|D|>1\}}\exp\left(|D| \left(\log\!\left(\frac{n |D|}{r^2}\right) - b_n\right)\right)\right]\\
  &\leq \bar{\P}^{\otimes 2}(|D| \leq 1) + \sum_{k=2}^r \exp\left(k \left(\log\!\left(\frac{n k}{r^2}\right) - b_n\right)\right) \bar{\P}^{\otimes 2}(|D| = k)\\
\label{eq:second_truncated_moment_part1_bound}
  &\leq 1 + \sum_{k=2}^r \exp\left(k \left(\log\!\left(\frac{n k}{r^2}\right) - b_n\right)\right) \bar{\P}^{\otimes 2}(|D| = k) \,.
\end{align}
Note that $|D| = |C_1 \cap C_2|$ has a hypergeometric distribution under $\bar{\P}^{\otimes 2}$, hence
\begin{align}
\bar{\P}(|D| = k)
  &= \frac{\binom{r}{k} \binom{n-r}{r-k}}{\binom{n}{r}}
  = \left((1 + \smallO(1)) \, \frac{r \e}{k} \, \frac{r-k}{n-r}\right)^k\\
\label{eq:hypergeometric_pdf_bound}
  &\leq \exp\left(- k \left(\log\!\left(\frac{n k}{r^2}\right) + \bigO(1)\right)\right) \,.
\end{align}
Plugging this into \eqref{eq:second_truncated_moment_part1_bound}, we obtain
\begin{align}
\textup{\textsf{P}}_1  &\leq 1 + \sum_{k=2}^r \exp\left(k \left(\log\!\left(\frac{n k}{r^2}\right) - b_n - \log\!\left(\frac{n k}{r^2}\right) + \bigO(1)\right)\right)\\
  &\leq 1 + \sum_{k=2}^r \exp\Bigl(k\bigl(\bigO(1) - b_n\bigr)\Bigr)
  \leq 1 + \smallO(1) \,,
\end{align}
where the final step follows because $b_n = \log\log(n / r) \to \infty$.

\paragraph{\normalfont\emph{Part 2}:}
Here we show that $\textup{\textsf{P}}_2 \leq \smallO(1)$. First, define
\begin{equation}
\label{eq:interpolate_between_cutoff}
\xi \coloneqq \frac{1}{2} \mspace{2mu} \frac{\log(\cutoff[D])}{\log(\scaling[C_1])} \,,
\end{equation}
where $\cutoff[D]$ was defined in Lemma~\ref{lem:cutoff_definition}. Then, by the same reasoning as in \eqref{eq:scaling_bound},
\begin{equation}
\label{eq:uniform_bound_on_interpolate_between_cutoff}
\max_{C_1 \subseteq V, |C_1|=r} \, \max_{D \in \truncationset[C_1]} \, \max_{i,j \in D} \, 
\left|\frac{\log(\cutoff[D]) / \log(\scaling[C_1])}{\theta_{ij}(\cutoff[D] p_{ij}) / \theta_{ij}(\scaling[C_1] p_{ij})} - 1\right|\to 0 \,,
\qquad
\text{as } n \to \infty \,.
\end{equation}
Loosely speaking, this means that, $\xi \asymp \frac{\theta_{ij}(\cutoff[D] p_{ij})}{2 \theta_{ij}(\scaling[C_1] p_{ij})} \leq 1$ uniformly over $i,j \in D$.

By definition of the truncation event $\truncation[C_1]$ in \eqref{eq:truncation_event}, for any $D \in \truncationset[C_1]$,
\begin{equation}
\sum_{i<j \in D} A_{ij} \, \theta_{ij}\bigl(\scaling[C] p_{ij}\bigr) \leq \sum_{i<j \in D} p_{ij} \cutoff[D] \, \theta_{ij}\bigl(\scaling[C] p_{ij}\bigr) \,.
\end{equation}
Then for $x \in [0, 1]$, we obtain the bound
\begin{align}
\textup{\textsf{P}}_2
  &= \bar{\E}^{\otimes 2}\Bigg[\1{\{D \in \truncationset[C_1]\}} \, \E_0\Bigg[\1{\truncation[C_1]} \exp\Bigg(\sum_{i < j \in D} 2 A_{ij} \theta_{ij}\bigl(\scaling[C_1] p_{ij}\bigr) - 2\Lambda_{ij}\bigl(\theta_{ij}\bigl(\scaling[C_1] p_{ij}\bigr)\bigr)\Bigg)\Bigg]\Bigg]\notag\\
  &\leq \begin{multlined}[t][0.92\displaywidth]
  \bar{\E}^{\otimes 2}\Bigg[\1{\{D \in \truncationset[C_1]\}} \, \E_0\Bigg[\exp\Bigg(\sum_{i < j \in D} 2 \theta_{ij}\bigl(\scaling[C_1] p_{ij}\bigr) \Bigl[x A_{ij} + (1 - x) \cutoff[D] p_{ij}\Bigr]\notag\\*[-0.6ex]
  {} - 2\Lambda_{ij}\bigl(\theta_{ij}\bigl(\scaling[C_1] p_{ij}\bigr)\bigr)\Bigg)\Bigg] \hphantom{\,.}\notag
  \end{multlined}\\
  &= \begin{multlined}[t][0.92\displaywidth]
  \bar{\E}^{\otimes 2}\Bigg[\1{\{D \in \truncationset[C_1]\}} \, \exp\Bigg(\sum_{i < j \in D} \Lambda_{ij}\bigl(2 \theta_{ij}\bigl(\scaling[C_1] p_{ij}\bigr) x\bigr)\notag\\*[-0.3ex]
  {} + \bigl(2 \theta_{ij}\bigl(\scaling[C_1] p_{ij}\bigr) - 2 \theta_{ij}\bigl(\scaling[C_1] p_{ij}\bigr)x\bigr) \cutoff[D] p_{ij} - 2\Lambda_{ij}\bigl(\theta_{ij}\bigl(\scaling[C_1] p_{ij}\bigr)\bigr)\Bigg)\Bigg] \,.\notag
  \end{multlined}
\end{align}
To obtain the best possible bound we optimize the above with respect to $x$. Here it can be seen from \eqref{eq:fenchel_lengendre_transform} that each individual term in the sum is minimal when $x = \frac{\theta_{ij}(\cutoff[D] p_{ij})}{2 \theta_{ij}(\scaling[C_1] p_{ij})}$. Therefore, by \eqref{eq:uniform_bound_on_interpolate_between_cutoff} it follows that the overall optimum is attained at $x = (1 + \smallO(1)) \xi$, where $\xi$ was defined in \eqref{eq:interpolate_between_cutoff}. Plugging this in, and using \eqref{eq:uniform_bound_on_interpolate_between_cutoff}, gives
\begin{align}
\textup{\textsf{P}}_2
  &\leq \bar{\E}^{\otimes 2}\Bigg[\1{\{D \in \truncationset[C_1]\}} \, \exp\Biggl(\sum_{i < j \in D} \Delta_{ij}^{\!\scriptscriptstyle(2)}\Biggr)\Bigg] \,,
\end{align}
where
\begin{align}
\hspace{4pt}\Delta_{ij}^{\!\scriptscriptstyle(2)}
  &\vcentcolon= \Bigl(\Lambda_{ij}(\theta_{ij}\bigl(\cutoff[D] p_{ij}\bigr)) - \cutoff[D] p_{ij} \theta_{ij}\bigl(\cutoff[D] p_{ij}\bigr)\Bigr)
   - 2 \Bigl(\Lambda_{ij}\bigl(\theta_{ij}\bigl(\scaling[C_1] p_{ij}\bigr)\bigr) - \cutoff[D] p_{ij} \theta_{ij}\bigl(\scaling[C_1] p_{ij}\bigr)\Bigr)\notag\\
  &\hphantom{\vcentcolon}= - H_{p_{ij}}(\cutoff[D] p_{ij}) - 2 \left(H_{\scaling[C_1] p_{ij}}(\cutoff[D] p_{ij}) - H_{p_{ij}}(\cutoff[D] p_{ij})\right)\notag\\
\label{eq:delta2}
  &\hphantom{\vcentcolon}= H_{p_{ij}}(\cutoff[D] p_{ij}) - 2H_{\scaling[C_1] p_{ij}}(\cutoff[D] p_{ij}) \,,
\end{align}
where we have used \eqref{eq:fenchel_lengendre_transform} in the second equality. To relate the Kullback-Leibler divergence $H_{p}(q)$, appearing in \eqref{eq:delta2}, to the function $h(x)$ from \eqref{eq:poisson_rate_function} we need the following lemma, the proof of which is deferred to Section~\ref{subsec:proof_of_auxiliary_results}:
\begin{restatable}{lemma}{hratio}
\label{lem:h_ratio}
For any $0 < p < q < 1/2$ (possibly depending on $n$) it follows that,
\begin{equation}
\left| \frac{H_{p}(q)}{p \mspace{1mu} h\left(\frac{q}{p} - 1\right)} - 1 \right|
  \leq \bigO\left(p + q\right) \,,
\end{equation}
where $H_{p}(q)$ is the Kullback-Leibler divergence between $\text{Bern}(p)$ and $\text{Bern}(q)$, and $h(x)$ is given in \eqref{eq:poisson_rate_function}.
\end{restatable}

Recall that $\cutoff[D] \leq \scaling[C]^2$ by Lemma~\ref{lem:cutoff_definition}, and therefore $\max_{i,j \in D} p_{ij} \cutoff[D] = \smallO(1)$ by Assumption~\ref{ass:lower_bound_sparsity}. Similarly, it follows that $\max_{i,j \in D} p_{ij} \scaling[C] = \smallO(1)$ and $\max_{i,j \in D} p_{ij} = \smallO(1)$. Then, using Lemma~\ref{lem:h_ratio} we obtain the bounds, uniformly over $i, j, \in D$,
\begin{align}
\left| \frac{H_{p_{ij}}(p_{ij} \cutoff[D])}{p_{ij} \mspace{1mu} h\Bigl(\cutoff[D] - 1\Bigr)} - 1 \right|
  &= \bigO\bigl(p_{ij} (\cutoff[D] + 1)\bigr)
  \leq \max_{i, j \in D} \, \bigO\bigl(p_{ij} (\cutoff[D] + 1)\bigr)
  = \smallO(1) \,,\\
\left| \frac{H_{p_{ij} \scaling[C]}(p_{ij} \cutoff[D])}{p_{ij} \scaling[C] \mspace{1mu} h\Bigl(\frac{\cutoff[D]}{\scaling[C]} - 1\Bigr)} - 1\right|
  &= \bigO\bigl(p_{ij} (\cutoff[D] + \scaling[C])\bigr)
  \leq \max_{i, j \in D} \, \bigO\bigl(p_{ij} (\cutoff[D] + \scaling[C])\bigr)
  = \smallO(1) \,.
\end{align}
Using the uniform bounds above, we can express $\Delta_{ij}^{\!\scriptscriptstyle(2)}$ from \eqref{eq:delta2} in terms on the function $h(x)$. This gives, uniformly over $i, j \in D$, 
\begin{align}
\Delta_{ij}^{\!\scriptscriptstyle(2)}
  &= H_{p_{ij}}(p_{ij} \cutoff[D]) - 2H_{\scaling[C_1] p_{ij}}(p_{ij} \cutoff[D])\\
  &= (1 + \smallO(1)) \left( p_{ij} h\Bigl(\cutoff[D] - 1\Bigr) - 2\scaling[C_1] p_{ij} h\Bigl(\frac{\cutoff[D]}{\scaling[C_1]} - 1\Bigr) \right) \,.
\end{align}
Therefore, for $D \in \truncationset[C_1]$, we have
\begin{align}
\hspace{15pt}&\hspace{-15pt}
\frac{1}{|D|} \sum_{i < j \in D} \Delta_{ij}^{\!\scriptscriptstyle(2)} - \log\!\left(\frac{n |D|}{r^2}\right)\\
  &= (1 + \smallO(1)) \frac{1}{|D|} \sum_{i < j \in D} \left[p_{ij} h\Bigl(\cutoff[D] - 1\Bigr) - 2\scaling[C_1] p_{ij} h\Bigl(\frac{\cutoff[D]}{\scaling[C_1]} - 1\Bigr)\right] - \log\!\left(\frac{n |D|}{r^2}\right)\\
  &= \begin{multlined}[t][0.88\displaywidth]
  (1 + \smallO(1)) \left[\frac{\E_0[e(D)]}{|D|} h\Bigl(\cutoff[D] - 1\Bigr) - 2\frac{\E_{C_1}[e(D)]}{|D|} h\Bigl(\frac{\cutoff[D]}{\scaling[C_1]} - 1\Bigr)\right]\\*
    {} - \left(\log\!\left(\frac{n}{|D|}\right) - 2\log\!\left(\frac{r}{|D|}\right)\right) \,.
  \end{multlined}
\end{align}
Then, by definition of $\cutoff[D]$ in Lemma~\ref{lem:cutoff_definition} and $a_n$ in Lemma~\ref{lem:cutoff_property}, this becomes
\begin{multline}
\max_{C_1 \subseteq V,\, |C_1|=r} \, \max_{D \in \truncationset[C_1]} \frac{1}{|D|} \sum_{i < j \in D} \Delta_{ij}^{\!\scriptscriptstyle(2)} - \log\!\left(\frac{n |D|}{r^2}\right)\\[-1ex]
\begin{aligned}[b]
  &\leq \max_{C_1 \subseteq V,\, |C_1|=r} \, \max_{D \in \truncationset[C_1]} 2\left(\log\!\left(\frac{r}{|D|}\right) - (1 + \smallO(1))\frac{\E_{C_1}[e(D)] h\Bigl(\frac{\cutoff[D]}{\scaling[C_1]} - 1\Bigr)}{|D|}\right)\\
  &\leq -2 a_n 
    \to -\infty \,.
\end{aligned}
\end{multline}
Combining the above and grouping the sets $D \in \truncationset[C_1]$ by their size $|D|$, together with \eqref{eq:hypergeometric_pdf_bound}, we obtain
\begin{align}
\textup{\textsf{P}}_2
  &\leq \bar{\E}^{\otimes 2}\biggl[\1{\{D \in \truncationset[C_1]\}} \, \exp\biggl(\sum_{i < j \in D} \Delta_{ij}^{\!\scriptscriptstyle(2)}\biggr)\biggr]\\
  &\leq \sum_{k = 1}^{r} \exp\left(k \left(-2 a_n + \log\!\left(\frac{nk}{r^2}\right)\right)\right) \, \bar{\P}(|D| = k)\\
  &\leq \sum_{k = 1}^{r} \exp\left(k \left(-2 a_n + \log\!\left(\frac{nk}{r^2}\right) - \log\!\left(\frac{nk}{r^2}\right) + \bigO(1)\right)\right)\\
  &\leq \sum_{k = 1}^{r} \exp\bigl(k (-2 a_n + \bigO(1))\bigr)
  \:\to\: \vphantom{\frac{1}{2}}0 \,,
\end{align}
where the final step follows because $a_n \to \infty$ by Lemma~\ref{lem:cutoff_property}. This shows that $\textup{\textsf{P}}_2 = \smallO(1)$.

Following our steps, we conclude that $\E_0[\tilde{L}] \to 1$ and $\E_0[\tilde{L}^2] = \textup{\textsf{P}}_1 + \textup{\textsf{P}}_2 \leq 1 + \smallO(1)$, and therefore $\bar{\risk}_n(\testlr) \to 1$. Finally, the risk of any test $\test_n$ is bounded by the average risk of the likelihood ratio test, that is $\risk_n(\test_n) \geq \bar{\risk}_n(\testlr) \to 1$, completing the proof of Theorem\ref{thm:lower_bound}.
\end{proof}

\subsection{Proof of auxiliary results}
\label{subsec:proof_of_auxiliary_results}
In this section we provide the proofs for Lemmas \ref{lem:cutoff_definition}, \ref{lem:cutoff_property}, and \ref{lem:h_ratio}. To simplify this, we first compile Assumptions \ref{ass:lower_bound_maximum_inhomogeneity} and \ref{ass:lower_bound_maximum_community_small} into a single result. This is the only place in the proof of Theorem~\ref{thm:lower_bound} where Assumptions \ref{ass:lower_bound_maximum_inhomogeneity} and \ref{ass:lower_bound_maximum_community_small} are used directly. Thus, Theorem~\ref{thm:lower_bound} can simply be extended to other assumptions, provided one can prove Lemma~\ref{lem:maximal_set_ratio} below under the new set of assumptions made.

\begin{restatable}{lemma}{maximalsetratio}
\label{lem:maximal_set_ratio}
Let \eqref{eq:lower_bound_condition}, Assumption~\ref{ass:lower_bound_sparsity}, and either Assumption \ref{ass:lower_bound_maximum_inhomogeneity} or \ref{ass:lower_bound_maximum_community_small} hold. Then, for all $C \subseteq V$ of size $|C| = r$ and for all $D \in \mathcal{E}_C$,
\begin{equation}
\label{eq:maximal_set_ratio}
\frac{\log(r / |D|)}{\log(n / r)} \bigl(\log(\scaling[C]) \vee 1\bigr) = \smallO(1) \,.
\end{equation}
Furthermore, $\log(n / r) / \log(\scaling[C]) \to \infty$ for all $C \subseteq V$ of size $|C| = r$.
\end{restatable}

\subsubsection{Proof of Lemma~\ref{lem:maximal_set_ratio}}
\begin{proof}[\unskip\nopunct]
Below we consider two cases depending on whether Assumption~\ref{ass:lower_bound_maximum_inhomogeneity} or Assumption~\ref{ass:lower_bound_maximum_community_small} holds. We note that some of these inequalities below only hold when $n$ is large enough.

\paragraph{\normalfont\emph{Case 1} (Assumptions \ref{ass:lower_bound_sparsity} and \ref{ass:lower_bound_maximum_inhomogeneity} hold):}
For all $C \subseteq V$ of size $|C| = r$, define $\eta_C \geq \scaling[C]$, such that
\begin{equation}
\frac{|C| \mspace{2mu} \overline{p}_C \mspace{1mu} h(\eta_C - 1)}{2 \mspace{1mu} \log(n / r)} = 1 - \frac{2}{3} \mspace{2mu} \epsilon \,,
\end{equation}
where $\epsilon$ comes from \eqref{eq:lower_bound_condition}. Further, by Assumption~\ref{ass:lower_bound_maximum_inhomogeneity}~\ref{ass:lbl:lower_bound_maximum_inhomogeneity_sparsity}, we obtain
\begin{equation}
\label{eq:xi_convergence_to_one}
h(\eta_C - 1) \leq \frac{2 \mspace{1mu} \log(n / r)}{r \mspace{2mu} \overline{p}_C} = \smallO(1) \,.
\end{equation}
Hence, $\eta_C \to 1$ and thus $(\eta_C - 1)^2 / h(\eta_C - 1) \to 2$ for every $C \subseteq V$ of size $|C| = r$. Using this together with Assumption~\ref{ass:lower_bound_maximum_inhomogeneity}~\ref{ass:lbl:lower_bound_maximum_inhomogeneity_inhomogeneity}, we obtain, for all $C \subseteq V$ of size $|C| = r$ and for all $D \subseteq C$ of size $|D| < r / \smash{(n / r)^{\gamma_n}}$, that
\begin{align}
(\scaling[C] - 1)^2 \frac{\E_0[e(D)]}{|D|}
  &\leq (\eta_C - 1)^2 \frac{|D| \mspace{2mu} \overline{p}_D}{2}
  \leq \delta (\eta_C - 1)^2 \frac{|C| \mspace{2mu} \overline{p}_C}{2}\\
  &= \bigl(1 - \frac{2}{3} \mspace{2mu} \epsilon\bigr) \log\left(\frac{n}{r}\right) \mspace{3mu} \delta \mspace{2mu} \frac{(\eta_C - 1)^2}{h(\eta_C - 1)}\\
  &\leq \bigl(1 - \frac{2}{3} \mspace{2mu} \epsilon\bigr) \log\left(\frac{n}{r}\right) \mspace{3mu} 2 \mspace{1mu} \delta \mspace{3mu} (1 + \smallO(1))\\
\label{eq:truncationset_bound}
  &\leq (1 - \epsilon / 2) \left(\log\left(\frac{n |D|}{r^2}\right) - b_n\right) \,,
\end{align}
where we recall that $b_n = \log\log(n / r)$. Furthermore, the final inequality above (in \eqref{eq:truncationset_bound}) follows since
\begin{align}
2 \delta \log(n/r)
  &\leq 2 \delta \log(n)\\
  &\leq \log(n / r^2) + \bigO(1)\\
  &\leq \log(n |D| / r^2) + \bigO(1)\\
  &\leq (1 + \smallO(1)) \, (\log(n |D| / r^2) - b_n) \,,
\end{align}
because $r = \bigO(\smash{n^{1/2 - \delta}})$ by Assumption~\ref{ass:lower_bound_maximum_inhomogeneity}~\ref{ass:lbl:lower_bound_maximum_inhomogeneity_size}.

Therefore, by definition of $\truncationset[C]$ (see \eqref{eq:truncationset_definition}) it follows that, for all $C \subseteq V$ of size $|C| = r$ and $D \in \truncationset[C]$, we have $|D| \geq r / (n/r)^{\gamma_n}$, or equivalently $\log(r / |D|) / \log(n / r) \leq \gamma_n = \smallO(1)$. Furthermore, by \eqref{eq:xi_convergence_to_one} we have $\scaling[C] \to 1$ for all $C \subseteq V$ of size $|C| = r$. Combining this, we obtain
\begin{equation}
\frac{\log(r / |D|)}{\log(n / r)} \bigl(\log(\scaling[C]) \vee 1\bigr)
  \leq \frac{\log(r / |D|)}{\log(n / r)}
  \leq \gamma_n
  = \smallO(1) \,.
\end{equation}
This shows that \eqref{eq:maximal_set_ratio} holds.

To complete the proof, we need to show that $\log(n / r) / \log(\scaling[C]) \to \infty$ for all $C \subseteq V$ of size $|C| = r$. This is trivial because $\scaling[C] \to 1$, and therefore we have proved Lemma~\ref{lem:maximal_set_ratio} when Assumptions \ref{ass:lower_bound_maximum_inhomogeneity} and \ref{ass:lower_bound_sparsity} hold.

\paragraph{\normalfont\emph{Case 2} (Assumptions \ref{ass:lower_bound_sparsity} and \ref{ass:lower_bound_maximum_community_small} hold):}
For all $C \subseteq V$ of size $|C| = r$ we have $\scaling[C] \overline{p}_C \leq 1$, and therefore
\begin{equation}
\log(\scaling[C])
  \leq \log(1 / \overline{p}_C) \,.
\end{equation}
Hence, by Assumption~\ref{ass:lower_bound_maximum_community_small} \ref{ass:lbl:lower_bound_maximum_community_small_size} and \ref{ass:lbl:lower_bound_maximum_community_small_sparsity}, we obtain, for all $C \subseteq V$ of size $|C| = r$,
\begin{equation}
\frac{\log(r / |D|)}{\log(n / r)} \bigl(\log(\scaling[C]) \vee 1\bigr)
  \leq \frac{\log(r)}{\log(n / r)} \bigl(\log(1 / \overline{p}_C) \vee 1\bigr)
  = \smallO(1) \,.
\end{equation}
This shows that \eqref{eq:maximal_set_ratio} holds. Similarly, for all $C \subseteq V$ of size $|C| = r$, we obtain
\begin{equation}
\frac{\log(\scaling[C])}{\log(n / r)}
  \leq \frac{\log(r)}{\log(n / r)} \log(1 / \overline{p}_C)
  = \smallO(1) \,,
\end{equation}
which shows that $\log(n / r) / \log(\scaling[C]) \to \infty$.

This proves Lemma~\ref{lem:maximal_set_ratio} when Assumptions \ref{ass:lower_bound_maximum_community_small} and \ref{ass:lower_bound_sparsity} hold.
\end{proof}

\subsubsection{Proof of Lemma~\ref{lem:cutoff_definition}}

\begin{proof}[\unskip\nopunct]
Begin by defining $\widetilde{q}_{ij}$ by
\begin{equation}
\label{eq:q_definiton}
\frac{\widetilde{q}_{ij} \, p_{ij}}{1 - \widetilde{q}_{ij} \, p_{ij}}
  = \frac{\bigl(\scaling[C] p_{ij}\bigr)^2}{p_{ij}} \, \frac{(1-p_{ij})}{\bigl(1 - \scaling[C] p_{ij}\bigr)^2} \,,
\end{equation}
which implies that $\theta_{ij}(\widetilde{q}_{ij} \, p_{ij}) = 2\theta_{ij}(\scaling[C] \, p_{ij})$. By Assumption~\ref{ass:lower_bound_sparsity} we have $p_{ij} \to 0$ and $\scaling[C]^2 p_{ij} \to 0$ for every $i,j \in V$ and therefore it follows that $\widetilde{q}_{ij} \asymp \scaling[C]^2$.

We show below that $h(\widetilde{q}_{ij} - 1) \geq (2 + \smallO(1)) (\scaling[C] - 1)^2$ for all $i, j \in D$ when $n$ is large enough. Using this and the fact that $D \in \truncationset[C]$ gives
\begin{align}
(1 + \epsilon) \frac{1}{|D|} \, \E_0[e(D)] h(\widetilde{q}_{ij} - 1)
  &\geq (2 + \smallO(1)) (1 + \epsilon) (\scaling[C] - 1)^2 \frac{\E_0[e(D)]}{|D|}\\
  &\geq (2 + \smallO(1)) (1 + \epsilon) (1 - \epsilon/2) \left(\log\!\left(\frac{n \, |D|}{r^2}\right) - b_n\right)\\
  &\geq 2  (1 + \epsilon / 4) \left(\log\!\left(\frac{n \, |D|}{r^2}\right) - b_n\right)\\
  &\geq 2  (1 + \epsilon / 4) \log\!\left(\frac{n}{|D|} \frac{|D|^2}{r^2} \, \frac{1}{\log(n / r)}\right) \,.
\intertext{Then by Lemma~\ref{lem:maximal_set_ratio}, for every $D \in \truncationset[C]$, we have $|D| / r \geq (n / r)^{-\smallO(1)}$, and therefore}
(1 + \epsilon) \frac{1}{|D|} \, \E_0[e(D)] h(\widetilde{q}_{ij} - 1)
  &\geq 2 (1 + \epsilon / 4) \log\!\left(\frac{n}{|D|} \frac{|D|^2}{r^2} \, \frac{1}{\log(n / r)}\right)\\
  &\geq 2 \log\biggl(\frac{n}{|D|}\biggr) + 2 \log\biggl(\left(\frac{n}{|D|}\right)^{\!\epsilon/4}  \left(\frac{|D|^2}{r^2} \frac{1}{\log(n / r)}\right)^{\!1 + \epsilon/4\,}\biggr)\notag\\
  &\geq 2 \log\biggl(\frac{n}{|D|}\biggr) + \smash{\underbrace{2 \log\biggl(\left(\frac{n}{r}\right)^{\!\epsilon/4 - \smallO(1)(1 + \epsilon/4)}\biggr)}_{\to \, \infty}}\\
  &\geq 2 \log\biggl(\frac{n}{|D|}\biggr) \,.
\end{align}
Note that $h(x - 1)$ is continuous and increasing on $x \geq 1$. This means that, for large enough $n$, there is a unique solution $\cutoff[D] \in (1, \min_{i,j \in D} \widetilde{q}_{ij})$ such that
\begin{equation}
(1 + \epsilon) \frac{1}{|D|} \, \E_0[e(D)] h(\cutoff[D] - 1) = \log\!\left(\frac{n}{|D|}\right) \,.
\end{equation}
Moreover, it follows that $\theta_{ij}(\cutoff[D] \, p_{ij}) \leq \theta_{ij}(\widetilde{q}_{ij} \, p_{ij}) = 2 \theta_{ij}(\scaling[C] \, p_{ij})$ for every $i, j \in D$ because $\cutoff[D] \in (1, \min_{i,j \in D} \widetilde{q}_{ij})$.

We are left to show $h(\widetilde{q}_{ij} - 1) \geq (2 + \smallO(1)) (\scaling[C] - 1)^2$, which we do by considering different cases depending on the asymptotic behavior of $\scaling[C]$ (which is sufficient by Remark~\ref{rem:subsubsequence}).

\paragraph{\normalfont\emph{Case 1} ($\scaling[C] \to 1$):}
By definition of $\widetilde{q}_{ij}$,
\begin{equation}
\widetilde{q}_{ij} - 1
  = (\scaling[C] - 1)\left(1 + \frac{(1 - p_{ij}) \scaling[C]^2}{1 - p_{ij} (2 \scaling[C] - \scaling[C]^2)}\right)
  \asymp 2 (\scaling[C] - 1) \,.
\end{equation}
Then, using the above together with $h(x - 1) \asymp (x - 1)^2 / 2$ as $x \to 1$, we obtain
\begin{equation}
h\left(\widetilde{q}_{ij} - 1\right)
  \asymp \left(\widetilde{q}_{ij} - 1\right)^2 / 2
  \asymp 2 (\scaling[C] - 1)^2 \,.
\end{equation}

\paragraph{\normalfont\emph{Case 2} ($\scaling[C] \to \alpha \in (1, \infty)$):}
Using $\widetilde{q}_{ij} \asymp \scaling[C]^2$, we obtain
\begin{align}
\frac{h(\widetilde{q}_{ij} - 1)}{(\scaling[C] - 1)^2}
  &\asymp \frac{h(\scaling[C]^2 - 1)}{(\scaling[C] - 1)^2}
  \asymp \frac{\scaling[C]^2 \log(\scaling[C]^2) - \scaling[C]^2 + 1}{(\scaling[C] - 1)^2}\\
  &\asymp 1 + \frac{2 \scaling[C]\bigl( \scaling[C]\log(\scaling[C]) - \scaling[C] + 1\bigr)}{(\scaling[C] - 1)^2}
  \geq 2 + \smallO(1) \,.
\end{align}

\paragraph{\normalfont\emph{Case 3} ($\scaling[C] \to \infty$):} 
Using $\widetilde{q}_{ij} \asymp \scaling[C]^2$ and $h(x - 1) \asymp x \log(x)$ as $x \to \infty$, we obtain
\begin{equation}
\frac{h(\widetilde{q}_{ij} - 1)}{(\scaling[C] - 1)^2}
  = (1 + \smallO(1)) \frac{\widetilde{q}_{ij} \log(\widetilde{q}_{ij})}{\scaling[C]^2}
  \geq (2 + \smallO(1)) \log(\scaling[C])
  \to \infty \,.
\end{equation}
In particular, $h(\widetilde{q}_{ij} - 1) \geq 2 (\scaling[C] - 1)^2$ when $n$ is large enough.
\end{proof}

\subsubsection{Proof of Lemma~\ref{lem:cutoff_property}}
\begin{proof}[\unskip\nopunct]
First note that \eqref{eq:lower_bound_condition} implies
\begin{equation}
h(\scaling[C] - 1)\leq (1 - \epsilon) \frac{|D| \log(n / |D|)}{\E_0[e(D)]} \,,
\end{equation}
and Lemma~\ref{lem:cutoff_definition} implies
\begin{equation}
h(\cutoff[D] - 1)= \frac{1}{1 + \epsilon} \, \frac{|D| \log(n / |D|)}{\E_0[e(D)]} \,.
\end{equation}
Therefore,
\begin{equation}
\label{eq:h_ratio_cutoff_scaling}
\frac{h(\cutoff[D] - 1)}{h(\scaling[C] - 1)}\geq \frac{1}{1-\epsilon^2} \,.
\end{equation}

To prove the lemma we consider three different cases depending on the asymptotic behavior of $\scaling[C]$ (any other case is handled as in Remark~\ref{rem:subsubsequence}).

\paragraph{\normalfont\emph{Case 1} ($\scaling[C] \to 1$):} 
From the proof of Lemma~\ref{lem:cutoff_definition} we have $\cutoff[D] \in \bigl(1, \min_{i,j \in D} \widetilde{q}_{ij}\bigr)$, where $\widetilde{q}_{ij} \asymp \scaling[C]^2 \to 1$, and therefore $\cutoff[D] \to 1$. Then using $h(x - 1) \asymp (x - 1)^2 / 2$ as $x \to 1$ together with \eqref{eq:h_ratio_cutoff_scaling}, we obtain
\begin{align}
\frac{(\cutoff[D] - 1)^2}{(\scaling[C] - 1)^2}
  \asymp \frac{h(\cutoff[D] - 1)}{h(\scaling[C] - 1)}
  \geq \frac{1}{1-\epsilon^2} \,.
\end{align}
Using this, we obtain
\begin{align}
\scaling[C] h\!\left(\frac{\cutoff[D]}{\scaling[C]} - 1\right)
  &\asymp \, \frac{\left(\cutoff[D] - \scaling[C]\right)^2}{2 \scaling[C]}
  = \, \frac{1}{2} \left(\cutoff[D] - 1\right)^2 \left(1 - \frac{\scaling[C] - 1}{\cutoff[D] - 1}\right)^2\\
  &\geq (1 + \smallO(1)) \, h(\cutoff[D] - 1) (1-\sqrt{1-\epsilon^2})
  = \bigOmega(1) \, h(\cutoff[D] - 1) \,.
\end{align}
This result, together with Lemma~\ref{lem:cutoff_definition}, yields
\begin{equation}
\frac{1}{|D|} \, \E_C[e(D)] \, h\!\left(\frac{\cutoff[D]}{\scaling[C]} - 1\right)
  \geq \bigOmega(1) \, \frac{1}{|D|} \, \E_0[e(D)] \, h\left(\cutoff[D] - 1\right)\\
  \geq \bigOmega(1) \, \log\left(n / |D|\right)  \,.
\end{equation}
Finally, by Lemma~\ref{lem:maximal_set_ratio} it follows that $r / |D| \leq (n / r)^{\smallO(1)}$, and therefore
\begin{multline}
(1 - \epsilon) \frac{1}{|D|} \, \E_C[e(D)] \, h\!\left(\frac{\cutoff[D]}{\scaling[C]} - 1\right) - \log\!\left(\frac{r}{|D|}\right)\\
  \geq \bigOmega(1) \log\!\left(\frac{n}{|D|}\right) - \log\!\left(\frac{r}{|D|}\right)
  \geq \bigl(\bigOmega(1) - \smallO(1)\bigr) \log\!\left(\frac{n}{r}\right)
  \to \infty \,.
\end{multline}

\paragraph{\normalfont\emph{Case 2} ($\scaling[C] \to \alpha \in (1, \infty)$):} 
By \eqref{eq:h_ratio_cutoff_scaling} it clearly follows that $\scaling[C]\leq \cutoff[C]$. Also, $h(x-1)$ is convex and has derivative $\log(x)$. It follows that $h(x - 1) - h(\scaling[C] - 1) \leq (x - \scaling[C]) \log(x)$ for $x \geq \scaling[C]$. Using this,
\begin{equation}
\label{eq:hscaling_bound}
\log(\cutoff[D]) (\cutoff[D] - \scaling[C])
  \geq h(\scaling[C] - 1) \left(\frac{h(\cutoff[D] - 1)}{h(\scaling[C] - 1)} - 1\right)
  \geq h(\scaling[C] - 1) \left(\frac{1}{1 - \epsilon^2} - 1\right) \,.
\end{equation}
In particular, this result implies that $\cutoff[D]$ is lower bounded away from $\scaling[C]$ (i.e.\ $\cutoff[D] \geq \scaling[C] + \bigOmega(1)$). Now, using that $h(x)\geq \frac{x}{2}\log(x+1)$ we obtain
\begin{align}
\scaling[C] h\!\left(\frac{\cutoff[D]}{\scaling[C]} - 1\right)
  &\geq \frac{\scaling[C]}{2} \left(\frac{\cutoff[D]}{\scaling[C]} - 1\right)\log\!\left(\frac{\cutoff[D]}{\scaling[C]}\right)
  \geq \frac{\cutoff[D]-\scaling[C]}{2}\left(\log(\cutoff[D]) - \log(\scaling[C])\right)\\
  &\geq \bigOmega(1) \frac{\cutoff[D]-\scaling[C]}{2} \log(\cutoff[D])
  \geq \bigOmega(1) \, h(\scaling[C] - 1) \,,
\end{align}
where the last step follows from the fact that $\cutoff[D]$ is lower bounded away from $\scaling[C]$. To proceed similarly as in case 1, we need to relate $h(\cutoff[D] - 1)$ to $h(\scaling[C] - 1)$. From the proof of case 2 in Lemma~\ref{lem:cutoff_definition} it follows that $\cutoff[D] \leq \widetilde{q}_{ij} \asymp \scaling[C]^2$, and since $\scaling[C]$ is bounded away from $1$ it follows that $h(\scaling[C] - 1) / h(\cutoff[D] - 1) \geq \bigOmega(1)$. Therefore we conclude that
\begin{equation}
\scaling[C] h\!\left(\frac{\cutoff[D]}{\scaling[C]} - 1\right)
  \geq \bigOmega(1) \, h(\cutoff[D] - 1) \,.
\end{equation}
From this point onward the proof continues as in case 1.

\paragraph{\normalfont\emph{Case 3} ($\scaling[C] \to \infty$):} 
We have $\cutoff[D]\geq \scaling[C]\to\infty$ and $h(x - 1) \asymp x \log(x)$ as $x \to \infty$. Therefore it follows by \eqref{eq:h_ratio_cutoff_scaling} that
\begin{align}
\frac{1}{1 - \epsilon^2}
  \leq \frac{h(\cutoff[D] - 1)}{h(\scaling[C] - 1)}
  \asymp \frac{\cutoff[D] \log(\cutoff[D])}{\scaling[C] \log(\scaling[C])}
  \asymp \frac{\cutoff[D]}{\scaling[C]}\left(1 + \frac{\log(\cutoff[D] / \scaling[C])}{\log(\scaling[C])}\right) \,.
\end{align}
Hence, $\cutoff[D] / \scaling[C] \geq 1+\bigOmega(1)$. Further, using $\cutoff[D] \leq \widetilde{q}_{ij} \asymp \scaling[C]^2$, we obtain
\begin{align}
\frac{\scaling[C] h(\cutoff[D] / \scaling[C] - 1)}{h(\cutoff[D] - 1)}
  &\asymp \frac{\cutoff[D] \log(\cutoff[D] / \scaling[C]) - \cutoff[D] + \scaling[C]}{\cutoff[D] \log(\cutoff[D])}\\
  &= \frac{\log(\cutoff[D] / \scaling[C])}{\log(\cutoff[D])} + \smallO(1)
  \geq \bigOmega(1) \frac{1}{\log(\cutoff[D])}
  \geq \bigOmega(1) \frac{1}{\log(\scaling[C])} \,.
\end{align}
Here it was crucial to use the fact that $\cutoff[D] / \scaling[C]$ is lower bounded away from 1. Finally, by Lemma~\ref{lem:maximal_set_ratio} we obtain $\log(r / |D|) \leq \smallO(\log(n / r) / \log(\scaling[C]))$, and therefore we get
\begin{align}
\hspace{60pt}&\hspace{-60pt}
(1 - \epsilon) \frac{1}{|D|} \, \E_C[e(D)] \, h\!\left(\frac{\cutoff[D]}{\scaling[C]} - 1\right) - \log\!\left(\frac{r}{|D|}\right)\\
  &\geq \bigOmega(1) \frac{1}{|D|} \, \E_0[e(D)] \, \frac{h(\cutoff[D] - 1)}{\log(\scaling[C])} - \log\!\left(\frac{r}{|D|}\right)\\
  &\geq \bigOmega(1) \frac{\log(n / |D|)}{\log(\scaling[C])} - \log\!\left(\frac{r}{|D|}\right)\\
  &\geq \bigl(\bigOmega(1) - \smallO(1)\bigr) \frac{\log(n / r)}{\log(\scaling[C])}
  \to \infty \,,
\end{align}
where $\log(n / r) / \log(\scaling[C]) \to \infty$ follows from Lemma~\ref{lem:maximal_set_ratio}.
\end{proof}

\subsubsection{Proof of Lemma~\ref{lem:h_ratio}}
\begin{proof}[\unskip\nopunct]
Define the function
\begin{equation}
f_{p}(q)
  \coloneqq H_{p}(q) - p \, h\!\left(\frac{q}{p} - 1\right)\\
  = (q - p) + (1 - q) \log\!\left(\frac{1 - q}{1 - p}\right) \,.
\end{equation}
Then the derivatives of $f_p(q)$ are given by
\begin{equation}
\frac{\partial f_p(q)}{\partial q}
  = \log\!\left(\frac{1 - p}{1 - q}\right) \,,
\quad\quad\quad
\frac{\partial^2 f_p(q)}{\partial q^2}
  = \frac{1}{1 - q} \,,
\quad\quad\quad
\frac{\partial^3 f_p(q)}{\partial q^3}
  = \frac{1}{(1 - q)^2} \,.
\end{equation}
Therefore, for $0 < p < q$, a Taylor expansion of $q$ around $p$ shows that there exists $\xi \in [p, q]$ such that
\begin{equation}
f_p(q) = \frac{1}{2(1 - p)} (q - p)^2 + \frac{1}{6 (1 - \xi)^2} (q - p)^3 \,.
\end{equation}

Now, we continue by considering two cases depending of the value of $q / p$.

\paragraph{\normalfont\emph{Case 1} ($q / p \leq 5$):} 
Here we use that $h(x - 1) \geq (x - 1)^2 / 4$ for all $1 < x \leq 5$. Therefore, 
\begin{align}
\frac{f_p(q)}{p  h\bigl(\frac{q}{p} - 1\bigr)}
  &\leq \frac{4 p \, f_p(q)}{(q - p)^2}\\
  &= \frac{4 p}{(q - p)^2} \left[\frac{(q - p)^2}{2(1 - p)} + \frac{(q - p)^3}{6(1 - \xi)^2}\right]\\
  &= \frac{2p}{1 - p} + \frac{2}{3} \, \frac{q - p}{(1 - \xi)^2}\\
  &\leq \bigO(p) + \bigO(q) \,,
\end{align}
for some $\xi \in [p, q]$.

\paragraph{\normalfont\emph{Case 2} ($q / p > 5$):} 
Here we use that $h(x - 1) \geq (x - 1)$ for all $x > 5$. Therefore,
\begin{equation}
\frac{f_p(q)}{p  h\bigl(\frac{q}{p} - 1\bigr)}
  \leq \frac{f_p(q)}{q - p}
  = \frac{q - p}{2(1 - p)} + \frac{(q - p)^2}{6(1 - \xi)^2}
  \leq \bigO(p) + \bigO(q) \,,
\end{equation}
for some $\xi \in [p, q]$.

To complete the proof, note that $f_p(q) \geq 0$ and $p  h\bigl(\frac{q}{p} - 1\bigr) \geq 0$ for all $0 < p < q < 1$. Therefore, it follows that
\begin{equation*}
0 \leq \frac{f_p(q)}{p  h\bigl(\frac{q}{p} - 1\bigr)} \leq \bigO(p + q) \,. \tag*{\raisebox{-10pt}{$\qedhere$}}
\end{equation*}
\end{proof}

\subsection{Derivation of equation~(\ref{eq:edge_count_identity})}
\label{subsec:derivation_edge_count_identity}
The choice of estimator in Section~\ref{subsec:scan_test_for_unknown_rank1_edge_probabilities} is based on the equality from \eqref{eq:edge_count_identity}. In this section we give a more detailed derivation of this equality. First, observe that $\sum_{i \notin D} \w_i \geq \sum_{i \in D} \w_i$, which is ensured by Assumption~\ref{ass:unknown_probability_maximum_inhomogeneity}. To see this, note that $r \frac{\w_\text{max}}{\w_\text{min}} \leq r^{4/3} \wedge \sqrt{n \, r} \leq n^{4/5}$. Hence, for $n$ large enough,
\begin{equation}
\sum_{i \notin D} \w_i - \sum_{i \in D} \w_i
  = \w_\text{min} \left((n - r) - r \frac{\w_\text{max}}{\w_\text{min}}\right)
  \geq \w_\text{min} \bigl((1 + \smallO(1))n - n^{4/5}\bigr)
  > 0 \,.
\end{equation}
Then, using that $\sum_{i \notin D} \w_i \geq \sum_{i \in D} \w_i$, we obtain
\begin{align}
\label{eq:edge_count_identity_initial}
2 \, {\textstyle \sum_{i \in D} \w_i}
  &= \textstyle \sum_{i \in V} \w_i - \Bigl(\sum_{i \notin D} \w_i - \sum_{i \in D} \w_i\Bigr)\\
  &= \sqrt{\Bigl({\textstyle \sum_{i \in V} \w_i}\Bigr)^{\!2}} - \sqrt{\Bigl({\textstyle \sum_{i \notin D} \w_i} - {\textstyle \sum_{i \in D} \w_i}\Bigr)^{\!2}}\\
  &= \sqrt{\Bigl({\textstyle \sum_{i \in V} \w_i}\Bigr)^{\!2}} - \sqrt{\Bigl({\textstyle \sum_{i \notin D} \w_i} + {\textstyle \sum_{i \in D} \w_i}\Bigr)^{\!2} - 4 \, {\textstyle \sum_{i \in D}\sum_{j \notin D} \w_i \w_j}}\\
  &= \sqrt{2 \E_0[e(V)] + {\textstyle \sum_{i \in V} \w_i^2}} - \sqrt{2 \E_0[e(V)] + {\textstyle \sum_{i \in V} \w_i^2} - 4 \E_0[e(D,-D)]} \,,
\end{align}
Finally, plugging \eqref{eq:edge_count_identity_initial} into the definition of $\E_0[e(D)]$, we obtain
\begin{multline}
\E_0[e(D)]
  = \frac{1}{2} \Bigl({\textstyle \sum_{i \in D} \w_i}\Bigr)^{\!2} - \frac{1}{2} {\textstyle \sum_{i \in D} \w_i^2}
  = \frac{1}{8} \Bigl(2 {\textstyle \sum_{i \in D} \w_i}\Bigr)^{\!2} - \frac{1}{2} {\textstyle \sum_{i \in D} \w_i^2}\\
\begin{aligned}[t]
  &= \frac{\!\biggl(\raisebox{-1pt}{$\!\sqrt{2 \E_0[e(V)] + \rule{0pt}{10pt}\smash{\raisebox{2.2pt}{\scalebox{0.8}{$\displaystyle\sum_{i \in V}$}}} \,\w_i^2} - \sqrt{2 \E_0[e(V)] + \rule{0pt}{10pt}\smash{\raisebox{2.2pt}{\scalebox{0.8}{$\displaystyle\sum_{i \in V}$}}} \,\w_i^2 - 4 \E_0[e(D,-D)]}$}\biggr)^{\!2}\!}{8} - \frac{1}{2}\raisebox{2pt}{\scalebox{0.9}{$\displaystyle\sum_{i \in D}$}} \,\w_i^2 \notag\\
  &= \frac{\!\biggl(\raisebox{-1pt}{$\!\sqrt{\E_0[e(V)] + \frac{1}{2} \rule{0pt}{10pt}\smash{\raisebox{2.2pt}{\scalebox{0.8}{$\displaystyle\sum_{i \in V}$}}} \,\w_i^2} - \sqrt{\E_0[e(V)] + \frac{1}{2} \rule{0pt}{10pt}\smash{\raisebox{2.2pt}{\scalebox{0.8}{$\displaystyle\sum_{i \in V}$}}} \,\w_i^2 - 2 \E_0[e(D,-D)]}$}\biggr)^{\!2}\!}{4} - \frac{1}{2} \raisebox{2pt}{\scalebox{0.9}{$\displaystyle\sum_{i \in D}$}} \,\w_i^2 \,. \notag
\end{aligned}
\end{multline}

\paragraph{Acknowledgements.} 
The work of RvdH was supported in part by the Netherlands Organisation for Scientific Research (NWO) through the Gravitation \textsc{Networks} grant 024.002.003.

{\setlength{\emergencystretch}{3em}\printbibliography[title={References}]}
\end{document}